\documentclass[a4paper,12pt,leqno, twoside]{article}
\usepackage{amsthm,amssymb,amsmath}
\usepackage{hyperref}
\usepackage[french]{babel}
\usepackage[arrow,matrix]{xy} 
\usepackage[latin1]{inputenc}

\def\NM{{\mathbb{N}}}

\def\HM{{\mathbb{H}}}

\def\PM{{\mathbb{P}}}

\def\QM{{\mathbb{Q}}}
\def\FM{{\mathbb{F}}}
\def\ZM{{\mathbb{Z}}}

\def\MG{{\mathfrak M}}

\def\mG{{\mathfrak m}}

\def\ZG{{\mathfrak Z}}

\def\AC{{\mathcal A}}

\def\CC{{\mathcal C}}
\def\HC{{\mathcal H}}

\def\RC{{\mathcal R}}

\def\OC{{\mathcal O}}
\def\MC{{\mathcal M}}
\def\KC{{\mathcal K}}
\def\PC{{\mathcal P}}

\def\FC{{\mathcal F}}

\def\DC{{\mathcal D}}

\def\ssi{si et seulement si}

\def\simto{\buildrel\hbox{\tiny{$\sim$}}\over\longrightarrow}

\def\leq{\leqslant}
\def\geq{\geqslant}
\def\injo{\hookrightarrow}
\def\id{\mathop{\mathrm{Id}}\nolimits}

\def\ba{\backslash}
\def\wt{\widetilde}
\def\wh{\widehat}
\def\o#1{\overline{#1}}
\renewcommand{\o}{\overline}

\def\application#1#2#3#4#5{\begin{array}{rcl}
                            #1 \;\;\; #2 & \to &  #3 \\
                              #4 & \mapsto & #5 
                            \end{array}} 
\def\cas#1#2#3#4#5{\begin{array}{rcl} #1 \; & = &
    \left\{\begin{array}{rcl} #2 & \hbox{ si } & #3 \\
                             #4 & \hbox{ si } & #5 \end{array}
                         \right. \end{array}}

\def\To#1{\buildrel\hbox{\tiny{$#1$}}\over\longrightarrow}
\def\to{\rightarrow}

\def\ker{\mathop{\hbox{\rm ker}\,}}

\def\endo#1#2{\hbox{\rm End}_{#1}\left(#2\right)}

\def\Spec{\hbox{\rm Spec}}

\def\Hom{\mathop{\hbox{\rm Hom}}\nolimits}
\def\Ext{\mathop{\hbox{\rm Ext}}\nolimits}
\def\Rhom{\mathop{R\hbox{\rm Hom}}\nolimits}

\def\opp{{\rm opp}}
\def\Mod{{\rm Mod}}
\def\Rep{{\rm {R}ep}}
\def\Modtop#1#2{\mathop{{\rm Mod}_{#1}(#2)}}%cat des modules d'un topos
\def\Mo#1#2{\mathop{\Rep^{\infty}_{#1}(#2)}}%cat des RG-modules lisses
\def\Moc#1#2{\mathop{\Rep^{c}_{#1}(#2)}}
\def\Mosc#1#2{\mathop{\Rep^{\rm sc}_{#1}(#2)}}
\def\Irr#1#2{\mathop{{\rm Irr}_{#1}\left(#2\right)}}%RG-modules irreductibles
\def\ind#1#2#3{\hbox {\rm Ind}_{#1}^{#2}\>\!\left(#3\right)}  %induction

\def\cind#1#2#3{\hbox {\rm ind}_{#1}^{#2}\>\!\left(#3\right)} %ind a supports compacts
\def\ip#1#2#3{\hbox {\sl i}_{#1}^{#2}\>\!(#3)}  %induction parabolique
\def\dim{\mathop{\mbox{\rm dim}}\nolimits}

\def\val{\mathop{\mbox{\sl val}}\nolimits}

\def\gal{\mathop{\mbox{\sl Gal}}\nolimits}
\def\det{\mathop{\mbox{\rm det}}\nolimits}
\def\nrd{\mathop{\mbox{\rm nrd}}\nolimits}

\def\tor#1#2#3#4{\mbox{\sl Tor}^{#4}_{#1}\>\!(#2,#3)}

\def\limproj{{\lim\limits_{\longleftarrow}}}
\def\limind{{\lim\limits_{\longrightarrow}}}

\makeatletter
\renewcommand{\subsubsection}{\@startsection{subsubsection}{3}{\parindent}{-\baselineskip}{-0.01\baselineskip}{\bf}}
\renewcommand*{\@seccntformat}[1]{%
  \csname the#1\endcsname\
}
\makeatother  

\def\ali{\subsubsection{}\setcounter{equation}{0}}
\def\alin#1{\setcounter{equation}{0}\subsubsection{\it  #1}. }

\newtheoremstyle{th}
  {\baselineskip}{.5\baselineskip}{\itshape}
  {\parindent}{\bf}
  {--}{.5em}{}

\newtheoremstyle{def}
  {\baselineskip}{\baselineskip}{}
  {\parindent}{\bf}
  {--}{.5em}{}

\newtheoremstyle{th*}
  {.5\baselineskip}{.5\baselineskip}{\itshape}
  {\parindent}{\bf}
  {--}{.5em}{}

\newtheoremstyle{remark*}
  {.5\baselineskip}{.5\baselineskip}{}
  {\parindent}{\bf}
  {--}{.5em}{}

\newtheoremstyle{remark}
  {.5\baselineskip}{.5\baselineskip}{}
  {\parindent}{\bf}
  {--}{.5em}{}

\swapnumbers
\theoremstyle{th}

\newtheorem{lemme}[subsubsection]{\sc Lemme.\bf}
\newtheorem{prop}[subsubsection]{\sc Proposition.\bf}
\newtheorem{coro}[subsubsection]{\sc Corollaire.\bf}

\theoremstyle{def}
\newtheorem{fact}[subsubsection]{\sc Fait\bf}

\newtheorem{DEf}[subsubsection]{\sc D{\'e}finition.\bf}

\theoremstyle{remark}
\newtheorem{rema}[subsubsection]{\sc Remarque.\bf}  %\renewcommand{\therema}{}

\theoremstyle{th*}
\newtheorem*{thm}{\sc Th{\'e}or{\`e}me.}
\newtheorem*{thm1}{\sc Th{\'e}or{\`e}me 1.}
\newtheorem*{thm2}{\sc Th{\'e}or{\`e}me 2.}
\newtheorem*{thm3}{\sc Th{\'e}or{\`e}me 3.}

\newtheorem*{thm4}{\sc Th{\'e}or{\`e}me 4.}

\newtheorem*{lem}{\sc Lemme.}
\newtheorem*{pro}{\sc Proposition.}
\newtheorem*{cor}{\sc Corollaire.}

\newtheorem*{defn}{\sc D\'efinition.}
\newtheorem*{fac}{\sc Fait.}

\theoremstyle{remark*}
\newtheorem*{rem}{\sc Remarque.}

\newtheorem*{exe}{\sc Exemple.}

\newcommand{\findem}{\hfill$\Box$\par\medskip}
\newcommand{\dem}{\indent {\it Preuve :} \rm }

\title{Théorie de Lubin-Tate non abélienne $\ell$-entière}
\author{J.-F. Dat\footnote{L'auteur remercie l'Institute for Advanced Study pour son
    hospitalité et la National Science Foundation  pour son soutien financier
    No. DMS-0635607. Les idées et conclusions exprimées dans ce
    texte sont  celles de l'auteur et n'engagent pas la NSF.}}
\date{}

\begin{document}
\maketitle
\bibliographystyle{plain}
\renewcommand{\proofname}{\indent Preuve}

\def\la{\langle}
\def\ra{\rangle}
\def\knr{{\wh{K^{nr}}}}
\def\ka{\wh{K^{ca}}}

\abstract{For two primes $\ell\neq p$, we investigate the
$\ZM_{\ell}$-cohomology of the Lubin-Tate towers of a $p$-adic field.
We prove  that it realizes some version of Langlands and Jacquet-Langlands
correspondences for flat families of irreducible supercuspidal representations
parametrized by a $\ZM_{\ell}$-algebra $R$, in a way compatible with extension of scalars.
Applied to $R=\o\FM_{\ell}$, this gives a cohomological realization of
the Langlands-Vigneras correspondence for supercuspidals, and a new proof of its
existence.
Applied to complete local algebras, this provides
bijections between deformations of matching $\o\FM_{\ell}$-representations.
Besides, we also get a
 virtual realization of both the \emph{semi-simple} Langlands-Vigneras correspondence
 and the $\ell$-modular Langlands-Jacquet transfer for \emph{all}
 representations, by using the cohomology
 complex and working in a suitable Grothendieck group.}

%\maketitle

\def\dd{D_d^\times}
\def\mdro{\MC_{{\rm Dr},0}}
\def\mdrn{\MC_{{\rm Dr},n}}
\def\mdr{\MC_{\rm Dr}}
\def\mlto{\MC_{{\rm LT},0}}
\def\mltn{\MC_{{\rm LT},n}}
\def\mltno{\MC_{{\rm LT},n}^{(0)}}
\def\mlt{\MC_{\rm LT}}
\def\plt{\PC_{\rm LT}}
\def\pdr{\PC_{\rm Dr}}
\def\mltK{\MC_{LT,K}}
\def\LJ{{\rm LJ}}
\def\JL{{\rm JL}}
\def\SL{{\rm SL}}
\def\GL{{\rm GL}}

\def\scusp{{\rm sc}}

\def\Ql{\QM_{\ell}}
\def\Zl{\ZM_{\ell}}
\def\Zlnr{\ZM_{\ell}^{\rm nr}}
\def\Fl{\FM_{\ell}}
\def\oQl{\o\QM_{\ell}}
\def\oZl{\o\ZM_{\ell}}
\def\bZl{\breve\ZM_{\ell}}
\def\bQl{\breve\QM_{\ell}}
\def\oFl{\o\FM_{\ell}}
\def\mltnc{\wh\MC_{{\tiny{\rm LT}},n}}
\def\mltnco{\wh\MC_{{\tiny{\rm LT}},n}^{(0)}}
\def\Wi{{\rm W}}

\tableofcontents

\section{Principaux résultats}

Soit $K$ un corps local non-archimédien d'anneau d'entiers $\OC$ et 
de corps résiduel $k\simeq
\FM_{q}$, où $q$ est une puissance d'un nombre premier $p$. 
Soit $k^{\rm ca}$ une clôture algébrique de $k$ et
$\wh{K^{\rm nr}}$ l'extension non ramifiée maximale complétée de $K$,
d'anneau des entiers $\wh{\OC^{\rm nr}}$ et de corps résiduel $k^{\rm
  ca}$. 

Fixons un entier $d\in\NM$. Nous notons $\mltnc$ le $n$-ème étage de
la tour de Lubin-Tate de hauteur $d$ du corps $K$. C'est donc le
$\wh{\OC^{\rm nr}}$-schéma formel qui classifie les déformations
\emph{par quasi-isogénies} et \emph{munies de structures de niveau $n$ à la Drinfeld} de l'unique 
$\OC$-module formel $\HM_{d}$ de dimension $1$ et
hauteur $d$ sur $k^{\rm ca}$.
Par définition, le groupe de quasi-isogénies  de $\HM_{d}$ agit sur 
$\mltnc$, et on sait qu'il s'identifie au groupe $D^{\times}$ des
unités de l'algèbre à division de centre $K$ et d'invariant $1/d$. Par
un jeu subtil sur les structures de niveau, 
le groupe linéaire
$G:=\GL_{d}(K)$ agit  sur \emph{la tour}
 des  $(\mltnc)_{n\in\NM}$, et cette action
 commute à celle de $D^{\times}$. De plus, l'action du produit
 $G\times D^{\times}$ se factorise par le quotient $GD:=(G\times
 D^{\times})/K^{\times}_{\rm diag}$.
 Nous noterons $\mltn$ la fibre générique de $\mltnc$. C'est un $\knr$-espace
 analytique lisse de dimension $d-1$. Il est muni d'une donnée de
 descente à $K$ qui permet de prolonger à $W_{K}$ l'action de l'inertie $I_{K}$
 sur $H^{*}_{c}(\mlt^{\rm ca},\Lambda):=\limind_{n} H^{*}_{c}(\mltn^{\rm ca},\Lambda)$.

Soit maintenant  $\ell\neq p$ un nombre premier. Par commodité, nous
choisissons une racine carrée de $q$ dans 
$\oZl$, ce qui nous permet de définir les torsions à la Tate
demi-entières, et de normaliser la correspondance de Langlands
$\ell$-adique.
Notre but est de
prouver des versions mod $\ell$ et $\ell$-entières des théorèmes locaux principaux de
Harris et Taylor dans \cite{HaTay} sur la cohomologie $\ell$-adique
de $\mlt$. Nous apportons donc plusieurs réponses au Problem 6 posé
par Harris dans \cite{HarrisICM}.

\subsection{Modulo $\ell$}

On rappelle que, suivant Vignéras, une $\oFl$-représentation irréductible est dite
\emph{supercuspidale} si elle n'est sous-quotient d'aucune induite
parabolique propre admissible. En général, cette propriété est plus
forte que celle d'être \emph{cuspidale}, \emph{i.e.} annulée par tous
les foncteurs de Jacquet propres.

\begin{thm1}
  Pour $\pi$ une $\oFl$-représentation irréductible supercuspidale de $G$, la
  $\oFl$-représen\-tation $ \Hom_{\Zl G}\left(H^{d-1}_{c}(\mlt^{\rm
  ca},\Zl),\pi\right)$ de $D^{\times}\times W_{K}$ est
irréductible. Si on l'écrit sous
 la forme
$$ \Hom_{\Zl G}\left(H^{d-1}_{c}(\mlt^{\rm
  ca},\Zl),\pi\right)\mathop\simeq\limits_{D^{\times}\times W_{K}}
\rho(\pi)\otimes\sigma(\pi)(\frac{d-1}{2}),$$
alors on a les propriétés suivantes : 
\begin{enumerate}
\item   $\pi\mapsto \sigma(\pi)$ est une bijection entre
 classes de $\oFl$-représentations supercuspidales de $G$ et
 classes de  $\oFl$-représentations irréductibles de dimension $d$ de $W_{K}$.
\item  $\pi\mapsto \rho(\pi)$ est une injection de l'ensemble des
 classes de  $\oFl$-représentations supercuspidales de $G$ dans celui
 des classes de  $\oFl$-représentations irréductibles de $D^{\times}$.
\item Pour tout $\oQl$-relèvement $\wt\pi$ de $\pi$, on a
  $\sigma(\pi)=r_{\ell}(\sigma(\wt\pi))$ et $\rho(\pi)=r_{\ell}(\rho(\wt\pi))$.
\end{enumerate}
\end{thm1}
Dans le dernier point, la notation $\rho(\wt\pi)$ désigne 
la correspondante de Jacquet-Langlands de $\wt\pi$,
  et $\sigma(\wt\pi)$ désigne sa correspondante de Langlands.
La notation $r_{\ell}$ désigne
  l'application de ``décomposition'' qui consiste à prendre un réseau
  stable, le réduire ``modulo $\ell$'' et semi-simplifier.

Quelques remarques s'imposent :
\begin{itemize}
\item l'existence d'une bijection satisfaisant les points i) et iii)
  (et donc nécessairement unique) n'est pas nouvelle et est dûe à
  Vignéras dans \cite{VigLanglands}. Les différences entre sa
  preuve et la notre sont dans le point iii), qu'elle obtient par un
  argument global de congruences entre formes automorphes sur des
  groupes anisotropes, et dans la preuve de la surjectivité. Les preuves
  de l'irréductibilité et de l'injectivité reposent toujours sur son
  ``critère numérique'' \cite[2.3]{VigAENS}.

\item De même, l'existence d'une injection satisfaisant les points ii) et iii)
  est déjà établie dans \cite{jlmodl},
  et la différence réside dans le point iii) qui y est obtenu à l'aide
  de ``caractères de Brauer''.

\item L'énoncé obtenu lorsque l'on remplace $\oFl$ par $\oQl$ (moins
  le point iii) évidemment) découle de \cite[Thm B]{HaTay}, une fois que l'on sait que
 la partie supercuspidale de la cohomologie est concentrée
en degré $d-1$, ce qui est prouvé par Mieda dans \cite{Mieda},
\emph{cf} aussi Strauch \cite{StrauchAdv}. Nous utilisons bien
évidemment tous ces résultats.
\end{itemize}

Bien qu'il n'apparaisse pas dans cet énoncé, un rôle crucial est joué
par le complexe 
%de cohomologie étale de la tour de Lubin-Tate à coefficients dans  $\Zl$,
  $$R\Gamma_{c}(\mlt^{\rm ca},\Zl) \in
\DC^{b}(\Rep^{\infty,c}_{\Zl}(GD\times W_{K}))$$
défini dans la section 3 de \cite{lt}.
% Ici, la notation $^{\rm ca}$ désigne un changement de
% base à une \emph{c}lôture \emph{a}lgébrique complétée $\wh{K^{\rm
%     ca}}$ de $\knr$,  et
Ici, la catégorie dérivée est celle de la catégorie abélienne 
 des $\Zl$-représentations
 de $GD\times W_{K}$ qui sont \emph{lisses} (signe $\infty$)  pour
l'action de $GD$, et qui sont \emph{continues} (signe $c$) pour
celle de $W_{K}$ au sens  
%En fait, la construction de {\em loc. cit} fournit
%un peu mieux : on peut imposer que l'action de $W_{K}$ soit lisse sur
où l'action de $I_{K}$ provient d'une structure de module sur
l'algèbre complétée $\Zl[[I_{K}]]$ et est lisse sur
le plus grand sous-groupe fermé $I_{K}^{\ell'}$ d'ordre premier à
$\ell$. 
La propriété essentielle sur laquelle repose cet article est prouvée
dans l'appendice \ref{appperfG} où l'on trouvera un énoncé plus précis :
\begin{pro}
  Le complexe $R\Gamma_{c}$ est isomorphe, dans $\DC^{b}(\Mo{\Zl}{G})$, 
à un complexe \emph{borné} de $\Zl G$-représentations lisses projectives
  et ``localement de type fini''.
\end{pro}
La stratégie générale pour prouver des énoncés de ce type remonte au
moins à SGA4, Exp XVII (5.2.10), mais la preuve n'est pas formelle
pour autant : on utilise par exemple le morphisme de périodes de
Gross-Hopkins et le fait que la catégorie $\Mo{\Zl}{G}$ des
$\Zl$-représentations lisses de $G$ est
localement noethérienne.

Cette propriété de finitude montre que 
pour une $\oFl$-représentation  lisse irréductible $\pi$ de $G$,  le complexe
$$
R_{\pi}:=\Rhom_{\Zl G}\left(R\Gamma_{c}(\mlt^{\rm ca},\Zl),\pi\right)(\frac{1-d}{2})[1-d] \in
\DC^{+}(\Rep^{\infty,c}_{\oFl}(D^{\times} \times W_{K})) 
$$
est à cohomologie bornée et de dimension finie. Ceci permet de
considérer son image 
$$[R_{\pi}] \in \RC(D^{\times}\times W_{K},\oFl)$$
dans le groupe de  Grothendieck des représentations de longueur finie
de $D^{\times}\times W_{K}$. Nous obtiendrons alors l'amplification
suivante du théorème 1 :

\begin{thm2} \label{thm1}
  Pour $\pi\in \Irr{\oFl}{G}$, la $\oFl$-représentation virtuelle
  $[R_{\pi}]$
est de la forme
$$ [R_{\pi}]=  \LJ(\pi) \otimes \sigma(\pi)^{\rm ss}, \,\hbox{ où} $$
\begin{enumerate}
\item $\sigma(\pi)^{\rm ss}$ est une $\oFl$-représentation
 semi-simple de dimension $d$ de $W_{K}$  vérifiant les propriétés suivantes : 
 \begin{enumerate}
\item $\sigma(\pi)^{\rm ss}=\sigma(\pi_{1})^{ \rm
    ss}+\sigma(\pi_{2})^{\rm ss}$ si $\pi$ est sous-quotient de
  l'induite parabolique $\pi_{1}\times\pi_{2}$.
 \item $\sigma(\pi)^{\rm ss}= r_{\ell}(\sigma(\wt\pi))$ pour toute
   $\oQl$-représentation irréductible entière $\wt\pi$ dont la
   réduc\-tion modulo $\ell$ contient $\pi$.
 \end{enumerate}
\item $\LJ(\pi)$ est une $\oFl$-représentation virtuelle de
  $D^{\times}$ et l'homomorphisme 
$$\LJ_{\oFl}:\,
  \RC(G,\oFl)\To{}\RC(D^{\times},\oFl)$$ obtenu par linéarité vérifie
  $\LJ_{\oFl}\circ r_{\ell}=r_{\ell}\circ \LJ_{\oQl}$, où
  $\LJ_{\oQl}$ est le transfert de Langlands-Jacquet $\ell$-adique
$\RC(G,\oQl)\To{}\RC(D^{\times},\oQl)$ de \cite[Cor. 2.1.5]{lt}.
% , et $r_{\ell}$
% désigne les applications de décomposition modulo $\ell$.
\end{enumerate}
\end{thm2}

Dans le texte, nous prouverons d'abord le théorème 2 et montrerons que
l'application $\pi\mapsto \sigma(\pi)^{\rm ss}$ vérifie la propriété
i) du théorème 1. Puis nous déduirons le théorème 1 en remarquant que
$[R_{\pi}]=[\Hom_{\Zl G}(H^{d-1}_{c},\pi)(\frac{d-1}{2})]$ lorsque
$\pi$ est \emph{supercuspidale}. Il est possible de prouver tout le
théorème 1 moins la surjectivité du point i), sans étudier
$[R_{\pi}]$, \emph{cf} remarque \ref{Rempreuvealter}.

On peut faire les mêmes remarques que précédemment :
\begin{itemize}
\item Le i) fournit une réalisation cohomologique de la correspondance de Langlands
  semi-simple établie par Vignéras dans \cite{VigLanglands} et donne
  une nouvelle\footnote{Noter que l'on utilise
  tout de même les résultats fins de classification des
  $\oFl$-représentations de $G$ dûs à Vignéras} preuve de son existence.
\item Le ii) fournit une réalisation cohomologique du transfert de
  Langlands-Jacquet $\ell$-modulaire de \cite{jlmodl} et donne une
  nouvelle preuve de son existence.
\item Lorsque l'on remplace $\oFl$ par $\oQl$, la formule
$[R_{\pi}]=\LJ(\pi)\otimes\sigma(\pi)$ découle des faits suivants :
\begin{itemize}
\item[a)] On a l'égalité
  $[R_{\pi}]=\sum_{i,j}(-1)^{i+j}[\Ext^{j}_{G}(H^{i}_{c}(\mlt^{\rm
    ca},\oQl),\pi)]$. 
\item[b)]  On connait déjà
$\sum_{i}  (-1)^{i}[H^{i}_{c}(\mlt^{\rm ca},\oQl)]$
%$\in \RC(G\times  D^{\times}\times W_{K},\oQl)$ 
grâce à \cite[Thm VII.1.5]{HaTay}.
\item[c)] On conclut 
grâce au calcul
  d'extensions de \cite[2.1.16]{lt}.
\end{itemize}
On notera que pour $\pi$ sur $\oFl$, l'égalité a) n'a parfois aucun
sens car la somme de droite est infinie. C'est par exemple le cas
 pour $\pi$ la représentation triviale et $q\equiv 1 [\ell]$.
\end{itemize}
% $\sum_{i,j}(-1)^{i+j}[\Ext^{i}_{G}(H^{j}_{c}(\mlt^{\rm ca},\oFl),\pi)]$. Lorsque 
% cette somme est finie, on voit facilement qu'elle coïncide avec
% $[R_{\pi}]$. Mais cette somme n'est pas toujours finie, par exemple
% pour $\pi$ la représentation triviale et $q\equiv 1 [\ell]$.

Par ailleurs, lorsque $\pi$ est sur $\oQl$, nous avons décrit beaucoup plus
précisément le complexe $R_{\pi}$ dans \cite[Thm. A]{lt}, en utilisant
les résultats de Boyer  \cite{Boyer2}. Selon cette
description, $R_{\pi}$ n'a de la cohomologie qu'en degrés \emph{de
  même parité}, et il n'y a donc aucune annulation lorsqu'on prend la
somme alternée. L'exemple suivant, détaillé en \ref{exemple}, montre que \emph{ceci n'est
plus toujours vrai sur $\oFl$}.

\begin{exe}
  Supposons $d=3$ et $q\equiv 1 [\ell]$. Soit $\pi$ le quotient de
  $\CC^{\infty}(\PM^{2}(K),\oFl)$ par les fonctions constantes. Alors 
$\HC^{-1}(R_{\pi})$ et $ \HC^{1}(R_{\pi})$ sont de dimension au moins
$2$.
\end{exe}

\subsection{En familles, et en déformations}

Dans le théorème \ref{theocoho} du texte, nous donnons une description explicite
de la ``partie supercuspidale'' de la cohomologie entière de
$\mlt$. Nous ne répétons pas cet énoncé ici, qui demande beaucoup de
notations, mais donnons plutôt
quelques corollaires frappants.

Soit $\Zlnr$ l'extension non ramifiée maximale de $\Zl$
dans $\oZl$.  Pour $R$ 
une $\Zlnr$-algèbre noethérienne, on appellera \emph{$R$-famille de représentations
  irréductibles} de $G$ toute $R$-repré\-sen\-tation lisse $(\Pi,V)$ vérifiant
\begin{itemize}\item $V^{H}$ est localement libre de rang fini pour
  tout pro-$p$-sous-groupe ouvert de $G$
\item $\Pi\otimes_{R}C$ est irréductible pour tout  corps
  algébriquement clos $C$ au-dessus de $R$.
\end{itemize}

\begin{thm3}
  Il existe deux foncteurs exacts 
$$ \application{}{\Rep^{\infty}_{\Zlnr}(G)}{\Mo{\Zlnr}{D^{\times}}}{\Pi}{\rho(\Pi)}
\hbox{ et }
\application{}{\Rep^{\infty}_{\Zlnr}(G)}{\Mo{\Zlnr}{W_{K}}}{\Pi}{\sigma'(\Pi)}$$
tels que pour toute $\Zlnr$-algèbre noethérienne $R$, on a :
\begin{enumerate}
\item $\Pi\mapsto \sigma'(\Pi)$ induit une bijection entre
 classes de $R$-familles de représentations irréduc\-tibles supercuspidales de $G$ et
 classes de  $R$-familles de représentations irréductibles de dimension $d$ de $W_{K}$.
\item  $\Pi\mapsto \rho(\Pi)$ induit 
une injection de l'ensemble des
  classes de  $R$-familles de représen\-tations irréductibles supercuspidales de $G$ dans celui
 des classes de  $R$-familles de représen\-tations irréductibles de $D^{\times}$.
\item Si $\Pi$ est une $R$-famille de représentations irréductibles
  supercuspidales de $G$, alors il existe un $R$-module inversible
  $\RC(\Pi)$ et un isomorphisme
$$ \Hom_{\Zl G}\left(H^{d-1}_{c}(\mlt^{\rm
    ca},\Zl),\Pi\right)\otimes_{R}\RC(\Pi)
\simeq \rho(\Pi)\otimes_{R}\sigma'(\Pi) $$
\end{enumerate}
\end{thm3}

En fait, les foncteurs de l'énoncé induisent des équivalences de
catégories entre certaines sous-catégories ``naturelles'' de
représentations des trois groupes ; on renvoie au paragraphe \ref{defsigmapi} pour l'énoncé précis.
Les points i) et ii)  fournissent des
correspondances de Langlands et Jacquet-Langlands pour les
$R$-familles de représentations irréductibles. En vertu de
l'exactitude des foncteurs annoncés, ces correspondances
sont \emph{compatibles à tout changement d'anneau $R$.}
De même, le module inversible et
l'isomorphisme du point iii) peuvent être choisis fonctoriels en la paire $(R,\Pi)$.
En parti\-culier, ce théorème implique le théorème 1. Néanmoins, nous ne
le prouverons qu'après avoir établi le théorème 1. 

Remarquons qu'on ne peut pas ``descendre'' ce théorème à $\Zl$, mais on peut 
``descendre'' une version où l'on impose une condition de type caractère central/déterminant
 fixé, \emph{cf} \ref{resultdescente}.

% On pourrait aussi utiliser le langage des champs ; par exemple le
% foncteur $\rho$ induit une équivalence du ``champ des familles de
% représentations irréductibles supercuspidales de $G$ au-dessus de $\Spec(\Zlnr)$'' sur
% le ``champ des familles de représentations irréductibles de dimension $d$
% de $W_{K}$ au-dessus de $\Spec(\Zlnr)$'' (ce sont des champs fpqc).
% Si l'on impose une condition de type caractère central/déterminant
% fixé, on peut descendre cette équivalence à $\Spec(\Zl)$, \emph{cf} ...

Comme corollaire, on obtient 
des correspondances entre \emph{déformations} de représenta\-tions de $G$,
$D^{\times}$ et $W_{K}$, ou si l'on préfère, une déformation
des correspondances de Langlands et Jacquet-Langlands du théorème 1.  
%Notons  $\bZl:=W(\oFl)$ l'anneau des
%vecteurs de Witt à coefficients dans $\oFl$, et 
Notons $\AC_{\oFl}$ la
catégorie des $\Zlnr$-algèbres locales complètes noethériennes de
corps résiduel $\oFl$.
Un \emph{relèvement} d'une $\oFl$-représentation $\pi$ de $G$ à
$\Lambda\in\AC_{\oFl}$ est une $\Lambda$-représentation admissible et
plate, dont la réduction modulo $\mG_{\Lambda}$ est isomorphe à
$\pi$. Les relèvements de $\pi$ s'organisent en une catégorie cofibrée
en groupoïdes au-dessus de $\AC_{\oFl}$, que nous noterons $\RC\rm
el(\pi)$. De même, on a les catégories (cofibrées sur $\AC_{\oFl}$) 
$\RC\rm el(\rho)$ et $\RC\rm el(\sigma)$ pour des représentations de
$D^{\times}$ et $W_{K}$.

\begin{cor} Soit $\pi$ une $\oFl$-représentation supercuspidale de $G$.
Les foncteurs $\wt\pi\mapsto \rho(\wt\pi)$ et $\wt\pi\mapsto
\sigma(\wt\pi):=\sigma'(\wt\pi)(\frac{1-d}{2})$ induisent des équi\-valences de catégories cofibrées sur
$\AC_{\oFl}$
$${\RC\rm el(\pi)}\simto {\RC\rm el(\rho(\pi))}
\,\hbox{ et }\,\, 
{\RC\rm el(\pi)}\simto {\RC\rm el(\sigma(\pi))},
$$
et il existe des isomorphismes fonctoriels en $(\Lambda,\wt\pi)\in\RC\rm el(\pi)$
$$  \Hom_{\Zl G}\left(H^{d-1}_{c}(\mlt^{\rm
  ca},\Zl),\wt\pi\right) \simto
\rho(\wt\pi)\otimes_{\Lambda}\sigma(\wt\pi) (\frac{d-1}{2}).
$$
De plus, $\wt\pi$ et $\rho(\wt\pi)$ ont un caractère central relié au
déterminant de $\sigma(\wt\pi)$ par la formule 
$$ {\rm car. cent.}(\wt\pi) =  {\rm car. cent.}(\rho(\wt\pi))  =
{\rm det}(\sigma(\wt\pi))\circ {\rm Art}_{K}^{-1} : \, K^{\times}\To{}\Lambda^{\times}.$$
\end{cor}

% Ce corollaire contient le fait que 
% $\pi$, $\rho(\pi)$ et
% $\sigma(\pi)$ ont le même anneau de déformations, avec ou sans
% contraintes sur les caractères centraux et les déterminants.
% En fait, ces anneaux de déformations serons calculés dans l'appendice.
Nous allons donner une version plus concrète de ce corollaire en montrant comment
les déformations universelles permettent de décrire la
\emph{partie supercuspidale} de la cohomologie de $\mlt$.
Précisons d'abord ce que signifie ``partie supercuspidale''. Comme une
$\oFl$-représentation irréduc\-tible supercuspidale n'a d'extensions de
Yoneda qu'avec elle-même, on peut scinder \emph{canoniquement}
toute $\Zlnr$-représentation lisse $V$ de $G$ en une somme directe
$V=V_{\scusp}\oplus V'$, dans laquelle tous les  sous-quotients 
irréductibles de $V_{\scusp}$ sont
supercuspidaux, et aucun sous-quotient irréductible de $V'$ n'est
supercuspidal, \emph{cf} \ref{catsupercusp}.

Fixons un élément $\varpi$ de $\OC$ de
valuation strictement positive, et notons
$\varpi^{\ZM}$ le sous-groupe discret de $K^{\times}$ qu'il
engendre. On peut voir $\varpi^{\ZM}$ comme un sous-groupe central
discret de $G$, de $D^{\times}$, ou de $GD$. On s'intéresse alors aux
$\varpi^{\ZM}$-coinvariants $H^{d-1}_{c}(\mlt^{\rm
  ca},\Zl)_{\varpi^{\ZM}}$ de la cohomologie, qui s'identifient à la
cohomologie $H^{d-1}_{c}(\mlt^{\rm ca}/\varpi^{\ZM},\Zl)$ de la tour
quotientée par l'action libre de $\varpi^{\ZM}$.
Soit $\pi$ une $\oFl$-représentation irréductible supercuspidale de $G/\varpi^{\ZM}$.
On sait que $\rho(\pi)$ est une représentation de
$D^{\times}/\varpi^{\ZM}$, et que le déterminant de $\sigma(\pi)$ vaut
$1$ sur l'élément $\varphi := {\rm Art}_{K}(\varpi)$  de $W_{K}^{\rm ab}$
correspondant à $\varpi$ par le corps de classes local.
Un \emph{$\varpi$-relèvement} de $\pi$, resp. $\rho(\pi)$, est
un relèvement qui se factorise par $G/\varpi^{\ZM}$, resp.
$D^{\times}/\varpi^{\ZM}$. Un
\emph{$\varphi$-relèvement} de $\sigma(\pi)$
 est un relèvement de déterminant $1$ sur $\varphi$.

Soit $\bZl$ le complété $\ell$-adique de $\Zlnr$. Posons maintenant
$$\Lambda_{\pi}^{\varpi}:=\bZl[{\rm
  Syl}_{\ell}(\FM_{q^{f_{\pi}}}^{\times}\times f_{\pi}\ZM/dv\ZM)],$$ où
$f_{\pi}$ est la longueur de $\pi_{|G^{0}}$ et $v={\rm val}_{\OC}(\varpi)$. On remarquera que c'est
une $\bZl$-algèbre locale complète noethérienne de corps résiduel $\oFl$.

\begin{thm4}
\emph{Tordons\footnote{ceci pour éviter l'apparition de contragrédientes} l'action de $G$  par $g\mapsto {^{t}g^{-1}}$.} 
 Il existe alors une décomposition
$$ H^{d-1}_{c}(\mlt^{\rm
  ca}/\varpi^{\ZM},\bZl)_{\scusp}\simeq \bigoplus_{\pi\in {\rm Scusp}_{\oFl}(G/\varpi^{\ZM})}
\wt\pi\otimes_{\Lambda_{\pi}^{\varpi}}\wt{\rho(\pi)}\otimes_{\Lambda_{\pi}^{\varpi}}\wt{\sigma(\pi)}
(\frac{1-d}{2}) $$
où $\wt\pi$ et $\wt{\rho(\pi)}$ sont des $\varpi$-relèvements de $\pi$ et
$\rho(\pi)$ sur $\Lambda_{\pi}^{\varpi}$, et $\wt{\sigma(\pi)}$ est un $\varphi$-relèvement de
$\sigma(\pi)$ sur $\Lambda_{\pi}^{\varpi}$. De plus, on a les propriétés suivantes.
\begin{enumerate}
\item  Ces relèvements sont \emph{universels} en tant que déformations
  ; $\Lambda_{\pi}^{\varpi}$ est donc
  l'anneau de $\varpi$- ou $\varphi$-déformations de $\pi$, $\rho(\pi)$ et $\sigma(\pi)$.
\item $\wt\pi$ est une enveloppe projective de
  $\pi$ dans $\Mo{\bZl}{G/\varpi^{\ZM}}$ et 
 $\wt{\rho(\pi)}$ est une enveloppe projective de $\rho(\pi)$ dans $\Mo{\bZl}{D^{\times}/\varpi^{\ZM}}$.
\end{enumerate}
\end{thm4}

Remarquons que la cohomologie fournit donc  l'anneau de $\varpi$- ou $\varphi$-
déformation de $\pi$, $\rho(\pi)$ et $\sigma(\pi)$ par la formule
$\Lambda_{\pi}^{\varpi}=\endo{\bZl GDW}{H^{d-1}_{c}(\mlt^{\rm
    ca}/\varpi^{\ZM},\bZl)_{\CC_{\pi}^{\varpi}}}$, où l'indice $\CC_{\pi}^{\varpi}$
    signifie ``localisé en $\pi$'' ou ``partie $\pi$-isotypique
généralisée'', \emph{cf} appendice \ref{secscinG}. 
%Par ailleurs, la décomposition du théorème vit en fait sur
%$\Zlnr$.
Par ailleurs, quitte à se débarasser de l'action de $D^{\times}$ en
appliquant le foncteur $\Hom_{\bZl D^{\times}}(\wt{\rho(\pi)},-)$, on
constate que $H^{d-1}_{c}(\mlt^{\rm ca}/\varpi^{\ZM})_{\CC_{\pi}}$
fournit un cas particulier (facile) d'existence de
 ``correspondance de Langlands locale en familles'' au sens de
 \cite{EmHelm}.

\medskip

\emph{Organisation de l'article.} Dans la section
2, nous prouvons le théorème 2, en reportant la preuve de la propriété
de perfection de $R\Gamma_{c}$ à l'appendice A, pour bien séparer les
arguments. Dans la section 3, nous étudions la partie supercuspidale
de la cohomologie et nous prouvons les autre théorèmes annoncés
ci-dessus. Nous avons isolé les ingrédients de pure théorie des
représentations dans l'appendice B, en espérant que cela puisse
clarifier les arguments.

\section{Réalisation virtuelle de Langlands et Jacquet-Langlands
  modulo $\ell$}

Dans cette section, nous prouvons le théorème 2, 
en renvoyant à l'appendice pour les propriétés cohomologiques utilisées. 
Au passage, nous
redémontrons donc l'existence et les propriétés de  
 la correspondance de Langlands $\ell$-modulaire de
 \cite{VigLanglands} et du transfert de Langlands-Jacquet
 $\ell$-modulaire de \cite{jlmodl}.

\subsection{Une propriété de finitude et ses conséquences}

Dans ce paragraphe, $\Lambda$ désigne une $\Zl$-algèbre finie.
Une représentation lisse $V$ de $G$ à coefficients dans
$\Lambda$ est dite \emph{localement de type fini} si pour tout entier
$n$, le module $V^{H_{n}}$ est de type fini sur l'algèbre de Hecke
$\HC_{\Lambda}(G,H_{n})$. Ici, $H_{n}$ désigne le sous-groupe de
congruences $\id+\varpi^{n}\MC_{d}(\OC)$. La propriété de noethériannité suivante nous
sera utile :
\begin{fact}[\cite{finitude}, Thms 1.3 et 1.5] \label{noetherien}
 L'anneau $\HC_{\Lambda}(G,H_{n})$ est noethérien.
En parti\-culier, toute sous-représentation d'une représentation localement de type
  fini est localement de type fini.
\end{fact}

\begin{DEf}
  Un complexe $\CC$ de $\DC^{b}(\Mo{\Lambda}{G})$ sera dit \emph{localement
    parfait} s'il est quasi-isomorphe à un 
complexe borné de la forme $ P_{a}\To{} P_{a+1}\To{}\cdots \To{} P_{b}$ dont chaque
terme $P_{i}$ est projectif et localement de type fini. 
\end{DEf}

On appelle \emph{amplitude parfaite} le plus petit intervalle $[a,b]$
pour lequel on peut trouver un complexe comme ci-dessus. On se gardera
de confondre l'amplitude parfaite et l'amplitude cohomologique.

\begin{prop}\label{propparfait}
Le complexe $R\Gamma_{c}(\mlt^{\rm ca},\Lambda) \in \DC^{b}(\Mo{\Lambda}{G})$
est \emph{localement parfait}, d'amplitude parfaite $[0,2d-2]$.
\end{prop}

Nous renvoyons à l'appendice A pour la preuve de cette proposition. 
On y trouvera aussi une propriété de perfection
pour $R\Gamma_{c}(\mlt^{\rm ca},\Lambda)$ en tant que complexe de 
$D^{\times}$-modules, mais l'amplitude parfaite est dans
ce cas $[d-1,2d-2]$, donc égale à l'amplitude cohomologique, \emph{cf}
proposition \ref{propparfaitD}. Par contre, le complexe \emph{n'est  pas}
simultanément parfait (\emph{i.e.} en tant que complexe de
$GD$-modules), \emph{cf} remarque \ref{remparfait}.

\begin{DEf} Une $\Lambda$-représentation lisse $V$ de $G$ est dite
  \emph{$Z$-finie} si toute $K^{\times}$-orbite dans $V$
  est contenue dans un sous-$\Lambda$-module de type fini.
\end{DEf}

\begin{prop}\label{proptfZf}
Soit $C=\oFl$ ou $\oQl$. 
Pour toute $C$-représentation \emph{$Z$-finie} et \emph{de type fini}
$V$ de $G$, le complexe
$$\Rhom_{\Zl G}(R\Gamma_{c}(\mlt^{\rm ca},\Zl),V^{\vee}) \in
\DC^{+}(\Rep^{\infty,c}_{C}(D^{\times}\times W_{K})) $$ 
est cohomologiquement  borné, d'amplitude contenue dans $[2-2d,0]$, et sa
 cohomologie est de dimension finie. 
\end{prop}
\begin{proof}
L'assertion sur l'amplitude cohomologique du complexe en question
découle immédiatement de celle sur l'amplitude parfaite de la
proposition \ref{propparfait}. Il reste à prouver que la cohomologie
est de dimension finie. Pour cela, il suffit de prouver que pour tous
$i$ et $j$ on a 
$$ \dim_{C} {\rm Ext}_{CG}^{j}(H^{i}_{c}(\mlt^{\rm ca},C),V^{\vee}) < \infty. $$
Soit $G^0:=\ker(\val_{K}\circ{\rm det})$.
D'après l'isomorphisme (3.5.2) de  \cite{lt}, on a  
$H^{i}_{c}(\mlt^{\rm ca},C)=\cind{G^{0}}{G}{H^{i}_{c}(\mlt^{(0),\rm ca},C)}$
où $\mlt^{(0)}$ désigne la tour des déformations de Lubin-Tate
(le lieu où la quasi-isogénie résiduelle est de degré $0$). 
On a donc ${\rm Ext}_{G}^{j}(H^{i}_{c}(\mlt^{\rm ca},C),V^{\vee})={\rm
  Ext}_{G^{0}}^{j}(H^{i}_{c}(\mlt^{(0),\rm ca},C),V^{\vee})$. 

Comme $K^{\times}G^{0}$ est d'indice fini dans $G$, nos hypothèses sur
$V$ impliquent qu'elle est de type fini sur $G^{0}$.
Comme on l'a déjà rappelé, (fait \ref{noetherien}), la catégorie des
représentations de type fini de $G$ est stable par sous-objet, et on
en déduit facilement que celle des représentations de type fini de
$G^{0}$ l'est aussi.
Il s'ensuit que $V_{|G^{0}}$ possède une résolution projective (en
général infinie)
par des représentations de la forme
$\cind{H_{n}}{G^{0}}{1_{C}}^{k}$ (induites à supports compacts). Par passage à la contragrédiente,
$V^{\vee}_{|G^{0}}$ admet une résolution injective par des
représentations de la forme $\ind{H_{n}}{G^{0}}{1_{C}}^{k}$ (induites
sans condition de supports). Or, on a
$\Hom_{G^{0}}(H^{i}_{c}(\mlt^{(0),\rm ca},C),\ind{H_{n}}{G^{0}}{1_{C}})\simeq
\Hom_{H_{n}}(H^{i}_{c}(\mlt^{(0),\rm ca},C),C)\simeq 
H^{i}_{c}(\mltn^{(0),\rm ca},C)^{\vee}$, et on sait que ce dernier $C$-espace
vectoriel est de dimension finie.
\end{proof}

Soit maintenant $\pi$ une $C$-représentation de longueur finie de $G$.
Appliquant la proposition ci-dessus à $\pi^{\vee}$, on voit que
 le complexe $R_{\pi}$ de l'introduction appartient à la sous-catégorie triangulée 
$\DC^{b}_{\rm tf}(\Rep^{\infty,c}_{C}(D^{\times}\times W_{K}))$ des objets de
$\DC^{+}(\Rep^{\infty,c}_{C}(D^{\times}\times W_{K}))$ cohomologiquement bornés et à
cohomologie de dimension finie. Ceci permet de définir la classe 
$$[R_{\pi}]:=\sum_{i\in\ZM}(-1)^{i}[\HC^{i}(R_{\pi})]\in \RC(D^{\times}\times W_{K},C) $$ 
de $R_{\pi}$ dans le groupe de Grothendieck des $C$-représentations de
longueur finie.
Toute suite exacte $\pi_{1}\injo
\pi_{2}\twoheadrightarrow \pi_{3}$ induit un triangle distingué
$R_{\pi_{3}}\To{}R_{\pi_{2}}\To{}R_{\pi_{1}}\To{+1}$, donc on a $[R_{\pi_{2}}]=[R_{\pi_{1}}]+[R_{\pi_{3}}]$.
 On obtient de la sorte un homomorphisme
 $$ \RC(G,C) \To{} \RC(D^{\times}\times W_{K},C)$$
entre groupes de Grothendieck. 

La proposition précédente nous donne une information supplémentaire
sur l'homomorphisme composé $\pi\mapsto [R_{\pi^{\vee}}]$. Celui-ci
doit en effet se factoriser à travers le morphisme canonique
$\RC(G,C)\To{}\KC_{Z}(G,C)$ où
$\KC_{Z}(G,C)$ désigne le groupe de Grothendieck de la catégorie
(abé\-lienne) des $C$-représentations $Z$-finies et de type fini. 
 En particulier, $[R_{\pi}]=0$ dès que la
classe de $\pi^{\vee}$ est nulle dans $\KC_{Z}(G,C)$.
Nous allons voir un exemple intéressant.
\begin{DEf}
  Notons $\RC_{I}(G,C)$ le sous-groupe de $\RC(G,C)$ engendré par les
  induites paraboliques propres. Une $C$-représentation $\pi$ est
 dite \emph{elliptique} si $[\pi] \notin \RC_{I}(G,C)$.
\end{DEf}

\begin{coro} Si $\pi$ n'est pas elliptique, alors
  $[R_{\pi}]=0$. \label{indparab}
\end{coro}
\begin{proof}
  Il suffit de le prouver lorsque $\pi$ est une induite parabolique propre.
  Dans ce cas, $\pi^{\vee}$ est aussi une induite parabolique, disons
  $\pi^{\vee}=\ip{P}{G}{\tau}$, et on sait par un argument qui remonte
  à Kazhdan que son image dans $\KC_{Z}(G,C)$ est nulle. Rappelons
  brièvement cet argument. Soit $P^{1}\subset P$ un sous-groupe ouvert
  contenant le centre $Z$ de $G$ et de quotient $P/P^{1}$ 
libre abélien de rang $1$. Un tel groupe contient aussi la préimage du sous-groupe
$M^{0}_{P}$ du quotient de Levi $M_{P}$ engendré par les éléments compacts. Ainsi
$\tau_{|P^{1}}$ est encore de longueur finie, donc de type fini sur
$P^{1}$, donc son induite $\cind{P^{1}}{P}{\tau}\simeq
\tau\otimes C[P/P^{1}]$ est de type fini
sur $P$. Soit $\gamma$ un générateur de $P/P^{1}$. On a une suite
exacte
$$ 0\To{}\tau\otimes C[P/P^{1}]\To{\id\otimes(1-\gamma)} \tau\otimes
C[P/P^{1}] \To{} \tau\To{} 0.$$ 
En appliquant le foncteur d'induction parabolique, on obtient une
suite exacte de représen\-tations $Z$-finies et de type fini de $G$
$$ 0\To{}\ip{P}{G}{\tau\otimes C[P/P^{1}]}\To{\id\otimes(1-\gamma)} \ip{P}{G}{\tau\otimes
C[P/P^{1}]} \To{} \pi^{\vee}\To{} 0$$ 
qui montre que la classe de $\pi^{\vee}$ est nulle dans $\KC_{Z}(G,C)$.

\end{proof}

\begin{prop}\label{cororeduc}
  Le diagramme suivant est commutatif :
$$ \xymatrix{ \RC(G,\oQl) \ar[r]^-{\pi\mapsto [R_{\pi}]} \ar[d]_{r_{\ell}} &
  \RC(D^{\times}\times W_{K},\oQl) \ar[d]^{r_{\ell}} \\
\RC(G,\oFl) \ar[r]_-{\pi\mapsto [R_{\pi}]} & \RC(D^{\times}\times W_{K},\oFl)
}$$
On y a noté $r_{\ell}$ les morphismes de
décomposition obtenus par semi-simplification de réduc\-tions de modèles entiers.
\end{prop}
\begin{proof}
Soit $\pi$ une $\oQl$-représentation irréductible. On sait qu'elle
admet un modèle $\pi_{\lambda}$ sur une extension finie $E_{\lambda}$ de $\Ql^{\rm nr}$, dont
on note $\Lambda$ l'anneau des entiers. Choisissons 
un sous-$\Lambda G$-module de type fini $\omega$ non nul dans $\pi_{\lambda}$. 

Supposons d'abord que $\pi_{\lambda}$ est $\ell$-entière au sens de
\cite[II.4.11]{Vig}. Alors $\omega$ est  $\Lambda$-admissible, et la   
proposition \ref{propparfait} montre que le complexe
$R_{\omega}:=\Rhom_{\Lambda G}(R\Gamma_{c}(\mlt^{\rm ca},\Lambda),\omega)$
 appartient à la sous-catégorie triangulée 
$\DC^{b}_{\rm tf}(\Rep^{\infty,c}_{\Lambda}(D^{\times}\times W_{K}))$ de la catégorie dérivée des
$\Lambda$-représentations de $D^{\times}\times W_{K}$ formée des
objets à cohomologie bornée et de type fini sur $\Lambda$. 
Par construction on a  
$$R_{\pi}=R_{\omega}\otimes^{L}_{\Lambda}\oQl \, \hbox{ et }\,
R_{\bar\omega}=R_{\omega}\otimes^{L}_{\Lambda}\oFl$$
où l'on a noté $\bar\omega:=\omega\otimes_{\Lambda}\oFl$. Nous devons
montrer que $[R_{\bar\omega}]=r_{\ell}([R_{\pi}])$.

Pour cela, 
notons $\KC(D^{\times}\times W_{K},\Lambda)$ le groupe de Grothendieck de la
catégorie des $\Lambda$-représentations de $D^{\times}\times W_{K}$ qui
sont de type fini sur $\Lambda$. Il coïncide canoniquement avec
le groupe de Grothendieck
$ \KC(\DC^{b}_{\rm tf}(\Rep^{\infty,c}_{\Lambda}(D^{\times}\times W_{K})))$ 
de la catégorie triangulée $\DC^{b}_{\rm tf}(\Rep^{\infty,c}_{\Lambda}(D^{\times}\times W_{K}))$.
Les foncteurs triangulés $-\otimes^{L}_{\Lambda}E_{\lambda}$ et
$-\otimes^{L}_{\Lambda}\oFl$ 
 induisent les deux homomorphismes
diagonaux du triangle
$$\xymatrix{ &
 \KC(D^{\times}\times W_{K},\Lambda) \ar[ld] \ar[rd] & \\
\RC(D^{\times}\times  W_{K}, E_{\lambda}) \ar[rr]^{r_{\ell}}  & & 
\RC(D^{\times}\times  W_{K},\oFl) .
}$$
Il nous suffit donc de voir que ce triangle est commutatif. Or le groupe
$\KC(D^{\times}\times W_{K},\Lambda)$ est engendré par les
$\Lambda(D^{\times}\times W_{K})$-modules sans $\ell$-torsion, et pour un
tel module $\tau$, on a par construction
$r_{\ell}([\tau\otimes_{\Lambda} E_{\lambda}]) = [\tau\otimes_{\Lambda}\oFl]$.

Reste à considérer le cas où $\pi$ n'est pas $\ell$-entière. Dans ce
cas, $\omega=\pi_{\lambda}$ donc $r_{\ell}(\pi)=0$, et il nous faut prouver que
$r_{\ell}([R_{\pi}])=0$. Or, on sait d'une part que $R_{\pi}$
est nul si $\pi$ n'est pas elliptique (corollaire \ref{indparab}), et d'autre part
qu'une représentation elliptique est $\ell$-entière \ssi\ son
caractère central l'est. Ainsi le caractère central de $\pi$ n'est pas
entier. Mais par construction, le centre $K^{\times}$ de $D^{\times}$
agit par l'inverse de ce caractère sur chaque $\HC^{i}(R_{\pi})$. Il
s'ensuit que $\HC^{i}(R_{\pi})$ n'est pas $\ell$-entière et par
conséquent que $r_{\ell}([\HC^{i}(R_{\pi})])=0$.
\end{proof}

\subsection{Les correspondances modulo $\ell$}\label{seccor}

\alin{Rappel du cas $\ell$-adique}
La description explicite de la cohomologie du complexe
$R_{\pi}$ dans \cite[Thm A]{lt} montre 
que pour toute $\oQl$-représentation
irréductible $\pi$, on a la formule
\begin{equation}
[R_{\pi}]=\LJ_{\oQl}(\pi)\otimes\sigma_{\oQl}(\pi)^{\rm ss}
\,\hbox{  dans $\,\RC(D^{\times}\times W_{K},\oQl)$. }
\label{ladic}
\end{equation}
En fait, comme on l'a expliqué dans l'introduction, contrairement à
\cite[Thm A]{lt}, on n'a pas besoin des résultats de Boyer sur la
cohomologie $\ell$-adique, mais simplement de la somme alternée comme
dans \cite[Thm VII.1.5]{HaTay}.

\alin{Définition de $\LJ_{\oFl}$}
D'après (\ref{ladic}), on a $[R_{\pi}]=
d\cdot\LJ_{\oQl}(\pi)$ dans $\RC(D^\times,\oQl)$.
 Pour tirer parti de la proposition \ref{cororeduc}, on doit
utiliser un fait non-trivial de théorie des représentations, qui
découle des travaux de classification de M.-F. Vignéras.

\begin{fac} \emph{cf.} \cite[Cor. (2.2.7)]{jlmodl}.\label{rlsurj}
L'application de décomposition $r_{\ell}:\RC(G,\oQl)\To{}\RC(G,\oFl)$
est surjective.
\end{fac}

Ceci, joint à la proposition \ref{cororeduc},
implique que pour tout $\pi\in\Irr{\oFl}{G}$, la représentation
virtuelle $[R_{\pi}]$ est divisible par $d$ dans
$\RC(D^{\times},\oFl)$. Posons alors 
$$\LJ_{\oFl}(\pi):=\frac{1}{d}[R_{\pi}]  \in \RC(D^{\times},\oFl).$$
L'homomorphisme $\LJ_{\oFl}$ vérifie la propriété de commutation à
$r_{\ell}$ annoncée dans le point ii) du théorème 2. La surjectivité
de $r_{\ell}$ montre d'ailleurs que cette propriété le caractérise. On
retrouve ainsi le transfert de Langlands-Jacquet défini dans
\cite{jlmodl} à l'aide des caractères de Brauer de groupes $p$-adiques.

\begin{rem} \label{remK}
Comme la correspondance de Jacquet-Langlands commute aux
contragrédientes, on a $d\cdot\LJ_{C}(\pi)=[R_{\pi}]=[R_{\pi^{\vee}}]^{\vee}$.
Il résulte alors du paragraphe précédent que l'homomorphisme
$d\cdot\LJ_{C}$ se factorise par un morphisme $\KC_{Z}(G,C)\To{}\RC(D^{\times},C)$. 
Mais ce dernier morphisme n'est pas divisible par $d$. Par exemple
pour $C=\oQl$, il envoie $V=\cind{\GL_{d}(\OC)}{G}{1_{\oQl}}$ sur
$(-1)^{d-1}[1_{\oQl}]$. 
\end{rem}

% \ali Avant d'étudier la partie Galoisienne de $[R_{\pi}]$, nous
% introduisons quelques notations. On note $\nu_{_{G}} :\,
% G\To{}\Zl^{\times}$, resp $\nu_{_{D}} :\,
% D^{\times}\To{}\Zl^{\times}$
% le caractère non ramifié défini par $\nu_{_{G}}(g):= q^{-{\rm
%     val}_{K}({\rm det}(g))}$, resp.  $\nu_{_{D}}(d):= q^{-{\rm
%     val}_{K}({\rm nrd}(d))}$. On note aussi $\nu_{_{W}}:
% W_{K}\To{}\Zl^{\times}$ le caractère défini par $\nu_{_{W}}(w):=
% q^{n(w)}$ où $n(w)$ est l'image de $w$ dans ${\rm Frob_{q}}^{\ZM}\subset
% {\rm Gal}(k^{\rm ca}/k)$. Ces caractères agissent par torsion sur les
% représentations. Par exemple pour une représentation $\pi$ de $G$, on notera simplement $\pi\nu_{_{G}}:=\pi\otimes \nu_{_{G}}$.

% \begin{lem} \label{lemmetorsion}
%   Pour toute $C$-représentation $\pi$ de longueur finie, on a un
%   isomorphisme canonique dans $D^{b}_{C}(D^{\times}\times W_{K})$
% $$ R_{\pi\nu_{_{G}}} \simeq R_{\pi}.(\nu_{_{D}}\otimes\nu_{_{W}})$$
% \end{lem}
% \begin{proof}
%   Cela découle de l'isomorphisme (3.5.2) de \cite{lt}.
% \end{proof}

\alin{Définition de $\sigma_{\oFl}(\pi)^{\rm ss}$} 
Soit  $\pi\in\Irr{\oFl}{G}$.

Si $\pi$ est supercuspidale,
on sait par \cite[III.5.10 ii)]{Vig}
qu'elle se relève en une $\oQl$-représentation.
La proposition
\ref{cororeduc} montre alors que $[R_{\pi}]$ est bien de la forme
$\LJ_{\oFl}(\pi)\otimes\sigma_{\oFl}(\pi)^{\rm ss}$, avec
$\sigma_{\oFl}(\pi)^{\rm ss}=r_{\ell}(\sigma_{\oFl}(\wt\pi))^{\rm ss}$
pour \emph{n'importe quel} $\oQl$-relèvement $\wt\pi$.

Si $\pi$ n'est pas supercuspidale, elle apparait comme sous-quotient d'une induite
parabo\-lique $\pi_{1}\times\cdots\times\pi_{r}$ avec les $\pi_{i}$
supercuspidales. La propriété ``d'unicité du support cuspidal'' \cite[V.4]{VigInduced}
dit que les $\pi_{i}$ sont uniquement déterminés à l'ordre près. On pose alors
\begin{equation}
 \sigma_{\oFl}(\pi)^{\rm ss}:=\sigma_{\oFl}(\pi_{1})^{\rm ss}\oplus\cdots\oplus
\sigma_{\oFl}(\pi_{r})^{\rm ss}.\label{defsig}
\end{equation}
Notons que par transitivité du support supercuspidal, l'égalité
ci-dessus est encore vraie si les $\pi_{i}$ ne sont pas supercuspidales.

\begin{prop}\label{propell}
  Soit $\pi\in\Irr{\oFl}{G}$. 
On a $$[R_{\pi}]=\LJ_{\oFl}(\pi)\otimes\sigma_{\oFl}(\pi)^{\rm ss},\,\,\,\hbox{
  dans } \RC(D^{\times}\times W_{K},\oFl).$$
\end{prop}

\begin{proof}
Supposons d'abord que $\LJ_{\oFl}(\pi)=0$. Dans ce cas, la seule
chose à démontrer est que $[R_{\pi}]=0$ dans $\RC(D^{\times}\times
W_{K},\oFl)$. Or, d'après le théorème (3.1.4) de \cite{jlmodl}, $\pi$
appartient au sous-groupe $\RC_{I}(G,\oFl)$ de $\RC(G,\oFl)$ engendré
par les induites paraboliques propres. Le corollaire \ref{indparab}
nous assure bien que $[R_{\pi}]=0$.

Supposons maintenant que $\pi$ est une représentation ``superSpeh'',
au sens de \cite[Def 2.2.3]{jlmodl}. Par définition, elle est de la
forme $r_{\ell}(\wt\pi)$ avec $\wt\pi$ l'unique sous-représentation
d'une induite parabolique normalisée
$\wt\pi_{0}\times\wt\pi_{0}\nu\times\cdots\times\wt\pi_{0}\nu^{r-1}$
où $r_{\ell}(\wt\pi_{0})=:\pi_{0}$ est une $\oFl$-représentation
supercuspidale de $\GL_{d/r}(K)$ et $\nu$ est le caractère 
$g\mapsto q^{-{\rm val}_{K}({\rm det}(g))}$.
D'après la proposition \ref{cororeduc}, on a
$[R_{\pi}]=\LJ_{\oFl}(\pi)\otimes
r_{\ell}(\sigma_{\oQl}(\wt\pi)^{\rm ss})$. Or, on calcule
$r_{\ell}(\sigma_{\oQl}(\wt\pi)^{\rm ss})=
r_{\ell}(\sum_{i=0}^{r-1}\sigma_{\oQl}(\wt\pi_{0}\nu^{i})^{\rm ss})=
\sum_{i=0}^{r-1} \sigma_{\oFl}(\pi_{0}\nu^{i})^{\rm ss}=\sigma_{\oFl}(\pi)^{\rm
ss}$, puisque les $\pi_{0}\nu^{i}$ sont supercuspidales.

Supposons enfin que $\LJ_{\oFl}(\pi)\neq 0$, ce qui équivaut,
toujours d'après le théoréme (3.1.4) de \cite{jlmodl}, à ce
que $\pi$ soit elliptique. En vertu de la
proposition (2.2.6) de \cite{jlmodl} précisée par le lemme (3.2.1)
 de \emph{loc. cit.}, on peut écrire
$$ [\pi] \equiv \sum_{i=0}^{r-1} a_{i}[\delta_{i}] \,\hbox{ modulo
}\, \RC_{I}(G,\oFl)$$
où les $\delta_{i}$ sont les représentations superSpeh de même support
supercuspidal que $\pi$, et les $a_{i}$ sont des entiers.
Le corollaire \ref{indparab} nous donne alors 
$$ [R_{\pi}] =\sum_{i=0}^{r-1} a_{i}
[R_{\delta_{i}}].$$
Nous venons de montrer que
$[R_{\delta_{i}}]=\LJ_{\oFl}(\delta_{i})\otimes\sigma_{\oFl}(\delta_{i})^{\rm
ss}$ pour tout $i$. Mais par définition,
$\sigma_{\oFl}(\delta_{i})^{\rm ss}$ ne
dépend que du support cuspidal de $\delta_{i}$, et on a
$\sigma_{\oFl}(\delta_{i})^{\rm ss}=\sigma_{\oFl}(\pi)^{\rm ss}$ pour
tout $i$. On en déduit
$$[R_{\pi}] =\sum_{i=0}^{r-1} a_{i}
\LJ_{\oFl}(\delta_{i})\otimes\sigma_{\oFl}(\pi)^{\rm ss} 
= \LJ_{\oFl}(\pi)\otimes\sigma_{\oFl}(\pi)^{\rm ss}.$$
\end{proof}

\begin{prop}
  Soit $\wt\pi\in \Irr{\oQl}{G}$. Pour tout sous-quotient
  irréductible $\pi$ de $r_{\ell}(\wt\pi)$, on a
  $\sigma_{\oFl}(\pi)^{\rm ss}= r_{\ell}(\sigma_{\oQl}(\wt\pi)^{\rm ss}).$
\end{prop}
\begin{proof}
  Supposons que $\wt\pi$ est sous-quotient de
  $\wt\pi_{1}\times\cdots\times\wt\pi_{r}$ avec les $\wt\pi_{i}$
  supercuspidales. Posons $\pi_{i}=r_{\ell}(\wt\pi_{i})$ ; c'est une
  représentation cuspidale, pas nécessairement supercuspidale. On sait
  que
$$ r_{\ell}(\sigma_{\oQl}(\wt\pi)^{\rm ss}) =
\sum_{i=1}^{r} r_{\ell}(\sigma_{\oQl}(\wt\pi_{i})^{\rm ss})$$
et on a par construction
$$ \sigma_{\oFl}(\pi)^{\rm ss}=\sum_{i=1}^{r}
\sigma_{\oFl}(\pi_{i})^{\rm ss}$$
puisque $\pi$ est sous-quotient de l'induite $\pi_{1}\times\cdots\times\pi_{r}$.
Il nous suffit donc de prouver que 
$r_{\ell}(\sigma_{\oQl}(\wt\pi_{i})^{\rm ss}) =
\sigma_{\oFl}(\pi_{i})^{\rm ss}$ pour tout $i$. Si $\pi_{i}$ est
\emph{super}cuspidale, cette égalité est vraie par définition. Si
$\pi_{i}$ n'est pas supercuspidale, on a les égalités
$$ \LJ_{\oFl}(\pi_{i})\otimes r_{\ell}(\sigma_{\oQl}(\wt\pi_{i})^{\rm
  ss}) =
[R_{\pi_{i}}] = \LJ_{\oFl}(\pi_{i})\otimes\sigma_{\oFl}(\pi_{i})^{\rm
  ss}, $$
la première découlant de la proposition \ref{cororeduc}, et la seconde
de la proposition \ref{propell}.
Or on sait que $\pi$ est elliptique, \emph{i.e.} $\LJ_{\oFl}(\pi)\neq
0$, donc on en déduit l'égalité voulue.
\end{proof}

Les deux propositions précédentes achèvent la preuve du théorème
2. Mais nous n'avons rien dit encore sur les propriétés de
l'application $\pi\mapsto\sigma_{\oFl}(\pi)^{\rm ss}$ lorsque $\pi$
est supercuspidale. 
C'est l'objet de la proposition suivante, qui redonne le théorème de Vignéras
sur la correspondance entre supercuspidales et irréduc\-tibles. Une
fois cette proposition prouvée, nous aurons retrouvé l'intégralité 
du théorème 1.6 de \cite{VigLanglands}.

\begin{prop} \label{propcusp} 
  L'application $\pi\mapsto \sigma_{\oFl}(\pi)^{\rm ss}$ induit une
  bijection de l'ensemble des $\oFl$-représentations supercuspidales
  de $G$
  sur celui des $\oFl$-représentations irréductibles de dimension $d$
  de $W_{K}$.
\end{prop}
\begin{proof}
  La proposition \ref{cororeduc} montre que la correspondance de
  Langlands $\ell$-adique est compatible aux congruences : si
  $\wt\pi_{1}$ et $\wt\pi_{2}$ sont deux $\oQl$-représentations
  supercuspidales entières
  et de même réduction $r_{\ell}(\wt\pi_{1})=r_{\ell}(\wt\pi_{2})=\pi$,
  alors
  $r_{\ell}(\sigma_{\oQl}(\wt\pi_{1}))=r_{\ell}(\sigma_{\oQl}(\wt\pi_{2}))=\sigma_{\oFl}(\pi)^{\rm
    ss}$. 

Si $\pi$ est supercuspidale, le critère numérique
\cite[2.3]{VigAENS} montre que $\sigma_{\oFl}(\pi)^{\rm ss}$ est
irré\-ductible. Ce même critère montre aussi que l'application
$\pi\mapsto \sigma_{\oFl}(\pi)^{\rm ss}$ restreinte
à l'ensemble des supercuspidales est \emph{injective}.

Pour prouver la surjectivité, soit
$\sigma$ une $\oFl$-représentation irréductible. Choisissons un
relèvement $\wt\sigma$ de $\sigma$ à $\oQl$ (théorème de
Fong-Swan) et posons $\wt\pi:=\sigma_{\oQl}^{-1}(\wt\sigma)$. C'est une
$\oQl$-représentation supercuspidale entière de $G$. Sa réduction
$\pi:=r_{\ell}(\wt\pi)$ est une $\oFl$-représentation irréductible
cuspidale qui vérifie $\sigma_{\oFl}(\pi)^{\rm
  ss}=r_{\ell}(\sigma_{\oQl}(\wt\pi))= \sigma$ par la proposition
précédente. Comme $\sigma$ est irréductible, $\pi$ est nécessairement
\emph{super}cuspidale (vu la définition de $\sigma_{\oFl}(\pi)^{\rm
  ss}$), et on a prouvé la surjectivité.
\end{proof}

Pour une représentation \emph{supercuspidale}, il n'y a donc plus lieu de
garder la notation $^{\rm ss}$, et on notera simplement $\sigma_{\oFl}(\pi)$.

\alin{Un exemple instructif}\label{exemple}
Reprenons les notations de l'exemple de l'introduction.
  Supposons donc $d=3$ et $q\equiv 1 [\ell]$, et notons $\pi$ le quotient de
  $\CC^{\infty}(\PM^{2}(K),\oFl)$ par les fonctions constantes. Ici on
  peut voir $\PM^{2}(K)$ indifféremment comme l'espace des droites ou
  l'espace des plans ; en d'autres termes, $\pi$ est autoduale.
On  peut alors relever $\pi$ en une $\oQl$-représentation elliptique
 de deux manières différentes : on prend l'espace
$\CC^{\infty}(\PM^{2}(K),\oQl)/\oQl$, et on le munit de l'action
$\wt\pi_{1}$ sur les droites, et de sa transposée $\wt\pi_{2}$ sur les
plans. Ces deux $\oQl$-représentations sont duales l'une de l'autre,
mais pas isomorphes. 

Nous avons calculé les complexes $R_{\wt\pi_{1}}$ et $R_{\wt\pi_{2}}$
dans \cite[4.4.2]{lt}, voir aussi \cite[Thm 4.1.5]{dp}. Leur cohomologie est concentrée
en degrés $-1$ et $1$. Plus précisément, on a 
$$\dim(\HC^{-1}(R_{\wt\pi_{1}}))= 2 \,\hbox{ et }\, 
\dim(\HC^{1}(R_{\wt\pi_{1}}))=1$$
tandis que
$$ \dim(\HC^{-1}(R_{\wt\pi_{2}}))=1 \,\hbox{ et }\, 
\dim(\HC^{1}(R_{\wt\pi_{2}}))=2.$$
On en déduit immédiatement que
$\dim(\HC^{-1}(R_{\pi}))\geq 2$ et $ \dim(\HC^{1}(R_{\pi}))\geq
2$. Cela implique en particulier que la cohomologie des complexes $R_{\omega_{1}}$ et
$R_{\omega_{2}}$ associés aux réseaux stables de $\wt\pi_{1}$ et
$\wt\pi_{2}$ a de la torsion. 
% En particulier,
% $\HC^{1}(R_{\omega_{1}})$ doit avoir de la torsion, de sorte que 
% $\dim(\HC^{0}(R_{\pi}))\geq 1$.

\section{La partie supercuspidale de la cohomologie $\ell$-entière}

\setcounter{subsubsection}{0}
 Jusqu'ici, la notion de \emph{supercuspidalité} n'a été définie que pour
 une représentation irréductible sur $\oFl$. Pour une
 $\Zl$-représentation générale, il nous faudra faire un détour par le groupe
$G^{0}=\ker(\val_{K}\circ\det)$.

\begin{DEf} \label{defsupercusp}
  Un objet $V$ de $\Mo{\Zl}{G^{0}}$ ou de $\Mo{\Zl}{G}$ est dit 
  \emph{supercuspidal} si aucun de ses $\Zl G^{0}$-sous-quotients
  n'est un sous-quotient d'une induite parabolique propre.
\end{DEf}

\begin{rem}Pour $V$ irréductible,
  il n'est pas complètement évident que cette défini\-tion coïncide avec
  la définition de Vignéras, où l'on demande simplement que $V$ ne
  soit pas sous-quotient d'une induite parabolique propre \emph{admissible}.
Nous vérifions que les deux définitions sont  compatibles dans le corollaire \ref{equivdefsupercusp}.
\end{rem}

Une représentation supercuspidale $V$ est en particulier cuspidale
(\emph{i.e.} annulée par tous les foncteurs de Jacquet propres), et
donc les fonctions $g\mapsto e_{H}gv$, pour $v\in V$ et $H$
pro-$p$-sous-groupe ouvert de $G$, sont à support {compact modulo
le centre.} 

Par conséquent, les $\Zl G^{0}$-sous-quotients
irréductibles de $V$ sont des $\Fl$-représentations, et celles-ci apparaissent
comme sous-quotient, et même sous-objet, d'une $\oFl$-représentation
irréductible supercuspidale (au sens habituel) de $G$. 
Vu la remarque ci-dessus, cette propriété des $\Zl
G^{0}$-sous-quotients irréductibles \emph{caractérise} les représentations
supercuspidales au sens de \ref{defsupercusp}.

\begin{exe}
 Si $V$ est une $\oQl$-représentation irréductible, alors elle est
supercuspidale au sens \ref{defsupercusp} \ssi\ c'est un twist non ramifié
d'une représentation $\ell$-entière de réduction supercuspidale au
sens habituel de Vignéras.
\end{exe}

La définition \ref{defsupercusp} peut paraitre alambiquée puisqu'elle fait un détour
par $G^{0}$, mais permet d'avoir la propriété
suivante, prouvée dans l'appendice \ref{coroscindscusp}.

\begin{prop}\label{catsupercusp}
  La sous-catégorie  $\Mosc{\Zl}{G} \subset \Mo{\Zl}{G}$
  (resp. $\Mosc{\Zl}{G^{0}}\subset \Mo{\Zl}{G^0}$)
 dont les objets sont les représentations
  supercuspidales est facteur direct.
\end{prop}

Ainsi,  toute représentation
$V\in\Mo{\Zl}{G}$ se décompose canoniquement en $V=V_{\scusp}\oplus
V'$ avec $V_{\scusp}\in \Mosc{\Zl}{G}$ et $V'$ une représentation dont
aucun $\Zl G^{0}$-sous-quotient n'est supercuspidal au sens de  \ref{defsupercusp}.

Dans cette section, 
nous étudions la partie supercuspidale $H^{n-1}_{c}(\mlt^{\rm
  ca},\Zl)_{\scusp}$ de la cohomologie et
nous prouvons en particulier les théorèmes 1, 3 et 4 de l'introduction.

\subsection{Projectivité et conséquences}

La propriété suivante est la base de notre étude. C'est une
conséquence de la propriété de perfection du complexe
$R\Gamma_{c}(\mlt^{\rm ca},\Zl)$ dans $\DC^{b}(\Mo{\Zl}{G})$.

\begin{prop} \label{propproj}
  La partie supercuspidale $H^{d-1}_{c}(\mlt^{\rm ca},\Zl)_{\scusp}$ est un objet projectif et localement de type fini de $\Mo{\Zl}{G}$.
\end{prop}
\begin{proof} Avant toute chose, remarquons que $H^{d-1}_{c}(\mlt^{\rm
  ca},\Zl)$ est sans $\ell$-torsion, puisque $H^{d-2}_{c}(\mlt^{\rm
  ca},\Fl)$ est nul. 

\medskip
\emph{Première étape : $H^{i}_{c}(\mlt^{\rm ca},\Zl)_{\scusp}=0$
  pour $i\neq d-1$.} Pour cela, rappelons d'abord que
 $$H^{i}_{c}(\mlt^{\rm ca},\Zl)_{\scusp}=
\cind{G^{0}}{G}{H^{i}_{c}(\mlt^{(0),\rm ca},\Zl)_{\scusp}}$$
où $\mlt^{(0)}$ est le lieu où la quasi-isogénie du problème de
déformations est un isomorphisme. Maintenant, 
 l'argument de \cite[Proof of Thm
3.7]{Mieda}, de nature géométrique, montre que pour $i\neq d-1$, 
$H^{i}_{c}(\mlt^{(0),\rm ca},\Zl)$ admet une filtration finie dont les
sous-quotients successifs sont des sous-quotients de représentations
de la forme $\ind{P\cap G^{0}}{G^{0}}{W^{0}}$ avec $W^{0}$ une
représentation \emph{$\Zl$-admissible} de $P\cap G^{0}$. Montrons qu'une telle
représentation est nécessairement triviale sur le radical unipotent.
En effet, d'après \cite[13.2.3]{Boyer1}, l'action du radical unipotent
sur la $\Ql$-représentation admissible $W^{0}\otimes_{\Zl} \Ql$ est triviale.
L'argument de \emph{loc. cit.} appliqué 
à $W^{0}_{\ell\rm -tors}$ en utilisant la longueur au lieu de la
dimension montre que le radical unipotent y agit aussi trivialement.
Comme il est localement pro-$p$, il agit donc trivialement sur $W^{0}$.
Mais alors, l'induite $\ind{P\cap G^{0}}{G^{0}}{W^{0}}$ est une
induite parabolique, et $H^{i}_{c}(\mlt^{(0),\rm ca},\Zl)$ n'a donc
pas de sous-quotient supercuspidal.

\medskip

  \emph{Deuxième étape : $H^{d-1}_{c}(\mlt^{\rm ca},\Zl)_{\scusp}$
  est de dimension cohomologique finie.} En effet, considérons la partie supercuspidale
  $R\Gamma_{c,\scusp}$ du complexe $R\Gamma_{c}(\mlt^{\rm
    ca},\Zl)$. C'est un facteur direct de ce dernier, donc il est
  localement parfait d'après la proposition \ref{propparfait}. D'après
  l'étape précédente, il est concentré en degré $d-1$, \emph{i.e.} on
  a 
  \begin{equation}
R\Gamma_{c}(\mlt^{\rm ca},\Zl)_{\scusp}\simeq H^{d-1}_{c}(\mlt^{\rm ca},\Zl)_{\scusp}[1-d],\label{rgammascusp}
\end{equation}
 d'où la
  propriété annoncée.

\medskip

\emph{Troisième étape : $H^{d-1}_{c}(\mlt^{\rm ca},\Zl)_{\scusp}$ est projectif.} 
Fixons un sous-groupe de congruences $H$ de $\GL_{d}(\OC)$. Pour tout
$V\in\Mo{\Zl}{G}$, notons $\langle V^{H}\rangle$ la sous $\Zl G$-représentation de
$V$ engendrée par les invariants $V^{H}$ de $V$. On sait que c'est un
facteur direct de $V$, \emph{cf} \cite[3.5.8]{lt}. Il nous suffira donc de prouver que
le facteur direct $\langle H^{d-1}_{c}(\mlt^{\rm ca},\Zl)_{\scusp}^{H}\rangle$ est projectif. Ce
dernier peut s'écrire sous la forme
$$ \langle H^{d-1}_{c}(\mlt^{\rm ca},\Zl)_{\scusp}^{H}\rangle =
\cind{G^{0}}{G}{\langle H^{d-1}_{c}(\mlt^{(0),\rm ca},\Zl)_{\scusp}^{H}\rangle}, $$
et il suffit donc de prouver que $\HC:=\langle H^{d-1}_{c}(\mlt^{(0),\rm
  ca},\Zl)_{\scusp}^{H}\rangle$ est projectif dans $\Mo{\Zl}{G^{0}}$.
On sait déjà qu'il est admissible, supercuspidal, et de dimension
cohomologique finie. 
On peut donc trouver une résolution 
$$ 0\To{} P_{d-1}\To{} \cdots\To{}P_{0} \To{} \HC \To{} 0 $$
où chaque $P_{i}$ est supercuspidal, projectif et de type fini (et
donc admissible). Comme tous les membres de cette suite longue sont
sans $\ell$-torsion, on obtient en passant aux contragrédientes une suite exacte longue
$$0\To{} \HC^{\vee} \To{} P_{0}^{\vee} \To{}\cdots\To{}
P_{d-1}^{\vee}\To{} 0.$$
Or, d'après le lemme \ref{lemprojcusp} ci-dessous, chaque $P_{i}^{\vee}$ est encore
projectif. On peut donc simplifier le complexe de manière
inductive à partir de la droite, et on obtient ainsi que $\HC^{\vee}$
est facteur direct de $P_{0}^{\vee}$ et par conséquent est projectif.
D'après le lemme à nouveau, $\HC$ est donc projectif.
\end{proof}

\begin{lemme}\label{lemprojcusp}
  Soit $P\in \Mo{\Zl}{G^{0}}$ une représentation cuspidale, projective
  et de type  fini. Alors sa contragrédiente $P^{\vee}$ est aussi
  cuspidale, projective et de type fini.
\end{lemme}
\begin{proof}
Comme le centre de $G^{0}$ est compact, $P$ est admissible sur
$\Zl$. Sa contragrédiente est donc aussi admissible et on en déduit
aisément qu'elle est cuspidale.  Choisissons des générateurs
$v_{1},\cdots, v_{n}$ de $P$ et un sous-groupe de
congruences $H$ tel que $v_{i}\in P^{H}$ pour tout $i$. Ils définissent
un épimorphisme
$$ p:\, \CC_{c}(G/H,\Zl)^{n}\twoheadrightarrow 
P$$
puis par passage à la contragrédiente une injection
$$ j:\, P^{\vee}\injo \CC(G/H,\Zl)^{n} $$
qui envoie $v^{\vee}$ sur les fonction $g\mapsto \langle v^{\vee}, g v_{i}\rangle$
 pour $i=1,\cdots, n$. Comme $P$ est cuspidale, ces fonctions sont à support
 compact, donc l'image $j(P^{\vee})$ est contenue dans
 $\CC_{c}(G/H,\Zl)^{n}$.
Maintenant, $P$ est projectif sonc $p$ admet une section $i$ qui
induit par dualité une rétraction $q$ de $j$. La restriction de $q$ à
$\CC_{c}(G/H,\Zl)^{n}$ est encore une rétraction, si bien que $P^{\vee}$
est facteur direct de $\CC_{c}(G/H,\Zl)^{n}$, et par conséquent projectif
et de type fini.
\end{proof}

\alin{Preuve du théorème 1} Compte tenu du théorème 2 et de la
proposition \ref{propcusp}, il suffit de montrer que 
$[\Hom_{\Zl G}\left(H^{d-1}_{c}(\mlt^{\rm
  ca},\Zl),\pi\right)]=[R_{\pi}](\frac{d-1}{2})$ dans
$\RC(D^{\times}\times W_{K},\oFl)$. Or on a :
\begin{eqnarray*}
  R_{\pi}(\frac{d-1}{2})  & =  & \Rhom_{\Zl G}\left(R\Gamma_{c}(\mlt^{\rm
      ca},\Zl),\pi\right)[1-d] \\
&  = & \Rhom_{\Zl G}\left(R\Gamma_{c}(\mlt^{\rm
    ca},\Zl)_{\scusp},\pi\right)[1-d] \\
& = & \Rhom_{\Zl G}\left(H^{d-1}_{c}(\mlt^{\rm
    ca},\Zl)_{\scusp}[1-d],\pi\right)[1-d] \\
& = & \Rhom_{\Zl G}\left(H^{d-1}_{c}(\mlt^{\rm
    ca},\Zl)_{\scusp},\pi\right) \\
& = & \Hom_{\Zl G}\left(H^{d-1}_{c}(\mlt^{\rm
    ca},\Zl)_{\scusp},\pi\right) \\
& = & \Hom_{\Zl G}\left(H^{d-1}_{c}(\mlt^{\rm
    ca},\Zl),\pi\right)
\end{eqnarray*}
La deuxième égalité vient du fait que $\pi$ est supposée supercuspidale,
la troisième provient de (\ref{rgammascusp}) et la cinquième
de la projectivité de $H^{d-1}_{c}(\mlt^{\rm
    ca},\Zl)_{\scusp}$.

\begin{rema}\label{Rempreuvealter}
  Soit $\pi$ une $\oFl$-représentation supercuspidale.
  La proposition \ref{propproj} implique que pour tout  $\oQl$-relèvement
  $\wt\pi$ de $\pi$, on a
$$ \left[ \Hom_{\Zl G}\left(H^{d-1}_{c}(\mlt^{\rm
  ca},\Zl),\pi\right) \right]= r_{\ell}\left(\Hom_{\Zl G}\left(H^{d-1}_{c}(\mlt^{\rm
  ca},\Zl),\wt\pi\right)\right),
$$
et donc que la correspondance de Langlands $\ell$-adique est
compatible aux congruences. Cela prouve le point iii) du théorème 1.
A partir de là, le critère numérique de
\cite[2.3]{VigAENS} montre l'irréductibilité de $\rho(\pi)$ et
$\sigma(\pi)$ et l'injectivité des correspondances.
Reste à prouver la surjectivité de $\pi\mapsto \sigma(\pi)$. 
Il faut pour cela prouver que si $\wt\pi'$ est supercuspidale de
réduction \emph{non supercuspidale}, alors $\sigma(\wt\pi')$ est de
réduction \emph{réductible}.
Peut-être qu'un argument de comptage, inconnu de l'auteur, pourrait
suffire. Sinon, il faut utiliser  le théorème 2 et sa
propriété i)(a).
\end{rema}

\begin{rema}
Dans l'énoncé du théorème 1, on peut remplacer 
la cohomologie entière $H^{d-1}_{c}(\mlt^{\rm ca},\Zl)$ par la
cohomologie modulaire $H^{d-1}_{c}(\mlt^{\rm ca},\Fl)$. 
En effet, l'application canonique
  $H^{d-1}_{c}(\mlt^{\rm ca},\Zl)_{\scusp}\otimes \Fl \To{}
H^{d-1}_{c}(\mlt^{\rm ca},\Fl)_{\scusp}$ est un isomorphisme puisque 
$H^{d}_{c}(\mlt^{\rm ca},\Zl)_{\scusp}=0$.   
\end{rema}

\def\bZl{\Zlnr}
\subsection{Description explicite de $H^{d-1}_{c}(\mlt^{\rm ca},\bZl)_{\scusp}$}

Dans ce paragraphe, on étend les scalaires à l'extension non-ramifiée
maximale de $\bZl$ dans $\oZl$, et on s'intéresse au facteur direct
$$ H^{d-1}_{c}(\mlt^{\rm ca},\bZl)_{\scusp} \simeq
H^{d-1}_{c}(\mlt^{\rm ca},\Zl)_{\scusp}\otimes_{\Zl}\bZl.$$
C'est la plus grande sous-$\bZl G$-représentation de
$H^{d-1}_{c}(\mlt^{\rm ca},\bZl)$ dont tous les $\bZl
G^{0}$-sous-quotients irréductibles sont supercuspidaux. 

Soit $\pi$ une $\oFl$-représentation
\emph{supercuspidale}.  D'après la
proposition \ref{propscindageG0}, 
la sous-catégorie pleine $\CC_{\pi}$ de $\Mo{\bZl}{G}$
formée des objets dont tous les $\bZl G^{0}$-sous-quotients irréductibles sont des
sous-quotients de $\pi_{|G^{0}}$ est un facteur direct de
$\Mo{\bZl}{G}$. On note
$V_{\CC_{\pi}}$ le facteur direct  d'une
$\bZl$-représentation $V$ appartenant à $\CC_{\pi}$. On a alors 
$$ V_{\scusp}= \bigoplus_{\pi\in{\rm Scusp}_{\oFl}(G)/\sim} V_{\CC_{\pi}}$$
où $\pi$ parcours un ensemble de représentants des orbites non
ramifiées dans l'ensemble des classes de $\oFl$-représentations
supercuspidales. 

\begin{prop}\label{progeneG}
Pour $\pi\in\Irr{\oFl}{G}$ supercuspidale, 
  $H^{d-1}_{c}(\mlt^{\rm ca},\bZl)_{\CC_{\pi}}$ est un progénérateur
  de $\CC_{\pi}$.
\end{prop}
\begin{proof}
Il est projectif, puisque facteur direct de l'objet projectif
$H^{d-1}_{c}(\mlt^{\rm ca},\bZl)_{\scusp}$.  Il est de type fini,
puisque localement de type fini et engendré par ses $H$-invariants
pour n'importe quel sous-groupe de congruence $H$ tel que $\pi^{H}\neq 0$.
Reste à voir que cet objet est générateur. Or pour tout caractère
  non ramifié $\psi :\, G\To{}\oFl^{\times}$, on a
$$  {\rm Hom}_{\bZl G}(H^{d-1}_{c}(\mlt^{\rm ca},\bZl)_{\CC_{\pi}},\pi\psi)=
 {\rm Hom}_{\bZl G}(H^{d-1}_{c}(\mlt^{\rm
   ca},\bZl),\pi\psi)\neq 0.$$ Les $\pi\psi$
épuisant les objets simples de $\CC_{\pi}$, on en déduit que
$H^{d-1}_{c}(\mlt^{\rm ca},\bZl)_{\CC_{\pi}}$ est générateur de
$\CC_{\pi}$.
  
\end{proof}

\ali Notons maintenant $\rho:=\rho(\pi)$ la correspondante
de Jacquet-Langlands de $\pi$ et $\sigma:=\sigma'(\pi)=\sigma(\pi)(\frac{d-1}{2})$ sa correspondante de
Langlands tordue. 
Dans l'appendice, nous définissons les sous-catégories facteurs directs
$\CC_{\rho^{\vee}}$  de $\Mo{\bZl}{D^{\times}}$, \emph{cf} \ref{propscindageD0} et
$\CC_{\sigma^{\vee}}$ de  $\Moc{\bZl}{W_{K}}$, \emph{cf} \ref{propscindageW}.

\begin{pro} \label{propcrho}
  $H^{d-1}_{c}(\mlt^{\rm ca},\bZl)_{\CC_{\pi}}$ est un objet de
  $\CC_{\rho^{\vee}}$ et de $\CC_{\sigma^{\vee}}$. 
\end{pro}
\begin{proof}
    Montrons d'abord que c'est bien un objet de $\CC_{\rho^{\vee}}$. Pour cela
  considérons le facteur direct $(H^{d-1}_{c}(\mlt^{\rm
    ca},\bZl)_{\CC_{\pi}})^{\CC_{\rho^{\vee}}}$ hors de
  $\CC_{\rho^{\vee}}$. 
S'il est non nul, il admet
  un quotient $G$-irréductible non nul, de la forme $\pi(\psi\circ
  \det)$ pour un caractère non ramifié de $K^{\times}$. Or on sait par
  le théorème 1 que 
$\Hom_{G}(H^{d-1}_{c}(\mlt^{\rm ca},\bZl)_{\CC_{\pi}},\pi(\psi\circ
  \det))$ est $\rho(\psi\circ
  \nrd)$-isotypique, et donc que $\rho^{\vee}(\psi^{-1}\circ
  \nrd)$ intervient comme quotient de $(H^{d-1}_{c}(\mlt^{\rm
    ca},\bZl)_{\CC_{\pi}})^{\CC_{\rho^{\vee}}}$, ce qui est par
  définition impossible.

De la même manière, on montre que   $H^{d-1}_{c}(\mlt^{\rm
  ca},\bZl)_{\CC_{\pi}}$ est un objet de 
 $\CC_{\sigma^{\vee}}$
\end{proof}

\begin{rem}
Les seuls résultats sur la cohomologie de Lubin-Tate  que nous utilisons
dans cet article sont ceux de Harris-Taylor, et ceux de
Mieda-Strauch. En utilisant plus, on peut obtenir plus. Voici quelques
exemples :
\begin{enumerate}
\item En utilisant les résultats de Boyer \cite{Boyer2} sur la $\oQl$-cohomologie de
  Lubin-Tate, on montre facilement que :
 $$H^{d-1}_{c}(\mlt^{\rm ca},\bZl)_{\CC_{\pi}}=
H^{d-1}_{c}(\mlt^{\rm ca},\bZl)_{\CC_{\rho^{\vee}}}=
H^{d-1}_{c}(\mlt^{\rm ca},\bZl)_{\CC_{\sigma^{\vee}}}.$$

\item En utilisant les résultats annoncés par Boyer \cite{Boyer3} sur l'absence de
  torsion dans les $H^{i}_{c}(\mlt^{\rm ca},\bZl)$, on montre même que
$$R\Gamma_{c}(\mlt^{\rm ca},\bZl)_{\CC_{\pi}}
=R\Gamma_{c}(\mlt^{\rm ca},\bZl)_{\CC_{\rho^{\vee}}} 
=R\Gamma_{c}(\mlt^{\rm ca},\bZl)_{\CC_{\sigma^{\vee}}}$$
en tant que facteurs directs de $R\Gamma_{c}(\mlt^{\rm ca},\bZl)$ dans
$\DC^{b}(\Rep^{\infty,c}_{\bZl}(GD\times W_{K}))$.

\item Enfin, en utilisant les résultats de Faltings et Fargues \cite{FarFal} de
  comparaison avec la cohomologie de la tour de Drinfeld, on peut
  montrer que $R\Gamma_{c}(\mlt^{\rm ca},\bZl)$ est parfait dans
  $\DC^{b}(\Rep^{\infty}_{\Zl}(D^{\times}))$, \emph{cf.} proposition
  \ref{propparfaitD},
 et on en déduit \emph{a priori} que
  \begin{center}
    $H^{d-1}_{c}(\mlt^{\rm ca},\bZl)_{\CC_{\pi}}$ est un objet
    projectif de $\CC_{\rho^{\vee}}$.
  \end{center}
En fait nous trouverons cette propriété \emph{a posteriori}, comme
conséquence de notre description explicite.
\end{enumerate}
\end{rem}

\alin{Résumé de l'appendice B}
Dans l'appendice, nous exhibons des
progénérateurs $P_{\pi}$, $P_{\rho^{\vee}}$ et $P_{\sigma^{\vee}}$ respectivement de
$\CC_{\pi}$, $\CC_{\rho^{\vee}}$ et $\CC_{\sigma^{\vee}}$, et calculons leurs
commutants ``opposés''
$$\ZG_{\pi}=\endo{\bZl G}{P_{\pi}}^{\opp},\,\,
\ZG_{\rho^{\vee}}=\endo{\bZl D^{\times}}{P_{\rho^{\vee}}}^{\opp}\, \hbox{ et }
\ZG_{\sigma^{\vee}}=\endo{\bZl W_{K}}{P_{\sigma^{\vee}}}^{\opp}.$$ On dispose
alors d'équivalences de catégories associées à ces progénérateurs,
\emph{cf} \ref{faitequiv}:
\begin{itemize}\item 
  $\Hom_{\bZl G}(P_{\pi},-):\,\CC_{\pi}\simto \Mod(\ZG_{\pi})$
  d'inverse 
$ P_{\pi}\otimes_{\ZG_{\pi}}- :\,  \Mod(\ZG_{\pi})\simto\CC_{\pi}$
\item 
  $\Hom_{\bZl G}(P_{\rho^{\vee}},-):\,\CC_{\rho^{\vee}}\simto
  \Mod(\ZG_{\rho^{\vee}})$
  d'inverse 
$ P_{\rho^{\vee}}\otimes_{\ZG_{\rho^{\vee}}}- :\, \Mod(\ZG_{\rho^{\vee}})\simto\CC_{\rho^{\vee}}$
\item $\Hom_{\bZl
    W_{K}}(P_{\sigma^{\vee}},-):\,\CC_{\sigma^{\vee}}\simto
   \Mod(\ZG_{\sigma^{\vee}})$  d'inverse 
$ P_{\sigma^{\vee}}\otimes_{\ZG_{\sigma^{\vee}}}- :\, \Mod(\ZG_{\sigma^{\vee}})\simto\CC_{\sigma^{\vee}}$.
\end{itemize}

D'après le $ii)(b)$ de la proposition  \ref{propscindageG0}, $\ZG_{\pi}$ est
isomorphe à l'algèbre $\bZl[{\rm
  Syl}_{\ell}(\FM_{q^{f_{\pi}}}^{\times})\times \ZM]$ où $f_{\pi}$
désigne la longueur de $\pi_{|G^{0}}$, et d'après la proposition
\ref{propscindageD0},
$\ZG_{\rho^{\vee}}$ est isomorphe à
l'algèbre $\bZl[{\rm
  Syl}_{\ell}(\FM_{q^{f_{\rho^{\vee}}}}^{\times})\times \ZM]$ où $f_{\rho^{\vee}}$
désigne la longueur de $\rho_{|\OC_{D}^{\times}}$. 
Ces anneaux sont donc commutatifs. Par ailleurs, d'après la
proposition \ref{propscindageW}, $\ZG_{\sigma^{\vee}}$  est
isomorphe au produit croisé de $\bZl[[\Zl]]$ par
$\ZM$, le générateur de $\ZM$ agissant par multiplication par
$q^{f_{\sigma^{\vee}}}$ sur $\Zl$,  où $f_{\sigma^{\vee}}$
désigne la longueur de $\sigma_{|I_{K}}$. Ainsi, $\ZG_{\sigma^{\vee}}$ n'est pas
commutatif, mais son plus grand quotient commutatif $\ZG_{\sigma^{\vee}}^{\rm
ab}$ est
isomorphe à l'algèbre $\bZl[{\rm
  Syl}_{\ell}(\FM_{q^{f_{\sigma^{\vee}}}}^{\times})\times \ZM]$.

\begin{lem}
On a les égalités $f_{\pi}=f_{\rho^{\vee}}=f_{\sigma^{\vee}}$.
\end{lem}
  \begin{proof}
Soit $\wt\pi$ un $\oQl$-relèvement de $\pi$. Une formule donnant
$f_{\pi}$ en termes d'un type contenu dans $\pi$ est donnée à la fin
de la preuve du $i)(b)$ de la proposition \ref{propscindageG0}. Cette
formule montre que $f_{\pi}=f_{\wt\pi}:={\rm
  long}(\wt\pi_{|G^{0}})$. Par la théorie de Clifford, ceci permet de réinterpréter $f_{\pi}$ comme
le nombre de $\oQl$-caractères non ramifiés $\psi$ de $G$ tels que
$\wt\pi\psi\simeq \wt\pi$.
De même $f_{\rho^{\vee}}= f_{\rho(\wt\pi)^{\vee}}$ est le nombre 
de $\oQl$-caractères non ramifiés $\psi$ de $G$ tels que
$\rho(\wt\pi)\psi\simeq \rho(\wt\pi)$, et idem pour $f_{\sigma^{\vee}}$.
Or, les correspondances $\wt\pi\mapsto \rho(\wt\pi)$ et $\wt\pi\mapsto
\sigma(\wt\pi)$ sont compatibles à la torsion par les $\oQl$-caractères.
 \end{proof}

Le lemme montre que $\ZG_{\pi}$,
    $\ZG_{\rho^{\vee}}$ et $\ZG_{\sigma^{\vee}}^{\rm ab}$ sont
    abstraitement isomorphes, en tant que $\bZl$-algèbres. Comme on va
    le voir, la cohomologie fournit des isomorphismes particuliers
    entre ses anneaux.

\ali  Posons maintenant $\HC:=H^{d-1}_{c}(\mlt^{\rm ca},\bZl)_{\CC_{\pi}}$ et $\ZG:=\endo{\bZl
    GDW}{\HC}$, et considérons
  \begin{eqnarray*}
 \HC' & := &\Hom_{\bZl [G\times D^{\times}\times W_{K}]}\left(
  P_{\pi}\otimes_{\bZl} P_{\rho^{\vee}}\otimes_{\bZl} P_{\sigma^{\vee}}, \HC \right)
\\ & \in & \Mod(\ZG_{\pi}\otimes_{\bZl}\ZG_{\rho^{\vee}}\otimes_{\bZl}\ZG_{\sigma^{\vee}})
.
\end{eqnarray*}

Via les équivalences mentionnées
ci-dessus, on a
$\ZG\simto \endo{\ZG_{\pi}\otimes\ZG_{\rho^{\vee}}\otimes\ZG_{\sigma^{\vee}}}{\HC'}$. 
En particulier, les actions des anneaux commutatifs  $\ZG_{\pi}$ et
$\ZG_{\rho^{\vee}}$ fournissent des morphismes de $\bZl$-algèbres
$\ZG_{\pi}\To{}\ZG$, et $\ZG_{\rho^{\vee}}\To{}\ZG$.

\begin{thm}\label{theocoho}
  \begin{enumerate}
  \item Les morphismes d'action $\ZG_{\pi}\To{}\ZG$ et
    $\ZG_{\rho^{\vee}}\To{}\ZG$ sont bijectifs.
  \item L'action de $\ZG_{\sigma^{\vee}}$ sur $\HC'$
se factorise par $\ZG_{\sigma^{\vee}}^{\rm ab}$ et induit un isomorphisme
    $\ZG_{\sigma^{\vee}}^{\rm ab}\simto \ZG$.
 \item $\HC'$ est libre de rang $1$ sur $\ZG$.
Tout générateur $\varphi$ de $\HC'$ sur $\ZG$ se factorise à
    travers un $\ZG[GDW]$-isomorphisme
$$\bar\varphi:\,
P_{\pi}\otimes_{\ZG}P_{\rho^{\vee}}\otimes_{\ZG}\otimes
    P_{\sigma^{\vee}}^{\rm ab} \simto H^{d-1}_{c}(\mlt^{\rm ca},\bZl)_{\CC_{\pi}}$$
où $P_{\sigma^{\vee}}^{\rm
  ab}:=P_{\sigma^{\vee}}\otimes_{\ZG_{\sigma^{\vee}}}\ZG_{\sigma^{\vee}}^{\rm
  ab}$ (\emph{cf} \ref{csigmaab}) et le
produit tensoriel est relatif aux actions de $\ZG$ déduites des 
isomorphismes précédents.
  \end{enumerate}
\end{thm}
\begin{proof}
  $i)$  D'après la proposition \ref{progeneG},
 on sait déjà que $\HC'$ est projectif et de type fini sur
  $\ZG_{\pi}$. Comme $\Spec(\ZG_{\pi})$ est connexe, le rang de $\HC'$
  est constant, et non nul puisque $\HC'$ est non nul.
Le morphisme
  $\ZG_{\pi}\To{a}\ZG\subset\endo{\bZl}{\HC'}$ est donc injectif.
Pour montrer qu'il est bijectif, il suffit de prouver que le rang de
$\HC'$ sur $\ZG_{\pi}$ est égal à $1$, puisqu'alors on aura
$\ZG\subset\endo{\ZG_{\pi}}{\HC'}=a(\ZG_{\pi})$.
Toujours par connexité de $\Spec(\ZG_{\pi})$, on peut
 calculer ce rang en (co)spécialisant en le $\oFl$-caractère
  ``trivial'' de $\ZG_{\pi}$, \emph{i.e.} celui qui correspond à $\pi$
via l'équivalence de catégories $\CC_{\pi}\simto \Mod(\ZG_{\pi})$. On touve alors
\begin{eqnarray*}
\Hom_{\ZG_{\pi}}(\HC',\oFl) & \simeq & \Hom_{\bZl G}\left(\Hom_{\bZl DW}(P_{\rho^{\vee}}\otimes
P_{\sigma^{\vee}}, \HC),\pi\right) \\
& \simeq & \Hom_{\bZl DW}\left(P_{\rho}\otimes
P_{\sigma}, \Hom_{\bZl G}(\HC,\pi)\right) \\
& \simeq & \Hom_{\bZl D}(P_{\rho}, \rho)\otimes_{\bZl} 
\Hom_{\bZl W_{K}}(P_{\sigma},\sigma) \simeq \oFl, 
\end{eqnarray*}
où le deuxième isomorphisme provient des lemmes \ref{dualite} et
\ref{dualiteW},
 et le troisième provient du
théorème 1.
Ainsi, $\HC'$ est bien de rang $1$ sur $\ZG_{\pi}$. 

Considérons maintenant le morphisme composé
$\ZG_{\rho^{\vee}}\To{}\ZG\simto \ZG_{\pi}$.
Ce morphisme induit une bijection entre $\oQl$-caractères de 
$\ZG_{\pi}$ et $\oQl$-caractères de $\ZG_{\rho^{\vee}}$
(la correspondance de Jacquet-Langlands !). A partir de là, on
raisonne comme dans le point ii) ci-dessous pour en conclure que c'est
un isomorphisme.

$ii)$ Comme l'action de $\ZG_{\sigma^{\vee}}$ commute à $\ZG_{\pi}$, elle est
donnée par un morphisme
$\ZG_{\sigma^{\vee}}\To{\beta}\endo{\ZG_{\pi}}{\HC'}=\ZG_{\pi}$ qui se factorise
donc par $\ZG_{\sigma^{\vee}}^{\rm ab}\To{}\ZG_{\pi}=\ZG$. 
Ce morphisme induit une bijection entre $\oQl$-caractères de
$\ZG_{\pi}$ et $\oQl$-caractères de $\ZG_{\sigma^{\vee}}$ (la correspondance
de Langlands !). Il est donc \emph{injectif}, puisque
$\ZG_{\sigma^{\vee}}^{\rm ab}$ est réduit et sans $\ell$-torsion.
 Identifions ce morphisme $\beta$ à un
endomorphisme de la $\bZl$-algèbre
$\bZl[C_{\ell^{a}}\times \ZM]$, où $C_{\ell^{a}}={\rm
  Syl}_{\ell}(\FM_{q^{f}}^{\times})$ est un groupe
\emph{cyclique} d'ordre $\ell^{a}$.
 Observons que  le groupe 
% des racines de l'unité d'ordre $\ell^{b}$ dans l'anneau
$\mu_{\ell^{a}}(\bZl[C_{\ell^{a}}\times \ZM])$ est égal à
$\mu_{\ell^{a}}(\bZl[C_{\ell^{a}}])$.
Ainsi $\beta$ stabilise $\bZl[C_{\ell^{a}}]$ et induit un
endomorphisme injectif de cet anneau. Or par définition, cet anneau
$\bZl[C_{\ell^{a}}]$ est engendré par son groupe de $\ell^{a}$-racines de l'unité
 $\mu_{\ell^{a}}(\bZl[C_{\ell^{a}}])$ qui est fini. Comme $\beta$
  est injectif, il induit une permutation de
  $\mu_{\ell^{a}}(\bZl[C_{\ell^{a}}])$, donc finalement un
  automorphisme de $\bZl[C_{\ell^{a}}]$.
 Quittes à composer par un prolongement de cet automorphisme à $\bZl[C_{\ell^{a}}\times\ZM]$, on
 peut considérer maintenant $\beta$ comme un endomorphisme injectif de
\emph{ $\bZl[C_{\ell^{a}}]$-algèbres}. Cet endomorphisme induit aussi une
 bijection entre $\oFl$-caractères (issue de la correspondance de Langlands mod
 $\ell$), donc ses fibres sont encore injectives, donc sont plates
 puisque morphismes finis entre anneaux principaux (isomorphes à
 $\oFl[X,X^{-1}]$). Le critère de platitude fibre à fibre nous dit
 alors que $\beta$ est plat, et donc localement libre puisque fini. Mais
 son rang est nécessairement $1$, donc $\beta$ est un isomorphisme.

$iii)$ Puisque $P_{\pi}$, $P_{\rho^{\vee}}$ et $P_{\sigma^{\vee}}$ sont des
progénérateurs, le
morphisme d'évaluation est un isomorphisme :
$$ (P_{\pi}\otimes_{\bZl}P_{\rho^{\vee}}\otimes_{\bZl}P_{\sigma^{\vee}})
\otimes_{\ZG_{\pi}\otimes_{\bZl}\ZG_{\rho^{\vee}}\otimes_{\bZl}\ZG_{\sigma^{\vee}}}
\HC'\simto \HC.$$
Par les points $i)$ et $ii)$, on peut le réécrire sous la forme
$$ (P_{\pi}\otimes_{\ZG}P_{\rho^{\vee}}\otimes_{\ZG}P_{\sigma^{\vee}}^{\rm ab})
\otimes_{\ZG} \HC'\simto \HC.$$
Il ne reste donc plus qu'à prouver que $\HC'$ est libre.

Pour cela, notons $\HC^{0}$ la plus grande sous-$G^{0}$-représentation de
$H^{d-1}_{c}(\mlt^{(0),\rm ca},\bZl)$ dont tous les sous-quotients
irréductibles sont isomorphes à un sous-quotient irréductible de
$\pi_{|G^{0}}$. C'est un facteur direct stable par $(G\times
D^{\times}\times W_{K})^{0}$, et on a
$\HC=\cind{GDW^{0}}{GDW}{\HC^{0}}$, \emph{cf} preuve du $ii)(a)$ de la
proposition \ref{propscindageG0}. On vérifie alors que
\begin{eqnarray*}
 \Hom_{\bZl DW}(P_{\rho^{\vee}}\otimes_{\bZl}P_{\sigma^{\vee}}, \HC)
& \simeq & \Hom_{\bZl D^{0}W^{0}}(P_{\rho^{\vee}_{0}}\otimes_{\bZl}P_{\sigma^{\vee}_{0}}, \HC)\\
& \simeq & \cind{G^{0}}{G}{\Hom_{\bZl
    D^{0}W^{0}}(P_{\rho^{\vee}_{0}}\otimes_{\bZl}P_{\sigma^{\vee}_{0}},
  \HC^{0})}
\end{eqnarray*}
où on a écrit
$P_{\rho^{\vee}}=\cind{D^{0}}{D}{P_{\rho^{\vee}_{0}}}$ et
$P_{\sigma^{\vee}}=\cind{W^{0}}{W}{P_{\sigma^{\vee}_{0}}}$. D'après le
lemme \ref{lemmecommut}, on en déduit que $\HC'$ est, en tant que $\ZG_{\pi}$-module, de la forme 
$\HC'=\ZG_{\pi}\otimes_{\ZG_{\pi^{0}}}\HC'_{0}$ pour un certain
$\ZG_{\pi^{0}}$-module $\HC'_{0}$. Comme l'extension
$\ZG_{\pi^{0}}\subset \ZG_{\pi}$ est fidèlement plate, $\HC'_{0}$ est
plat et de type fini sur $\ZG_{\pi^{0}}$. Or, ce dernier anneau est local, donc
$\HC'_{0}$ est libre sur $\ZG_{\pi^{0}}$, et par conséquent $\HC'$ est
libre sur $\ZG_{\pi}$.
\end{proof}

\alin{Les foncteurs $\Pi\mapsto \sigma'(\Pi)$ et $\Pi\mapsto \rho(\Pi)$}\label{defsigmapi}

Soit $(\Pi,V)$ une $\bZl$-représentation de $G$.
Supposons dans un premier temps que $\Pi$ est  engendrée
par ses invariants sous un sous-groupe ouvert compact. Cette propriété
de finitude assure que $V_{\CC_{\pi}}$ est nul sauf pour un nombre fini de
classes inertielles de $\oFl$-représentations supercuspidales. 
Le théorème \ref{theocoho} et les propositions \ref{dualite} et
\ref{dualiteW} nous donnent un isomorphisme :
$$ \Hom_{\Zl G}\left(H^{d-1}_{c}(\mlt^{\rm ca},\Zl),\Pi\right)
\simeq \bigoplus_{\pi\in{\rm Scusp}_{\oFl}(G)/\sim}
P_{\rho(\pi)}\otimes_{\ZG_{\pi}}P_{\sigma'(\pi)}^{\rm ab}\otimes_{\ZG_{\pi}}
\Hom_{\bZl G}(P_{\pi},\Pi)$$
que l'on peut réécrire sous la forme plus condensée
\begin{equation}
  \label{Hpi}
  \Hom_{\Zl G}\left(H^{d-1}_{c}(\mlt^{\rm ca},\Zl),\Pi\right) \simeq
P_{\scusp}^{D}\otimes_{\ZG_{\scusp}} P^{W,\rm ab}_{\scusp}
\otimes_{\ZG_{\scusp}} \Hom_{\bZl G}\left(P_{\scusp}^{G},\Pi\right)
\end{equation}
où l'on a posé 
$ P_{\scusp}^{G}:=\bigoplus  P_{\pi}, $
$ P_{\scusp}^{D}:=\bigoplus  P_{\rho(\pi)}, $
$ P_{\scusp}^{W,\rm ab}:=\bigoplus  P_{\sigma'(\pi)}^{\rm ab}, $
 et $\ZG_{\scusp}:=\prod \ZG_{\pi}$.
Notons que le terme de droite fait sens pour tout $\Pi$ (sans
l'hypothèse de finitude), et cela nous permet de poser 
$$\sigma'(\Pi):=P_{\scusp}^{W,{\rm ab}}\otimes_{\ZG_{\scusp}} 
 \Hom_{\bZl G}\left(P_{\scusp}^{G},\Pi\right) 
\in \Moc{\bZl}{W_{K}}$$
et 
$$ \rho(\Pi) := P_{\scusp}^{D}\otimes_{\ZG_{\scusp}} 
 \Hom_{\bZl G}\left(P_{\scusp}^{G},\Pi\right) 
\in \Mo{\bZl}{D^{\times}}.$$
Considérons maintenant les trois catégories suivantes :
\begin{enumerate}
\item $\Rep^{\scusp}_{\Zlnr}(G)$ la sous-catégorie pleine
 de $\Mo{\Zlnr}{G}$ formée des représentations 
supercuspidales au sens de \ref{defsupercusp}. Elle est facteur
direct, pro-engendrée par $P_{\scusp}^{G}$, et son centre s'identifie
canoniquement à $\ZG_{\scusp} =\endo{\Zlnr}{P_{\scusp}^{G}}$.
\item $\Rep^{\scusp}_{\Zlnr}(D^{\times})$ la sous-catégorie pleine
 de $\Mo{\Zlnr}{D^{\times}}$ formée des représentations dont tous les
 $\Zlnr D^{0}$-sous-quotients irréductibles sont des sous-quotients de
 $\oFl D$-représenta\-tions qui correspondent à des supercuspidales.
 Elle est facteur direct, pro-engendrée par
 $P_{\scusp}^{D}$, et le théorème \ref{theocoho} nous donne un
  isomorphisme $\ZG_{\scusp}\simto \endo{\bZl
    D^{\times}}{P_{\scusp}^{D}}$.

\item $\Rep^{d, \rm ab}_{\Zlnr}(W_{K})$ la plus petite sous-catégorie
 pleine de $\Moc{\Zlnr}{W_{K}}$ contenant $P_{\scusp}^{W,\rm ab}$,
 stable par sous-quotients et sommes directes. Elle est pro-engendrée
 par $P_{\scusp}^{W,\rm ab}$, et le théorème \ref{theocoho} nous donne un
  isomorphisme $\ZG_{\scusp}\simto \endo{\bZl
    W_{K}}{P_{\scusp}^{W,\rm ab}}$.

Notons qu'elle n'est pas stable par extensions. Sa ``clôture par
extensions'' est la sous-catégorie facteur direct
$\Rep^{d}_{\Zlnr}(W_{K})$ formée des représentations dont tous les
 $\Zlnr I_{K}$-sous-quotients irréductibles sont sous-quotients d'une
 $\oFl$-représentation irréductible de dimension $d$ de $W_{K}$.
\end{enumerate}

\begin{thm} \label{equivcat}
  Avec les notations ci-dessus,
  \begin{enumerate}
\item le fonteur $\Pi\mapsto \sigma'(\Pi)$ est une équivalence
de catégories $\Rep^{\scusp}_{\bZl}(G)\simto {\rm
  Rep}^{d, \rm ab}_{\bZl}(W_{K})$,
  \item le fonteur $\Pi\mapsto \rho(\Pi)$ est une équivalence
de catégories $\Rep^{\scusp}_{\bZl}(G)\simto {\rm
  Rep}^{\scusp}_{\bZl}(D^{\times})$,
\item Pour $\Pi,\Pi' \in\Mo{\bZl}{G}$ engendrées par leurs invariants sous un
  sous-groupe ouvert compact, on a des isomorphismes bifonctoriels en
  $\Pi$, $\Pi'$ 
$$\Hom_{\Zl G}\left(H^{d-1}_{c}(\mlt^{\rm ca},\Zl),\Pi\right)\otimes_{\ZG_{\scusp}}
\Hom_{\bZl G}(P_{\scusp}^{G},\Pi')\simto
\rho(\Pi)\otimes_{\ZG_{\scusp}}\sigma'(\Pi')
$$
\end{enumerate}
\end{thm}
Dans le point iii), l'action de $\ZG_{\scusp}$ sur chaque terme est l'action qui
provient fonctoriellement de
l'action canonique du centre de la catégorie ${\rm
  Rep}^{\scusp}_{\bZl}(G)$ sur $\Pi$ et $\Pi'$.
\begin{proof}
$i)$ En revenant à la somme indexée par $\pi\in\rm Scusp_{\oFl}(G)$ et
en utilisant les équivalence de catégories du type \ref{faitequiv}, 
on vérifie aisément que le foncteur
  $\PC\mapsto P_{\scusp}^{G}\otimes_{\ZG_{\scusp}}
  \Hom_{\bZl D^{\times}}\left(P_{\scusp}^{D},\PC\right) $ est une
  équivalence inverse de $\Pi\mapsto \rho(\Pi)$.

$ii)$ De même, le foncteur  $\Sigma\mapsto
  P_{\scusp}^{G}\otimes_{\ZG_{\scusp}} \Hom_{\bZl
    W_{K}}\left(P_{\scusp}^{W,\rm ab},\Sigma\right) $ est une
  équivalence inverse de $\Pi\mapsto \sigma'(\Pi)$.

$iii)$ découle de (\ref{Hpi}) et des définitions de $\sigma'(\Pi)$ et $\rho(\Pi)$.
\end{proof}

\alin{Preuve du théorème 3} \label{preuvethm3}
Le point $ii)$ du théorème 3 découle clairement de celui du théorème
\ref{equivcat}. Le point $i)$ du théorème 3 découlera de celui du
théorème \ref{equivcat}, une fois qu'on aura vérifié que toute
$R$-famille $\Sigma$ de représentations irréductibles de dimension $d$ de
$W_{K}$ est bien un objet de $\Rep^{d, \rm ab}_{\bZl}(W_{K})$.
Pour cela, on peut supposer $\Spec(R)$ connexe, auquel cas  $\Sigma$ est un objet de
$\CC_{\sigma}$ pour une certaine $\oFl$-représentation irréductible de
dimension $d$ de $W_{K}$. Mais alors, avec les notations de la
proposition \ref{propscindageW}, la $R W_{L}$-représentation
$R_{\wt\tau}(\Sigma)$ est une $R$-famille de caractères de $W_{L}$,
donc se factorise par $W_{L}^{\rm ab}$, et  $\Sigma$ appartient à
$\CC_{\sigma}^{\rm ab}$ qui est contenue dans $\Rep^{d, \rm ab}_{\bZl}(W_{K})$.

Il reste à  prouver le point iii) du théorème 3.
Supposons que $\Pi$ est une $R$-famille de représentations
irréductibles supercuspidales de $G$. Comme $R$ est supposée
noethérienne,  $\Spec(R)$ a un nombre fini de composantes connexes et
$\Pi$ est donc engendrée par ses invariants sous un sous-groupe ouvert
compact. 
De plus, $\Hom_{\bZl G}(P_{\scusp}^{G},\Pi)$ est une $R$-famille
plate de \emph{caractères} de $\ZG$. Son $R$-module sous-jacent
$\RC(\Pi)$ est donc localement libre de rang $1$ et
l'action de $\ZG_{\scusp}$ est donnée par un morphisme
$\ZG_{\scusp}\To{} R$. Le point iii) du théorème  \ref{equivcat} nous
donne  alors un $R\otimes_{\ZG_{\scusp}} R$-isomorphisme 
$$\Hom_{\Zl G}\left(H^{d-1}_{c}(\mlt^{\rm ca},\Zl),\Pi\right)\otimes_{\ZG_{\scusp}}
\RC(\Pi) \simto
\rho(\Pi)\otimes_{\ZG_{\scusp}}\sigma'(\Pi)
$$
et on en déduit le point iii) du théorème 3 par extension des
scalaires selon le morphisme produit $R\otimes_{\ZG_{\scusp}}R\To{} R$.

\alin{Preuve du théorème 4} 
On fixe $\varpi\in K^{\times}$ de valuation $v>0$ et on s'intéresse à 
$H^{d-1}_{c}(\mlt^{\rm  ca}/\varpi^{\ZM},\bZl)_{\scusp}$. On a une
décomposition
$$H^{d-1}_{c}(\mlt^{\rm  ca}/\varpi^{\ZM},\bZl)_{\scusp}
=\bigoplus_{\pi\in{\rm Scusp}_{\oFl}(G/\varpi^{\ZM})}
H^{d-1}_{c}(\mlt^{\rm  ca}/\varpi^{\ZM},\bZl)_{\CC_{\pi}^{\varpi}} $$
où la notation $\CC_{\pi}^{\varpi}$ est expliquée dans l'appendice
\ref{secscinG}.
Fixons donc $\pi$, notons $\rho:=\rho(\pi)$ et
$\sigma:=\sigma'(\pi)=\sigma(\pi)(\frac{d-1}{2})$, et
posons $$\HC_{\varpi}:=H^{d-1}_{c}(\mlt^{\rm ca}/\varpi^{\ZM},\bZl)_{\CC_{\pi}^{\varpi}}.$$ 
Par les mêmes arguments que les propositions \ref{progeneG} et
\ref{propcrho}, on obtient que $\HC_{\varpi}$
est un objet projectif de $\CC_{\pi}^{\varpi}$, et qu'en tant que
$D^{\times}$-représentation, c'est un objet de
$\CC_{\rho^{\vee}}^{\varpi}$. C'est aussi un objet de
$\CC_{\sigma^{\vee}}$. Nous allons le décrire en suivant la preuve
du théorème \ref{theocoho} ci-dessus.

Soit $P_{\pi}^{\varpi}$ le progénérateur de $\CC_{\pi}^{\varpi}$, de commutant
noté $\ZG^{\varpi}_{\pi}$, et  décrit dans la proposition
\ref{propscindageG}, et de même soient $P_{\rho^{\vee}}^{\varpi}$ et
$\ZG^{\varpi}_{\rho^{\vee}}$ comme dans la proposition \ref{propscindageD}.
Nous utiliserons aussi les foncteurs $R_{\wt\tau}$ et $I_{\wt\tau}$ de
la proposition \ref{propscindageW} (appliquée à $\sigma^{\vee}$).
Posons
$$ P_{\sigma^{\vee}}^{\varpi}:=\Hom_{\bZl G\times
  D^{\times}}(P_{\pi}^{\varpi}\otimes_{\bZl}P_{\rho^{\vee}}^{\varpi},
\HC_{\varpi}) \hbox{ et } \HC_{\varpi}':=R_{\wt\tau}(P_{\sigma^{\vee}}^{\varpi}).$$
Ce sont des  $\ZG^{\varpi}_{\pi}\otimes_{\bZl}\ZG^{\varpi}_{\rho^{\vee}}$-modules,
munis respectivement d'une
action de $W_{K}$ et de $W_{L}$ (notation de la proposition \ref{propscindageW}).
De plus, on a les isomorphismes suivants :
$$P_{\sigma^{\vee}}^{\varpi}\simeq I_{\wt\tau}(\HC_{\varpi}')\,\hbox{ et
}\,(P_{\pi}^{\varpi}\otimes_{\bZl}P_{\rho^{\vee}}^{\varpi}) 
\otimes_{\ZG^{\varpi}_{\pi}\otimes_{\bZl}\ZG^{\varpi}_{\rho^{\vee}}} P_{\sigma^{\vee}}^{\varpi} \simto \HC_{\varpi}.$$
Maintenant, on sait que $\HC_{\varpi}'$ est un module projectif sur
$\ZG^{\varpi}_{\pi}$, donc libre, puisque $\ZG^{\varpi}_{\pi}$ est local.
De plus,  par le
même argument que dans la preuve du $i)$ du théorème \ref{theocoho}, on
obtient qu'il est de rang $1$.
Par conséquent, le morphisme d'action
$\ZG^{\varpi}_{\pi}\To{}\endo{\bZl}{\HC_{\varpi}'}$ est injectif et le
morphisme d'action
$\ZG^{\varpi}_{\rho^{\vee}}\To{}\endo{\bZl}{\HC_{\varpi}'}$
se factorise par $\ZG^{\varpi}_{\rho^{\vee}}\To{\beta}\ZG^{\varpi}_{\pi}$. Comme
dans le point ii) du théorème \ref{theocoho}, ce dernier morphisme est
injectif car il induit des bijections entre caractères et les anneaux
sont réduits, et il est
surjectif car ces deux anneaux sont engendrés par leur groupe (fini)
de racines $\ell^{?}$-èmes de l'unité.

Ceci nous permet de réécrire
le second isomorphisme ci-dessus comme
$$
P_{\pi}^{\varpi}\otimes_{\ZG_{\pi}^{\varpi}}P_{\rho^{\vee}}^{\varpi}
\otimes_{\ZG_{\pi}^{\varpi}} P_{\sigma^{\vee}}^{\varpi} \simto \HC_{\varpi}.
$$
Le $\ZG_{\pi}^{\varpi} [W_{K}]$-module $P_{\sigma^{\vee}}^{\varpi}=I_{\wt\tau}(\HC'_{\varpi})$ est
projectif sur $\ZG_{\pi}^{\varpi}$, et de réduction isomorphe à $\sigma^{\vee}$ ; c'est
donc un relèvement de $\sigma^{\vee}$ sur $\ZG_{\pi}^{\varpi}$. On sait que 
$$P_{\sigma^{\vee}}^{\varpi}\otimes\oQl
=\bigoplus_{\pi^{\dag}\in\Irr{\oQl}{G/\varpi^{\ZM}},
  r_{\ell}(\pi^{\dag})=\pi} \sigma_{d}(\pi^{\dag})^{\vee}.$$
Par les propriétés de la correspondance de Langlands, chaque
$\sigma_{d}(\pi^{\dag})^{\vee}$ est de déter\-minant $1$, et il s'ensuit que
$P_{\sigma^{\vee}}^{\varpi}$ est aussi de déterminant $1$. C'est donc un
$\varphi$-relèvement. 

Pour prouver le théorème 4, il ne reste plus qu'à étendre les
scalaires de $\Zlnr$ à son complété $\breve\Zl$.
En effet :
\begin{itemize}\item   $\wt\pi:=P_{\pi}^{\varpi}\otimes_{\bZl}\breve\Zl$ est la
  $\varpi$-déformation universelle de $\pi$ d'après
  \ref{propscindageG} iii).
\item
  $\wt{\rho^{\vee}}:=P_{\rho^{\vee}}^{\varpi}\otimes_{\bZl}\breve\Zl$ est la
  $\varpi$-déformation universelle de $\rho$ d'après
  \ref{propscindageD} iii).

\item 
  $\wt\sigma:=P_{\sigma}^{\varpi}\otimes_{\bZl}\breve\Zl$ est
  la $\varphi$-déformation universelle  d'après le ii) du corollaire
  \ref{coroW}.
\end{itemize}

\subsection{Descente à $\Zl$}

Si l'on veut descendre le théorème $3$ aux familles indexées par une
$\Zl$-algèbre, il faut pouvoir étendre les foncteurs $\Pi\mapsto
\sigma'(\Pi)$ et $\Pi\mapsto \rho(\Pi)$, et pour cela il faut trouver
des modèles $\Zl$-rationnels de $P_{\scusp}^{G}$, $P_{\scusp}^{D}$, et
$P_{\scusp}^{W}$. Or cela est possible pour $P_{\scusp}^{G}$, mais
\emph{pas pour $P_{\scusp}^{D}$, et
$P_{\scusp}^{W}$} ! (exercice laissé au lecteur).

Dans cette section, nous fixons un élément $\varpi\in K$ de valuation
$v>0$, et nous allons donner une version $\Zl$-rationnelle
du théorème 4 de l'introduction. Cela permet de descendre les
foncteurs du théorème 3 à la catégorie $\Mo{\Zl}{G/\varpi^{\ZM}}$. 

\alin{Compatibilité des correspondances à l'action de Galois}
Comme dans le cas $\ell$-adique, les correspondances de Langlands et
Jacquet-Langlands $\ell$-modulaires sont compatibles à l'action de
Galois :
\begin{lem}
  Les applications $\pi\mapsto \rho(\pi)$ et $\pi\mapsto \sigma'(\pi):=\sigma(\pi)(\frac{d-1}{2})$
  du théorème 1 sont compatibles à l'action de $\gal(\oFl/\Fl)$ sur
  les ensembles de classes de $\oFl$-représentations irréductibles concernés.
\end{lem}
\begin{proof}
  La compatibilité de  $\pi\mapsto \rho(\pi)$ à l'action de Galois découle de la
  caractérisation par les caractères de Brauer de \cite{jlmodl}.
Le point i) du théorème 1 nous assure que l'application $\pi\mapsto
\rho(\pi)\otimes\sigma'(\pi)$ est aussi compatible à l'action de
Galois. Il s'ensuit que $\pi\mapsto \sigma'(\pi)$ l'est aussi.
\end{proof}

\alin{Correspondance entre $\Fl$-représentations}
Dans le cas $\ell$-adique, on sait que les $\oQl$-représentations
supercuspidales sont définies sur leur corps de rationalité, mais ce
n'est pas toujours le cas pour les représentations de $D^{\times}$ et $W_{K}$.
Cette obstruction n'existe pas sur $\oFl$ puisque tout corps fini est
commutatif. 
On déduit donc du lemme précédent :
\begin{cor}
  Pour tout corps fini $\FM$ de caractéristique $\ell$,
 les correspondances $\pi\mapsto \sigma'(\pi)$ et $\pi\mapsto
      \rho(\pi)$ descendent en des correspondances entre
      classes de $\FM$-représentations \emph{absolument irréductibles}, la
      première étant une bijection entre ``supercuspidales'' et ``de
      dimension $d$'', la seconde étant injective. De plus,  on a encore
$$\Hom_{\Zl G}\left(H^{d-1}_{c}(\mlt^{\rm ca},\Zl),\pi\right) \simeq \rho(\pi)\otimes_{\FM}\sigma'(\pi).$$
\end{cor}

Changeant de point de vue, on passe aux représentations irréductibles
sur $\Fl$. Le commutant d'une telle représentation
$\pi$ est une extension finie $\FM_{\pi}$ de $\Fl$ de degré $d_{\pi}$.
Si l'on voit $\pi$ comme une $\FM_{\pi}$-représentation irréductible, 
 le corollaire ci-dessus nous fournit des
$\FM_{\pi}$-représentations absolument irréductibles $\sigma'_{\FM_{\pi}}(\pi)$ et
$\rho_{\FM_{\pi}}(\pi)$. Comme les $\gal(\FM_{\pi}/\Fl)$-orbites de
ces représentations ont le même cardinal $d_{\pi}$ que celui de la
$\gal(\FM_{\pi}/\Fl)$-orbite de $\pi$, elles sont aussi irréductibles,
\emph{en tant que $\Fl$-représentations.}

\begin{cor} Par le procédé ci-dessus, on obtient :
  \begin{enumerate}
  \item Une bijection $\pi\mapsto \sigma'(\pi)$ entre classes de
    $\Fl$-représentations irréductibles supercuspidales de $G$ et
    classes de $\Fl$-représentations irréductibles de $W_{K}$ de
    dimension $d$ sur leur commutant,
  \item Une injection $\pi\mapsto \rho(\pi)$ des classes de
    $\Fl$-représentations irréductibles supercuspidales de $G$ dans
    les classes de $\Fl$-représentations irréductibles de
    $D^{\times}$,
  \end{enumerate}
uniquement caractérisées par la propriété suivante :
\begin{enumerate} \setcounter{enumi}{2}
\item la $\Fl$-représen\-tation (irréductible) $\Hom_{\Zl G}\left(H^{d-1}_{c}(\mlt^{\rm
  ca},\Zl),\pi\right)$ de $D^{\times}\times W_{K}$ est
  $\rho(\pi)$-isotypique et $\sigma'(\pi)$-isotypique.

\end{enumerate}
\end{cor}

\alin{Scindage de catégories}
Soit $\pi$ une $\Fl$-représentation irréductible supercuspidale de
$G$. Choisissons un plongement $\FM_{\pi}\injo \oFl$.
La $\oFl$-représentation  $\o\pi:=\pi\otimes_{\FM_{\pi}}\oFl$ contient
un type $(J,\o\lambda')$ comme dans
le paragraphe \ref{propscindageG}. L'unicité de la paire
$(J,\o\lambda')$ à $G$-conjugaison près nous assure que le corps de
rationalité de $\o\lambda'$ est le même que celui de $\o\pi$, à
savoir $\FM_{\pi}$. Donc  $\o\lambda'$ est de la forme
$\lambda'\otimes_{\FM_{\pi}}\oFl$. Choisissons une enveloppe
projective $P_{\lambda'}$ de $\lambda'$ dans
$\Mo{\Wi(\FM_{\pi})}{J/\varpi^{\ZM}}$ (ici $\Wi(\FM_{\pi})$ désigne
l'anneau des vecteurs de Witt à coefficients dans $\FM_{\pi}$), et posons
$P_{\pi}^{\varpi}:=\cind{J}{G}{P_{\lambda'}}$. C'est une
$\FM_{\pi}$-représentation, mais nous oublions maintenant la
$\FM_{\pi}$-structure. 

\begin{pro}
  La sous-catégorie pleine $\CC_{\pi}^{\varpi}$ de $\Mo{\Zl}{G/\varpi^{\ZM}}$
  formée des objets dont tous les sous-quotients irréductibles sont isomorphes à
  $\pi$, est facteur direct et pro-engendrée par $P_{\pi}^{\varpi}$. 
De plus, on a 
$$\ZG_{\pi}^{\varpi}:=\endo{G}{P_{\pi}^{\varpi}}\simeq \Wi(\FM_{\pi})[{\rm
  Syl}_{\ell}(\FM_{q^{f}}^{\times} \times f\ZM/dv\ZM ]$$
où $f$ est la longueur de $\pi_{|G^{0}}$ sur $\FM_{\pi}$.  
\end{pro}
\begin{proof}
Les $\FM_{\pi} J$-sous-quotients irrédutibles de $P_{\lambda'}$ sont tous
isomorphes à $\lambda'$. 
Comme $\lambda'$ est $\Fl J$-irréductible, la même propriété est vraie
pour les $\Fl J$-sous-quotients, et on en déduit les deux premières
assertions comme dans \ref{propscindageG}. Le calcul de commutant a
été effectué dans la proposition \ref{propscindageG} ii).
\end{proof}

De la même manière, on construit 
  un pro-générateur $P_{\rho(\pi)}^{\varpi}$ de la sous-catégorie
  $\CC_{\rho(\pi)}^{\varpi}$ de $\Mo{\Zl}{D^{\times}}$ définie de
  manière maintenant usuelle, et dont le commutant $\ZG_{\rho(\pi)}^{\varpi}$
  est isomorphe (abstraitement) à $\ZG_{\pi}^{\varpi}$.

\alin{Les résultats}\label{resultdescente}
Décomposons $$H^{d-1}_{c}(\mlt^{\rm ca}/\varpi^{\ZM},\Zl)_{\scusp} =
\bigoplus_{\pi\in{\rm Scusp}_{\Fl}(G/\varpi^{\ZM})}  
H^{d-1}_{c}(\mlt^{\rm ca}/\varpi^{\ZM},\Zl)_{\CC_{\pi}^{\varpi}}.$$
Par les mêmes arguments que le théorème \ref{theocoho} et le
théorème 4, on prouve :

\begin{thm}
Il existe un isomorphisme $\ZG_{\pi}^{\varpi}\simto
\ZG_{\rho(\pi)}^{\varpi}$ et une décomposition
$$P_{\pi}^{\varpi}\otimes_{\ZG_{\pi}^{\varpi}}P_{\rho(\pi)^{\vee}}^{\varpi}
\otimes_{\ZG_{\pi}^{\varpi}} P_{\sigma'(\pi)^{\vee}}^{\varpi} \simto
H^{d-1}_{c}(\mlt^{\rm ca}/\varpi^{\ZM},\Zl)_{\CC_{\pi}^{\varpi}},$$
où $P_{\sigma'(\pi)^{\vee}}^{\varpi}$ est un $\varphi$-relèvement de
$\sigma'(\pi)^{\vee}$ sur $\ZG_{\pi}^{\varpi}$, qui est universel pour
les $W(\FM_{\pi})$-$\varphi$-déformations.
\end{thm}

\`A partir de là, on définit deux foncteurs 
$$ \Pi\mapsto \rho(\Pi):= \bigoplus_{\pi\in{\rm
    Scusp}_{\Fl}(G/\varpi^{\ZM})}  
P_{\rho(\pi)}^{\varpi}\otimes_{\ZG_{\pi}^{\varpi}}\Hom_{\Zl
  G}(P_{\pi}^{\varpi},\pi) $$
% \Mo{\Zl}{G/\varpi^{\ZM}}\To{}\Mo{\Zl}{G/\varpi^{\ZM}} 
$$ \hbox{et }\,\,\Pi\mapsto \sigma'(\Pi):= \bigoplus_{\pi\in{\rm
    Scusp}_{\Fl}(G/\varpi^{\ZM})}  
P_{\sigma'(\pi)}^{\varpi}\otimes_{\ZG_{\pi}^{\varpi}}\Hom_{\Zl
  G}(P_{\pi}^{\varpi},\pi).$$
% :\,\, \Mo{\Zl}{G/\varpi^{\ZM}}\To{}\Mo{\Zl}{W_{K}}.
On obtient alors le même énoncé (et avec la même preuve) que le
théorème 3 mais pour $R$ une $\Zl$-algèbre, et en se restreignant aux
$R$-familles de $G/\varpi^{\ZM}$-représentations, resp. de
$D^{\times}/\varpi^{\ZM}$-représentations, resp. de
$W_{K}$-représentations de déterminant $1$
sur $\varphi$.

% \begin{cor} Ces foncteurs sont exacts. Pour $R$ une $\Zl$-algèbre noethérienne, on a 
% \begin{enumerate}
% \item $\Pi\mapsto \sigma'(\Pi)$ induit une bijection entre
%  classes de $R$-familles de représentations irréduc\-tibles supercuspidales de $G/\varpi^{\ZM}$ et
%  classes de  $R$-familles de représentations irréduc\-tibles de
%  dimension $d$ de $W_{K}$ et de déterminant $1$ sur $\varphi$.
% \item  $\Pi\mapsto \rho(\Pi)$ induit une injection de l'ensemble des
%   classes de  $R$-familles de représen\-tations irréductibles supercuspidales de $G$ dans celui
%  des classes de  $R$-familles de représen\-tations irréductibles de $D^{\times}$.
% \item Si $\Pi$ est une $R$-famille de représentations irréductibles
%   supercuspidales de $G$, alors il existe un $R$-module inversible
%   $\RC(\Pi)$ et un isomorphisme
% $$ \Hom_{\Zl G}\left(H^{d-1}_{c}(\mlt^{\rm
%     ca},\Zl),\Pi\right)\otimes_{R}\RC(\Pi)
% \simeq \rho(\Pi)\otimes_{R}\sigma'(\Pi) $$
% \end{enumerate}
% \end{cor}

\begin{rema}
Lorsque l'on enlève la condition ``$\varpi$'', on ne peut parfois pas trouver de $P_{\rho(\pi)}$
 dont le commutant soit un anneau commutatif (ou,
  de manière équivalente, tel que $\Hom_{\Zl
    G}(P_{\rho(\pi)},\rho(\pi))$ soit de dimension $1$). Ceci est lié
  au fait que $\rho(\pi)$ et $\rho(\pi)^{0}$ n'ont pas nécessairement
  le même corps de rationalité. Lorsque cela se produit, on ne peut
  plus décomposer $H^{d-1}_{c}(\mlt^{\rm ca},\Zl)_{\CC_{\pi}}$ en un
 triple produit tensoriel. Plus précisément, on ne peut pas séparer
 l'action de $D^{\times}$ de celle de $W_{K}$.

\end{rema}

\appendix

\section{Propriétés de perfection du complexe de cohomologie}

\subsection{Perfection dans $\DC(G)$}
\label{appperfG}
Rappelons l'énoncé que nous voulons prouver, pour une $\Zl$-algèbre
finie $\Lambda$ :

\begin{pro}
Le complexe $R\Gamma_{c}(\mlt^{\rm ca},\Lambda) \in \DC^{b}(\Mo{\Lambda}{G})$
est \emph{localement parfait}, d'amplitude parfaite $[0,2d-2]$.
\end{pro}

Notre preuve est une adaptation de l'argument de Deligne et Lusztig
de \cite[Prop. 3.7]{DL}, avec quelques complications techniques.

\alin{Preuve dans le cas de torsion} Nous fixons un niveau $H=H_{n}$
et notons $\HC:=\HC_{\Lambda}(G,H_{n})$ l'algèbre de Hecke correspondante.
Nous noterons simplement
$R\Gamma_{c}^{H}$ le complexe $R\Gamma_{c}(\mlt^{\rm ca},\Lambda)^{H}$
qui est un objet de la catégorie dérivée $\DC^{b}(\Mod(\HC))$  des
$\HC$-modules à gauche. On veut montrer qu'il est parfait d'amplitude
contenue dans $[0,2d-2]$.

\emph{Première étape.} 
D'après \cite[Lemme 3.5.9]{lt}, ce complexe
s'identifie à $R\Gamma_{c}(\mltn,\Lambda)$. En particulier, 
ses groupes de cohomologie  sont
 de type fini sur $\HC$,
et nuls en dégré $> 2d-2$.
Par \ref{noetherien}, on peut donc représenter $R\Gamma_{c}^{H}$ par
un complexe $(P_{n}, d_{n})_{n\leq 2d-2}$ de $\HC$-modules projectifs de type fini, à composantes
nulles en degré $> 2d-2$. Tronquons ce complexe en degré $0$ en
remplaçant $P_{0}$ par $Q_{0}:=P_{0}/{\rm im}(d_{-1})$ et en posant
$Q_{i}=P_{i}$ pour $i>0$. On obtient
un complexe borné de $\HC$-modules, nul en dehors de $[0,2d-2]$, dont tous les termes sont
projectifs, sauf éventuellement $Q_{0}$, et isomorphe à
$R\Gamma_{c}^{H}$ dans $\DC^{b}(\Mod(\HC))$. 

\emph{Deuxième étape.} Il nous faut maintenant prouver
que $Q_{0}$ est projectif. Comme il est de présentation finie, il
suffira de prouver qu'il est \emph{plat}, \emph{i.e.} que pour tout
$i>0$ et pour tout $\HC$-module à droite $M$, le $\Lambda$-module
 $\tor{i}{M}{Q_{0}}{\HC}=\HC^{-i}(M\otimes^{L}_{\HC}Q_{0})$ est
 nul. Il est alors utile de remarquer que le morphisme canonique
 $R\Gamma_{c}^{H}\To{} Q_{0}$ dans $\DC^{b}(\Mod(\HC))$ induit justement des isomorphismes
$$ \HC^{-i}(M\otimes^{L}_{\HC}R\Gamma_{c}^{H})\simto \HC^{-i}(M\otimes^{L}_{\HC}Q_{0})$$
pour $i>0$ ; ceci se voit facilement sur la suite spectrale
$E_{1}^{pq}=\tor{p}{M}{Q_{q}}{\HC}\Rightarrow
\HC^{q-p}(M\otimes^{L}_{\HC}R\Gamma_{c}^{H})$
qui dégénère en $E_{2}$ puisque $E_{1}^{pq}=0$ dès que $pq\neq 0$.

\emph{Troisième étape.} Afin de mieux comprendre l'objet
$M\otimes^{L}_{\HC}R\Gamma_{c}^{H}$, nous exprimons le complexe
$R\Gamma_{c}^{H}$ d'une autre manière. Soit $\PC_{\rm
  LT}$
l'espace des périodes de Lubin-Tate,
et soit 
$$\xi_{n}:\mltn\To{}\PC_{\rm LT}^{\rm nr}\simeq\PM^{d-1}_{\knr}$$  le morphisme de
périodes de Gross-Hopkins et Rapoport-Zink. Comme $\xi_{n}$ est un
morphisme étale, on a $R\xi_{n,!}(\Lambda) =
\xi_{n,!}(\Lambda)$, et donc un isomorphisme canonique dans $\DC^{b}(\Mod(\Lambda))$
\begin{equation}
 R\Gamma_{c}(\mltn^{\rm ca},\Lambda) \simto R\Gamma_{c}(\PC^{\rm
   ca}_{\rm LT}, \xi_{n,!}(\Lambda)).\label{isom}
 \end{equation}

\begin{lem}
 Le faisceau étale $\xi_{n,!}(\Lambda)$ sur $\plt^{\rm nr}$ est naturellement
  muni d'une structure de faisceau de $\HC$-modules, telle que 
  \begin{enumerate}
\item    ses fibres sont isomorphes au $\HC$-module
    $\CC_{c}(H\ba G,\Lambda)$ (fonctions à support compact).
  \item l'isomorphisme (\ref{isom}) provient d'un isomorphisme dans $\DC^{b}(\Mod(\HC))$.
  \end{enumerate}
\end{lem}
\begin{proof}
Commençons plus généralement avec un faisceau étale de
$\Lambda$-modules $\FC$ sur
$\plt^{\rm nr}$, et montrons comment le faisceau
$\xi_{n,!}\xi_{n}^{*}(\FC)$ est naturellement un faisceau de $\HC$-modules.
Soit $U\To{f}\PC_{\rm LT}^{\rm nr}$ un morphisme étale.
 Pour expliciter la structure de $\HC$-module sur
 $\Gamma(U,\xi_{n,!}\xi_{n}^{*}(\FC))$, posons $U_{n}:=\mltn\times_{\plt^{\rm nr}}U$.
On peut alors identifier $\Gamma(U,\xi_{n,!}\xi_{n}^{*}(\FC))$ au
 sous-module $\Gamma_{!}(U_{n},\xi_{n}^{*}(\FC))\subset \Gamma(U_{n},\xi_{n}^{*}(\FC))$
 des sections à support propre au-dessus de $U$.
Comme dans \cite[3.3.2]{lt}, la colimite 
$\limind_{m}\Gamma_{!}(U_{m},\xi_{m}^{*}(\FC))$ est munie d'une action lisse
de $G$. Comme dans \cite[3.5.10]{lt}, le morphisme canonique
$$ \Gamma(U,\xi_{n,!}\xi_{n}^{*}(\FC))=\Gamma_{!}(U_{n},\xi_{n}^{*}(\FC))\To{}
\left(\limind_{m}\Gamma_{!}(U_{m},\xi_{m}^{*}(\FC)\right)^{H_{n}} $$
est un isomorphisme et permet donc de munir
 le terme de gauche d'une action de $\HC$. De même on a un
 isomorphisme
$$ \Gamma_{c}(U,\xi_{n,!}\xi_{n}^{*}(\FC))=\Gamma_{c}(U_{n},\xi_{n}^{*}(\FC))\To{}
\left(\limind_{m}\Gamma_{c}(U_{m},\xi_{m}^{*}(\FC)\right)^{H_{n}} $$
qui montre que $\Gamma_{c}(U,\xi_{n,!}\xi_{n}^{*}(\FC))$ est un
sous-$\HC$-module de $\Gamma(U,\xi_{n,!}\xi_{n}^{*}(\FC))$.
Si le $\Lambda$-module $\FC$ est injectif, il en est de même de
$\xi_{n}^{*}(\FC)$ (puisque $\xi_{n}^{*}$ est adjoint à droite du
foncteur exact  $\xi_{n,!}$), et le faisceau
$\xi_{n,!}\xi_{n}^{*}(\FC)$ est donc  $\Gamma_{c}(\plt^{\rm
  ca},-)$-acyclique (par \cite[Thm 5.2.2]{Bic2}).
 Appliquant ceci à une résolution injective de $\Lambda$
sur $\plt^{\rm nr}$, 
on obtient l'isomorphisme annoncé au ii) dans
$\DC^{b}(\Mod(\HC))$ :
$$ R\Gamma_{c}(\plt^{\rm ca},\xi_{n,!}(\Lambda))\simto
R\Gamma_{c}(\mlt^{\rm ca},\Lambda)^{H_{n}}.$$
Finalement, le point i) découle du fait que $\xi_{n}$ fait de
$\mltn$ un revêtement étale au sens de 
\cite{DeJong}, dont les fibres géométriques sont isomorphes à $G/H$.

\end{proof}

\emph{Quatrième étape.}
Rappelons maintenant que le foncteur $R\Gamma_{c}(\plt^{\rm ca}, -)$,
dont la source naturelle est $\DC^{+}(\Modtop{\Lambda}{\wt\plt_{\rm
  et}})$ 
(catégorie dérivée des faisceaux étales de $\Lambda$-modules sur
$\plt$,  \emph{cf} \cite[3x.3]{lt} pour les notations), 
est de dimension cohomologique finie, égale à $2d-2$. Il se prolonge
donc à la catégorie dérivée non bornée $\DC(\Modtop{\Lambda}{\wt\plt_{\rm
    et}})$, ce  
qui nous permet de définir le complexe $R\Gamma_{c}(\PC_{\rm
  LT}^{\rm ca},M\otimes^{L}_{\HC}\xi_{n,!}(\Lambda))$ pour un
 $\HC$-module à droite $M$. Lorsque $M$ est libre de type fini sur
 $\HC$, le morphisme canonique
 \begin{equation}
 M\otimes^{L}_{\HC} R\Gamma_{c}(\PC_{\rm LT}^{\rm ca}, \xi_{n,!}(\Lambda))
\To{} R\Gamma_{c}(\PC_{\rm LT}^{\rm ca},M\otimes^{L}_{\HC}\xi_{n,!}(\Lambda))\label{comptens}
\end{equation}
est clairement un isomorphisme. Lorsque $M$ est simplement de type fini,
la noethériannité de $\HC$ permet d'en trouver une résolution (infinie) par des $\HC$-modules
libres de type fini. Un argument de double complexe montre alors 
que \emph{le morphisme (\ref{comptens}) est encore un isomorphisme}.

Maintenant, les fibres de $\xi_{n,!}(\Lambda)$, étant isomorphes à
$\CC_{c}(H\ba G,\Lambda)$,  sont \emph{plates sur $\HC$.} En effet, on a
\begin{fac} %(Voir \cite[Fait 3.5.8]{lt} ou \cite[Thm 3.1]{MS1}) 
  La sous-catégorie de $\Mo{\Lambda}{G}$ formée des objets
  engendrés par leurs $H$-inva\-riants est une sous-catégorie facteur direct.
\end{fac}
Voir \cite[Fait 3.5.8]{lt}, ou appliquer \cite[Thm 3.1]{MS1} au système d'idempotents
$(e_{x})_{x\in BT^{\circ}}$ associé aux sous-groupes de congruences de
niveau $n$, comme dans \cite[Lem 2.6]{MS1}. 
Dans cette situation, le foncteur $V\mapsto
V^{H}=\Hom_{G}(\CC_{c}(H\ba G,\Lambda),V)$ induit une équivalence de
cette sous-catégorie sur la catégorie des $\HC$-modules à droite, et
son adjoint à gauche 
\begin{equation}
M\mapsto M\otimes_{\HC}\CC_{c}(H\ba G,\Lambda)\label{equiv}
\end{equation}
en est une équivalence inverse, donc est exact.
%  Voir aussi \cite[Fait
% 3.5.8]{lt} pour un autre point de vue sur ces équivalences.
% est exact, puisqu'il induit même une équivalence  entre la catégorie des $\HC$-modules à droite
% et celle des $\Lambda$-représentations lisses de $G$ qui sont
% engendrées par leurs invariants.

On a donc simplement
$M\otimes^{L}_{\HC}\xi_{n,!}(\Lambda)=M\otimes_{\HC}\xi_{n,!}(\Lambda)$
et on obtient finalement un isomorphisme
$$M\otimes^{L}_{\HC}R\Gamma_{c}^{H} \simto R\Gamma_{c}(\plt^{\rm ca},M\otimes_{\HC} \xi_{n,!}(\Lambda)).$$
Mais le complexe de droite est le complexe de cohomologie d'un faisceau,
donc n'a pas de cohomologie en degrés négatifs. On a donc bien
$\HC^{-i}(M\otimes^{L}_{\HC}R\Gamma_{c}^{H})=0$ pour $i>0$.

\emph{Cinquième étape.} A ce point, nous avons prouvé que
$R\Gamma_{c}^{H}$ est parfait d'amplitude contenue dans
$[0,2d-2]$. Le fait que l'amplitude est exactement
$[0,2d-2]$ découle du lemme ci-dessous.
On y note $\mathbf{1}_{\Lambda}$ le caractère ``trivial'' de $\HC$ à
valeurs dans $\Lambda$, qui envoie la fonction caractéristique de
$HgH$ sur $\#(H\ba HgH)\in q^{\ZM}$. C'est l'image de la
$\Lambda$-représentation triviale de $G$ par le foncteur $V\mapsto V^{H}$.

\begin{lem}
 On a un isomorphisme canonique
$$ \mathbf{1}_{\Lambda}\otimes^{L}_{\HC} R\Gamma_{c}(\mlt^{\rm
  ca},\Lambda)^{H}\simto R\Gamma_{c}(\plt^{\rm
  ca},\Lambda)=\bigoplus_{k=0}^{2d-2} \Lambda(-k)[-2k].$$  
\end{lem}
\begin{proof}
  D'après la quatrième étape, le terme de gauche s'identifie à 
$R\Gamma_{c}(\plt^{\rm
  ca},\mathbf{1}_{\Lambda}\otimes_{\HC}\xi_{n,!}(\Lambda))$.
Comme le foncteur (\ref{equiv}) est une équivalence inverse de
$V\mapsto V^{H}$, on a 
$\mathbf{1}_{\Lambda}\otimes_{\HC}\CC_{c}(H\ba G,\Lambda)\simeq \Lambda$,
et le faisceau $\mathbf{1}_{\Lambda}\otimes_{\HC}\xi_{n,!}(\Lambda)$
est donc un
système local en $\Lambda$-modules libres de rang $1$ sur $\plt^{\rm
  nr}$. D'après \cite[Lemma 7.3]{DeJong}, un tel système local est
trivial sur  $\plt^{ \rm ca}$.
\end{proof}

\alin{Preuve dans le cas $\ell$-adique} 
Il n'y a rien à modifier aux deux premières étapes, mais
l'isomorphisme (\ref{isom}) n'est plus correct.
Pour le corriger, nous devons
rappeler la construction du complexe $R\Gamma_{c}$.
On utilise librement les notations de
\cite[3.3]{lt}. On note en particulier $\Lambda_{\bullet}$ le
pro-anneau $(\Lambda/\ell^{n}\Lambda)_{n\in\NM}$, 
$\Modtop{\Lambda_{\bullet}}{\wt{X_{\rm et}}}$ la catégorie
abélienne des faisceaux étales de $\Lambda_{\bullet}$-modules sur
un espace analytique $X$, et $\DC^{+}(\Mod_{\Lambda_{\bullet}}(\wt{X_{\rm
  et}} ))$ sa catégorie dérivée. On dispose alors du foncteur limite projective
$$\limproj^{\wt{X_{et}}} :\,\Modtop{\Lambda_{\bullet}}{\wt{X_{\rm et}}}\To{}\Modtop{\Lambda}{\wt{X_{\rm et}}}.$$
Par \cite[(3.5.10)]{lt}, 
$R\Gamma_{c}(\mlt^{\rm ca},\Lambda)^{H}$ est l'évaluation en le pro-faisceau
constant $\Lambda_{\bullet}$ sur $\plt^{\rm nr}$ du foncteur dérivé du foncteur 
$$\FC_{\bullet} \mapsto \Gamma_{c}\left(\mltn^{\rm
    ca},\limproj^{\wt{\MC_{\rm LT, n, et}}} \left(\xi_{n}^{*}(\FC_{\bullet})\right)\right)
= \Gamma_{c}\left(\plt^{\rm
    ca}, \xi_{n,!}\xi_{n}^{*}\left(\limproj^{\wt{\PC^{\rm nr}_{\rm LT,et}}}(\FC_{\bullet})\right)\right).$$
L'égalité utilise l'isomorphisme
$\limproj^{\wt{\MC_{\rm LT, n, et}}}\circ \xi_{n}^{*}
\simto \xi_{n}^{*}\circ \limproj^{\wt{\PC^{\rm nr}_{\rm LT,et}}}$, qui
lui-même vient du fait que $\xi_{n}$ est étale. Maintenant, la formule
de projection nous donne un isomorphisme de faiceaux sur $\plt^{\rm nr}$
$$ \xi_{n,!}\xi_{n}^{*}\left(\limproj^{\wt{\PC^{\rm nr}_{\rm
        LT,et}}}(\FC_{\bullet})\right) \simto
\xi_{n,!}(\Lambda)\otimes_{\Lambda}\left(\limproj^{\wt{\PC^{\rm nr}_{\rm
        LT,et}}}(\FC_{\bullet})\right)$$
dans lequel $\xi_{n,!}(\Lambda)$ désigne l'image directe à supports
propres du faisceau constant de fibres $\Lambda$ sur $\mltn$.
Appliquant ceci à une résolution injective de $\Lambda_{\bullet}$, on
obtient l'isomorphisme suivant dans $\DC^{b}(\Mod(\Lambda))$, analogue de (\ref{isom}) :
\begin{equation}
R\Gamma_{c}^{H}=R\Gamma_{c}(\mltn^{\rm ca},\Lambda) \simeq R\Gamma_{c}\left(\plt^{\rm ca}, \xi_{n,!}(\Lambda)\otimes^{L}_{\Lambda}(R\limproj^{\wt{\PC^{\rm nr}_{\rm
        LT,et}}}(\Lambda_{\bullet}))\right).\label{isom2} 
\end{equation}
Comme dans le cas de torsion (lemme ci-dessus), cet isomorphisme vit
par construction dans
$\DC^{b}(\Mod(\HC))$. Pour tout $\HC$-module à droite de type
fini $M$, l'isomorphisme (\ref{comptens}) devient
\begin{equation}
 M\otimes^{L}_{\HC}R\Gamma_{c}^{H} \simto R\Gamma_{c}\left(\plt^{\rm ca},
  (M\otimes^{L}_{\HC} \xi_{n,!}(\Lambda))\otimes^{L}_{\Lambda}(R\limproj^{\wt{\PC^{\rm nr}_{\rm
        LT,et}}}(\Lambda_{\bullet}))\right).\label{comptens2}
\end{equation}
Les fibres de $\xi_{n,!}(\Lambda)$ étant des $\HC$-modules plats, on
peut supprimer le premier $L$ dans le terme de droite. 
Le complexe $R\limproj^{\wt{\PC^{\rm nr}_{\rm  LT,et}}}(\Lambda_{\bullet}))$
de $\DC^{+}(\Mod_{\Lambda}(\wt{\PC_{\rm LT,et}^{\rm nr}}))$ est concentré en
degrés $\geq 0$, et son $\HC^{0}$ est le faisceau constant
$\Lambda$, qui est donc sans $\ell$-torsion.
Il s'ensuit que
le complexe $(M\otimes_{\HC}
\xi_{n,!}(\Lambda))\otimes^{L}_{\Lambda}(R\limproj^{\wt{\PC^{\rm
      nr}_{\rm  LT,et}}}(\Lambda_{\bullet})) \in
\DC^{+}(\Mod_{\Lambda}(\wt{\PC_{\rm LT,et}^{\rm nr}}))$ est concentré en
degrés $\geq 0$, et il en est donc de même du complexe
$M\otimes^{L}_{\HC}R\Gamma_{c}^{H}$.

\subsection{Perfection dans $\DC(D^{\times})$}

\begin{prop} \label{propparfaitD}
Soit $\Lambda$ une $\Zl$-algèbre finie et $H$ un sous-groupe de
congruences de $\GL_{d}(\OC)$.
  Le complexe $R\Gamma_{c}(\mlt^{\rm ca},\Lambda)^{H}$ est 
parfait d'amplitude $[d-1,2d-2]$ dans $\DC^{b}(\Mo{\Lambda}{D^{\times}})$.
\end{prop}
\begin{proof}
La preuve de cette proposition 
repose sur le théorème de Faltings-Fargues \cite[IV.13.2]{FarFal} qui
affirme l'existence d'un isomorphisme naturel 
$$ R\Gamma_{c}(\mlt^{\rm ca},\Lambda)\simeq R\Gamma_{c}(\mdr^{\rm
  ca},\Lambda) \,\,\hbox{ dans }\, \DC^{b}(\Rep^{\infty, c}_{\Lambda}(GD\times W_{K}))$$
où le terme de droite désigne le complexe de cohomologie équivariant
de la ``tour de Drinfeld''. Nous allons donc travailler du côté
``Drinfeld'' en renvoyant à \cite[3.3]{lt} pour les
notations employées.

On sait que la cohomologie de $R\Gamma_{c}(\mlt^{\rm ca},\Lambda)^{H}$ est
de type fini sur le centre de $G$. Puisque le centre de
$D^{\times}$ agit comme celui de $G$, on en conclut que la cohomologie
de $R\Gamma_{c}(\mlt^{\rm ca},\Lambda)$ est de type fini sur $D^{\times}$.
 Fixons alors un sous-groupe de congruences $J_{n}:=1+\varpi^{n}\OC_{D}$ de
$D^{\times}$ agissant trivialement sur ces espaces de cohomologie.
On sait d'après \cite[3.5.6]{lt} que $R\Gamma_{c}(\mdr^{\rm
  ca},\Lambda)^{J_{n}}\simeq R\Gamma_{c}(\mdrn^{\rm ca},\Lambda)$.
Soit $\xi_{n}:\, \mdrn\To{}\pdr=\Omega^{d-1}_{\knr}$ le morphisme de
périodes de Drinfeld et Rapoport-Zink. Comme dans la preuve de la
proposition \ref{propparfait}, le faisceau $\xi_{n,!}(\Lambda)$ est
naturellement un faisceau $G$-équivariant de $\Lambda[D^{\times}/J_{n}]$-modules et on a un isomorphisme
$$ R\Gamma_{c}(\mdr^{\rm 
  ca},\Lambda)^{J_{n}}\simeq R\Gamma_{c}(\pdr^{\rm ca},\xi_{n,!}(\Lambda))$$
dans $\DC^{b}(\Rep_{\Lambda}(D^{\times}/J_{n}))$, et même dans
$\DC^{b}(\Mo{\Lambda}{G\times D^{\times}/J_{n}})$. Les fibres de $\xi_{n,!}(\Lambda)$ sont
ici encore des $\Lambda[D^{\times}/J_{n}]$-modules plats, et même libres de rang $1$. 
On a donc pour tout $\Lambda[D^{\times}/J_{n}]$-module $M$ un isomorphisme 
$$M\otimes^{L}_{\Lambda[D^{\times}/J_{n}]} R\Gamma_{c}(\mdr^{\rm
  ca},\Lambda)^{J_{n}}\simeq R\Gamma_{c}(\pdr^{\rm
  ca},M\otimes_{\Lambda[D^{\times}/J_{n}]}\xi_{n,!}(\Lambda))$$
dans $\DC^{b}(\Mo{\Lambda}{G})$, d'où l'on déduit un isomorphisme
$$M\otimes^{L}_{\Lambda[D^{\times}/J_{n}]} R\Gamma_{c}(\mdr^{\rm
  ca},\Lambda)^{H}\simeq R\Gamma_{c}(\pdr^{\rm
  ca},M\otimes_{\Lambda[D^{\times}/J_{n}]}\xi_{n,!}(\Lambda))^{H}$$
dans $\DC^{b}(\Mod(\Lambda))$.
On peut
donc suivre le même raisonnement que dans la preuve de la proposition
\ref{propparfait} pour en conclure que \emph{le complexe
  $R\Gamma_{c}(\mlt^{\rm ca},\Zl)^{H}$ est 
parfait d'amplitude contenue dans $[0,2d-2]$ dans 
$\DC^{b}(\Rep_{\Zl}(D^{\times}/J_{n}))$.}

Il reste maintenant à montrer que l'amplitude parfaite est en fait la
meilleure possible, égale à l'amplitude cohomologique $[d-1,2d-2]$. Pour cela, notons
$D^{0}:=\OC_{D}^{\times}$ le groupe des entiers inversibles de $D$. On
se souvient \cite[3.5.2]{lt} que
$$R\Gamma_{c}(\mdr^{\rm
  ca},\Lambda)=\cind{D^{0}}{D}{R\Gamma_{c}(\mdr^{(0),\rm
    ca},\Lambda)},$$
où $R\Gamma_{c}(\mdr^{(0),\rm
    ca},\Lambda) \in \DC^{b}(\Rep^{\infty}_{\Lambda}(D^{0}))$ est le complexe de
  cohomologie de la tour $(\mltn^{(0)})_{n\in\NM}$ formée de
  revêtements étales finis de $\pdr$. De même 
$$R\Gamma_{c}(\mdr^{\rm
  ca},\Lambda)^{H}= \Lambda[D^{\times}/J_{n}]\otimes_{\Lambda[D^{0}/J_{n}]} (R\Gamma_{c}(\mdr^{(0),\rm
    ca},\Lambda)^{H}).$$
Notons que la même preuve que ci-dessus montre que $R\Gamma_{c}(\mdr^{(0),\rm
    ca},\Lambda)^{H}$ est un complexe parfait de
  $\Lambda[D^{0}/J_{n}]$-modules. Il nous suffira donc de montrer que son
  amplitude parfaite est bien $[d-1,2d-2]$.
On peut aussi se restreindre au cas $\Lambda=\Zl$ puisque les autres
s'en déduisent par changement de coefficients. 

Faisons ici une petite digression. Soit $A$ un groupe et $\CC$ un complexe parfait de
$\Zl[A]$-modules d'amplitude cohomologique $[a,b]$ et d'amplitude parfaite $[\alpha,\beta]$. 
Par définition, on a $[a,b]\subset[\alpha,\beta]$. On voit aussi
facilement que $b=\beta$, car si $\beta>b$, la différentielle
$d_{\beta-1}:P_{\beta-1}\To{}P_{\beta}$ est surjective et donc
scindable. Supposons de plus que $A$ est \emph{fini}. Dans ce cas le
$\Zl$-dual d'un $\Zl[A]$-module projectif est encore un
$\Zl[A]$-module projectif. Ainsi le complexe dual $\CC^{\vee}$ est
parfait d'amplitude parfaite $[-\beta,-\alpha]$. De plus, les suites
exactes
$$\Ext^{1}_{\Zl}(\HC^{-i+1}(\CC),\Zl) \To{}\HC^{i}(\CC^{\vee})\To{}\Hom_{\Zl}(\HC^{-i}(\CC),\Zl)
 $$ montrent que $\HC^{-i}(\CC^{\vee})$ est nul
pour $i>-a+1$, et même pour $i=-a+1$ si $\HC^{a}(\CC)$ n'a pas de
$\ell$-torsion. Dans ce dernier cas, on a alors $\alpha=a$,
\emph{i.e.} l'amplitude parfaite coïncide avec l'amplitude
cohomologique.

Ceci s'applique au complexe $\CC=R\Gamma_{c}(\mdr^{(0),\rm
    ca},\Zl)^{H}$ et $A=D^{0}/J_{n}$, puisque l'on sait que $H^{d-1}_{c}(\mdr^{(0),\rm
    ca},\Zl)$ est sans torsion (vu que $H^{d-2}_{c}(\mdr^{(0),\rm
    ca},\Fl)$ est nul, par \cite[Cor 6.2]{Bic4}).
\end{proof}

\begin{rema}  \label{remparfait}
 Contrairement à ce que pourraient laisser penser
 les propositions \ref{propparfait} et \ref{propparfaitD}, le
 complexe $R\Gamma_{c}(\mlt^{\rm  ca},\Lambda)$ n'est pas parfait en
 tant que complexe de $\DC^{b}_{\Lambda}(G\times D^{\times})$ ou $\DC^{b}\Rep^{\infty}_{\Lambda}(GD)$. Pour le voir, faisons
 $\Lambda=\Fl$ et considérons le complexe
$$\CC:=R\Gamma_{c}(\mlt^{\rm ca},\Fl)\otimes^{L}_{\GL_{d}(\OC)K^{\times}}
\Fl \in \DC^{b}\Rep^{\infty}_{\Fl}(D^{\times}/K^{\times}).$$ Si ce complexe n'est pas
parfait dans $\DC^{b}\Rep^{\infty}_{\Fl}(D^{\times})$,
resp. $\DC^{b}\Rep^{\infty}_{\Fl}(D^{\times}/K^{\times})$, alors le complexe
$R\Gamma_{c}(\mlt^{\rm  ca},\Fl)$ ne l'est pas dans
$\DC^{b}\Rep^{\infty}_{\Fl}(G\times D^{\times})$, resp. $\DC^{b}\Rep^{\infty}_{\Fl}(GD)$.
Comme dans \cite[3.5.9]{lt} et \cite[3.5.5]{lt}, on a
$$ \CC \simeq R\Gamma_{c}(\mlto^{\rm
  ca},\Fl)\otimes^{L}_{\varpi^{\ZM}}\Fl
=\cind{\OC_{D}^{\times}\varpi^{\ZM}}{D^{\times}}{R\Gamma_{c}(\mlto^{(0),
    \rm
  ca},\Fl)}.$$
Or, $\mlto^{(0)}$ est une boule unité ouverte de dimension $d-1$
donc 
$$ \CC \simeq
\cind{\OC_{D}^{\times}/\OC^{\times}}{D^{\times}/K^{\times}}{\Fl}[2-2d].$$
Supposons maintenant que l'ordre de $q$ dans $\Fl$ soit exactement
$d$. Dans ce cas, la $\Fl$-représentation triviale du groupe
$\FM_{q^{d}}^{\times}$ est de dimension cohomologique infinie, et il
en est donc de même de la $\Fl$-représentation lisse triviale de
$\OC_{D}^{\times}/K^{\times}$, ainsi que de la $\Fl$-représentation 
$\cind{\OC_{D}^{\times}/\OC^{\times}}{D^{\times}/K^{\times}}{\Fl}$,
qu'elle soit vue comme représentation de $D^{\times}$ ou de $D^{\times}/K^{\times}$.
Le complexe $\CC$ n'est donc pas parfait dans 
$\DC^{b}\Rep^{\infty}_{\Fl}(D^{\times})$, ni dans $\DC^{b}\Rep^{\infty}_{\Fl}(D^{\times}/K^{\times})$.

On en déduit, sous la même hypothèse de congruence,  que le
complexe de cohomologie du demi-espace de Drinfeld
$R\Gamma_{c}(\mdro^{(0),\rm ca},\Fl)$ n'est pas parfait dans
$\DC^{b}\Rep^{\infty}_{\Fl}(G)$ ni dans $\DC^{b}\Rep^{\infty}_{\Fl}(G/K^{\times})$. Comme ce
complexe est scindé (par les poids), cela signifie qu'au moins un des
espaces de cohomologie est de dimension cohomologique infinie. 

\end{rema}

\section{Scindages de catégories, enveloppes projectives, et déformations}
\setcounter{subsubsection}{0}

Dans cet appendice, nous rassemblons les résultats de pure théorie des
représentations utilisés dans le corps du texte.
Nous utiliserons le lemme suivant sur une manière générale de scinder
une catégorie abélienne.
\begin{lemme} \label{lemmescind}
  Soit $\CC$ une catégorie abélienne avec produits et coproduits, et 
% possédant suffisament d'injectifs et de projectifs.
$S$ un ensemble d'objets simples de $\CC$. Notons $\CC_{S}$,
  resp. $\CC^{S}$ la
  sous-catégorie pleine de $\CC$ formée des objets dont tous les
  sous-quotients simples sont dans $S$, resp. aucun sous-quotient
  simple n'est dans $S$.
  \begin{enumerate}
  \item Si $\CC_{S}$ contient un objet projectif $P_{S}$ de $\CC$ tel
    que $\Hom_{\CC}(P_{S},s)\neq 0$ pour tout $s\in S$, 
alors tout objet $V$ de $\CC$ possède une unique filtration
    $V_{S}\injo V \twoheadrightarrow V^{S}$ avec $V_{S}\in\CC_{S}$ et $V^{S}\in\CC^{S}$.
  \item Si de plus $\CC_{S}$ contient un objet injectif $I_{S}$ de $\CC$ tel
    que $\Hom_{\CC}(s,I_{S})\neq 0$ pour tout $s\in S$,
alors la filtration est (canoniquement) scindée. 
Le foncteur somme directe $\CC_{S}\times \CC^{S}\To{}\CC$
 est donc une équivalence de catégories.
  \end{enumerate}
\end{lemme}
\begin{proof}
$i)$
Comme la somme de deux sous-objets de $V$ apartenant à $\CC_{S}$
appartient encore à $\CC_{S}$, on peut définir
le plus grand sous-objet de $V$ appartenant à $\CC_{S}$. Notons-le
$V_{S}$. Supposons que $V/V_{S}$ n'est pas dans $\CC^{S}$. Alors $V$ possède un
sous-objet $W$ contenant $V_{S}$ tel que $W/V_{S}$ se surjecte sur un
objet $s\in S$.
Mais alors on en déduit un morphisme $P_{S}\To{}W$ dont l'image n'est
pas dans $V_{S}$ : contradiction.

$ii)$ Dualement, comme le quotient de $V$ par l'intersection de deux
sous-objets de quotients dans $\CC_{S}$ est encore dans $\CC_{S}$, on
peut définir le
plus grand quotient de $V$ appartenant à $\CC_{S}$. Notons-le $V_{S}^{q}$.
Toujours dualement, l'existence de $I_{S}$ assure que le noyau de
$V\twoheadrightarrow V^{q}_{S}$ est dans $\CC^{S}$. Ceci entraine
alors que la composée $V_{S}\To{} V^{q}_{S}$ est un isomorphisme.

Le foncteur du $ii)$ est clairement fidèle. Il
est aussi plein car $\Hom(A,B)=\Hom(B,A)=0$ si $A\in\CC_{S}$ et $B\in\CC^{S}$.
On vient de montrer qu'il est essentiellement surjectif.
\end{proof}

Nous aurons aussi besoin du fait suivant, très classique. 
\begin{fact} \label{faitequiv}
  Soit $P$  un générateur projectif compact (=de type fini) d'une catégorie abélienne $\CC$
essentiellement petite avec limites inductives exactes,  et soit
$\ZG:=\endo{\CC}{P}^{\rm opp}$ l'anneau opposé de son commutant. Alors les foncteurs
$ V\mapsto \Hom_{\CC}(P,V)$  et $M\mapsto P\otimes_{\ZG} M$ sont des
équivalences de catégories entre $\CC$ et $\Mod(\ZG)$ ``inverses'' l'une de l'autre.
\end{fact}

\subsection{Représentations supercuspidales et scindages de $\Rep(G)$}
\label{secscinG}

\ali Soit $\pi$ une $\oFl$-représentation irréductible
\emph{supercuspidale} de $G$ (au sens habituel de Vignéras rappelé
au-dessus du théorème 1), et soit $\pi^{0}$ un sous-quotient irréductible
de $\pi_{|G^{0}}$. On rappelle que $\bZl$ désigne l'extension non
ramifiée maximale de $\Zl$ dans $\oZl$.
On  désigne par $\varpi$ un élément de $K$ de valuation $v>0$, et
par $\varpi^{\ZM}$ le sous-groupe discret qu'il engendre, que
l'on considère parfois comme un sous-groupe central de $G$.
Sans perte de généralité pour ce qui suit, nous supposerons que
le caractère central de $\pi$ est trivial sur $\varpi$. 
On s'intéresse aux sous-catégories pleines suivantes :
\begin{itemize}
\item $\CC_{\pi}\subset \Mo{\bZl}{G}$ formée des 
  objets dont tous les $\Zlnr G^{0}$-sous-quotients irréductibles sont
  isomorphes à un sous-quotient de $\pi_{|G^{0}}$.
\item $\CC_{\pi}^{\varpi}\subset \Mo{\bZl}{G/\varpi^{\ZM}}$ formée des $\bZl$-représentations de
  $G/\varpi^{\ZM}$ et dont tous les $\Zlnr G$-sous-quotients irréductibles sont
  isomorphes   à  $\pi$.
\item $\CC_{\pi^{0}}^{0}\subset \Mo{\bZl}{G^{0}}$ formée des $\bZl$-représentations de
  $G^{0}$ dont tous les $\Zlnr G^{0}$-sous-quotients  sont  isomorphes à  $\pi^{0}$.
\end{itemize}
Nous voulons montrer qu'elles sont facteurs directs, en exhiber des
progénérateurs convenables, et calculer les commutants de ces générateurs.
Pour cela, choisissons un type simple  $(J^{\circ},\lambda)$
contenu dans $\pi^{0}$, \emph{cf.} \cite[III.5.10.i)]{Vig} où le type est
dit ``minimal-maximal''. Ainsi, $J^{\circ}$ est un sous-groupe ouvert
compact de $G$ et $\lambda$ une $\oFl$-représentation irréductible de
$J^{\circ}$, et l'on a d'après \cite[III.5.3]{Vig}
$$\pi^{0}\simeq \cind{J^{\circ}}{G^{0}}{\lambda}.$$
Soit $P_{\lambda}$ une enveloppe projective de $\lambda$ dans
$\Mo{\bZl}{J^{\circ}}$. 
Dans les énoncés qui suivent, l'entier $f$ désigne la longueur de $\pi_{|G^{0}}$.

\begin{prop}\label{propscindageG0}
Posons ${P}_{\pi^{0}}:=\cind{J^{\circ}}{G^{0}}{{P_{\lambda}}}$ et ${P}_{\pi}:=\cind{J^{\circ}}{G}{{P_{\lambda}}}$.
\begin{enumerate}\item 
  \begin{enumerate}\item
    La sous-catégorie $\CC^{0}_{\pi^{0}}$ est facteur direct de
    $\Mo{\bZl}{G^{0}}$, pro-engendrée par ${P}_{\pi^{0}}$.
  \item Le commutant $\ZG_{\pi^{0}}^{\rm opp}:=\endo{\bZl
      G^{0}}{{P_{\pi^{0}}}}$ est isomorphe
    à l'algèbre $\bZl[{\rm Syl}_{\ell}(\FM_{q^{f}}^{\times})]$ du
    $\ell$-Sylow du groupe $\FM_{q^{f}}^{\times}$.
  \end{enumerate}
\item 
  \begin{enumerate}
     \item La sous-catégorie $\CC_{\pi}$ est facteur direct dans
    $\Mo{\bZl}{G}$,  pro-engendrée par ${P_{\pi}}$.
  \item Le commutant $\ZG_{\pi}^{\rm opp}:=\endo{\bZl G}{{P_{\pi}}}$ est isomorphe à l'algèbre
    $\bZl[{\rm Syl}_{\ell}(\FM_{q^{f}}^{\times})\times\ZM]$.
  \end{enumerate}
\end{enumerate}

\end{prop}
\begin{proof}
$i)(a)$ Le point clef est
\cite[III.4.28]{Vig} qui nous dit que, puisque $\pi$ est
\emph{super}cuspidale, 
tous les sous-quotients irréductibles de
$P_{\lambda}\otimes_{\Zlnr}\oFl$, et donc aussi ceux de ${P}_{\lambda}$, sont isomorphes à
$\lambda$ (en fait, nous donnerons une preuve alternative de ce fait (moins
élémentaire) lorsque nous démontrerons le point ii)). 
Il s'ensuit que $P_{\pi^{0}}\otimes\oFl$ est une représentation
de longueur finie de $G^{0}$, et que tous ses sous-quotients
irréductibles, ainsi que ceux de ${P_{\pi^{0}}}$,
 sont isomorphes à $\pi^{0}$. En particulier, ${P_{\pi^{0}}}$ est bien un objet de
$\CC^{0}_{\pi^{0}}$. C'est même un objet projectif, puisqu'induit à supports compacts
d'un objet projectif. 

Il reste à voir que  $\CC^{0}_{\pi^{0}}$ est bien facteur direct de
$\Mo{\bZl}{G^{0}}$. Nous avons le nécessaire pour appliquer le i)
du lemme \ref{lemmescind} avec $\CC=\Mo{\bZl}{G^{0}}$,
$S=\{\pi^{0}\}$ et $P_{S}= {P_{\pi^{0}}}$. 
Pour appliquer le ii) de ce lemme,  remarquons que la 
$\bZl$-contragédiente $P_{\pi^{0}}^{\vee}$ de $P_{\pi^{0}}$
est encore un objet projectif de $\Mo{\bZl}{G^{0}}$ (\emph{cf} lemme
\ref{lemprojcusp}) et que tous ses 
sous-quotients irréductibles sont isomorphes à $(\pi^{0})^{\vee}$.
Posons alors
$I_{\pi^{0}}:=\Hom_{\bZl}(P_{\pi^{0}}^{\vee},\Ql^{\rm nr}/\bZl)$.
C'est un objet injectif (pas de type fini) de $\Mo{\bZl}{G^{0}}$ qui
appartient $\CC_{\pi^{0}}^{0}$ et qui contient $\pi_{0}$. On peut donc
appliquer le lemme \ref{lemmescind} et conclure.

$i)(b)$ Pour calculer
  l'anneau $\endo{G^{0}}{{P_{\pi^{0}}}}$, on rappelle deux propriétés du type
  $(J^{\circ},\lambda)$.
  \begin{enumerate}
  \item[a)] L'ensemble d'entrelacement de $\lambda$ dans $G^{0}$ est égal
    à $J^{\circ}$. En conséquence,  le  morphisme canonique
  $\endo{J^{\circ}}{{P_{\lambda}}}\To{}\endo{G^{0}}{{P_{\pi^{0}}}}$
est un isomorphisme.
\item[b)] Le quotient de $J^{\circ}$ par son pro-$p$-radical $J^{1}$ est
isomorphe à un groupe du type $\GL_{d'}(\FM_{q^{f'}})$ et $\lambda$ est de la forme
$\kappa\otimes \tau$ avec $\kappa$ une représentation dont la
restriction à $J^{1}$ est irréductible et admettant un unique
relèvement $\wt\kappa$, et $\tau$ une représentation supercuspidale du
quotient $\GL_{d'}(\FM_{q^{f'}})$. En conséquence, ${P_{\lambda}}$
est de la forme ${\wt\kappa}\otimes {P_{\tau}}$ et on a un isomorphisme
$\endo{\GL_{d'}(\FM_{q^{f'}})}{{P_{\tau}}}\simto \endo{J^{\circ}}{{P_{\lambda}}}$.
  \end{enumerate}
Reste alors à calculer
$\endo{\GL_{d'}(\FM_{q^{f'}})}{{P_{\tau}}}$. Après son étude des
relèvements possibles de $\tau$, Vignéras a observé dans
\cite[III.2.9]{Vig} que d'après les propriétés du triangle de
Cartan-Brauer, $P_{\tau}\otimes\oQl$ est isomorphe à la somme directe $\bigoplus_{\wt\tau}
\wt\tau$ des $\oQl$-relèvements de $\tau$, chacun apparaissant avec
multiplicité $1$. Elle a compté $\ell^{v_{\ell}(q^{f'd'}-1)}=|{\rm Syl}_{\ell}(\FM_{q^{f'd'}}^{\times})|$ tels
relèvements, et l'anneau qui nous intéresse est donc un ordre local
dans l'algèbre $\oQl^{\ell^{v_{\ell}(q^{f'd'}-1)}}$. L'auteur n'ayant pas su
identifier cet ordre par des moyens élémentaires, 
nous utiliserons la cohomologie de la variété de Deligne-Lusztig $Y_{d',f'}$ associée à
l'élément de Coxeter de $\GL_{d'}(\FM_{q^{f'}})$, et qui est munie
d'une action de $\GL_{d'}(\FM_{q^{f'}})\times
\FM_{q^{f'd'}}^{\times}$. 
Notons $\CC_{\tau}$ la sous-catégorie facteur direct de
$\Mo{\Zlnr}{\GL_{d'}(\FM_{q^{f'}})}$ formée des représentations dont
tous les sous-quotients irréductibles sont isomorphes à $\tau$, qui est donc
pro-engendrée par $P_{\tau}$.
On sait que le complexe $R\Gamma_{c}(Y_{d',f'}^{\rm ca},\bZl)$ est parfait
dans $\DC^{b}(\Rep_{\bZl}( \GL_{d'}(\FM_{q^{f'}})))$,
donc son facteur direct $R\Gamma_{c}(Y_{d',f'}^{\rm
  ca},\bZl)_{\CC_{\tau}}$ est aussi parfait. Comme 
 dans la preuve de la proposition
\ref{progeneG}, ce facteur direct est concentré en degré $d'-1$ : cela
découle par exemple du théorème 3.10 de \cite{BonRou2} car $P_{\tau}$ est
facteur direct du module de Gelfand-Graev. Il s'ensuit que 
$H^{d'-1}_{c}(Y_{d',f'}^{\rm
  ca},\bZl)_{\CC_{\tau}}$ est projectif, et donc est un progénérateur
de  $\CC_{\tau}$. 
 Cet espace se décompose
encore selon l'action de $\FM_{q^{f'd'}}^{\times}$ en 
$$H^{d'-1}_{c}(Y_{d',f'}^{\rm ca},\bZl)_{\CC_{\tau}}=\bigoplus_{\theta}H^{d'-1}_{c}(Y_{d',f'}^{\rm
  ca},\bZl)_{\CC_{\tau},\CC_{\theta}}$$
 où $\theta$ décrit l'ensemble des $\oFl$-caractères de $\FM_{q^{f'd'}}^{\times}$
 et $\CC_{\theta}$ désigne la sous-catégorie 
des $\bZl$-représentations de $\FM_{q^{f'd'}}^{\times}$ dont tous les
sous-quotients sont isomorphes à $\theta$.
La théorie de Deligne-Lusztig nous dit qu'il y a au moins un 
facteur comme ci-dessus non nul et que de plus on a
$$H^{d'-1}_{c}(Y_{d',f'}^{\rm ca},\bZl)_{\CC_{\tau},\CC_{\theta}}\otimes\oQl =
\bigoplus_{\wt\tau}\wt\tau\otimes \wt\theta
$$
où $\wt\tau$ décrit l'ensemble des relèvements de $\tau$ et $\wt\theta$ est
l'unique caractère relevant $\theta$ et associé à $\wt\tau$ par la
``correspondance de Deligne-Lusztig''. En
particulier, $H^{d'-1}_{c}(Y_{d',f'}^{\rm
  ca},\bZl)_{\CC_{\tau},\CC_{\theta}}$ est une enveloppe projective de $\tau$.
Maintenant, le complexe $R\Gamma_{c}(Y_{d',f'}^{\rm ca},\bZl)$ est aussi parfait
dans $\DC^{b}(\Rep_{\bZl}(\FM_{q^{d'f'}}^{\times}))$,
donc le facteur direct $H^{d'-1}_{c}(Y_{d',f'}^{\rm 
  ca},\bZl)_{\CC_{\tau},\CC_{\theta}}$ est aussi projectif dans
$\CC_{\theta}$.  
Considérons alors le morphisme canonique
$$ \bZl[\FM_{q^{f'd'}}^{\times}]_{\CC_{\theta}}\simeq  \bZl[{\rm Syl}_{\ell}(\FM_{q^{f'd'}}^{\times})]
\To{} \endo{\bZl\GL_{d'}(\FM_{q^{f'}})}{H^{d'-1}_{c}(Y_{d',f'}^{\rm
    ca},\bZl)_{\CC_{\tau},\CC_{\theta}}} $$
Ce morphisme entre deux $\bZl$-algèbres libres de rang fini devient
bijectif après inversion de $\ell$ d'après la décomposition
précédente. Pour voir qu'il est lui-même bijectif, il suffit de
prouver que sa réduction
$$ \oFl[\FM_{q^{f'd'}}^{\times}]_{\CC_{\theta}}\simeq \oFl[{\rm Syl}_{\ell}(\FM_{q^{f'd'}}^{\times})]
\To{} \endo{\oFl\GL_{d'}(\FM_{q^{f'}})}{H^{d'-1}_{c}(Y_{d',f'}^{\rm
    ca},\oFl)_{\CC_{\tau},\CC_{\theta}}} $$
l'est, ou encore, par égalité des rangs, que cette réduction est
injective. Or, cela découle de la projectivité de $H^{d'-1}_{c}(Y_{d',f'}^{\rm
    ca},\oFl)_{\CC_{\tau},\CC_{\theta}}$ sur l'anneau local $\oFl[\FM_{q^{f'd'}}^{\times}]_{\CC_{\theta}}$.

Pour achever la preuve du point $i)(b)$, il reste à montrer l'identité
$f=f'd'$. 
On sait que le normalisateur $J$ de la paire $(J^{\circ},\lambda)$ est
de la forme  $J^{\circ}E^{\times}$ pour une extension $E\supset K$ de
degré $d/d'$ et degré résiduel $f'$, \emph{cf} \cite[III.5]{Vig},
et que $\lambda$ s'étend en une représentation $\lambda'$ telle
que $\cind{J}{G}{\lambda'}\simeq \pi$.
On a alors
$$\pi_{|G^{0}}=\bigoplus_{x\in J\ba G/G^{0}}
\cind{(J^{0})^{x}}{G^{0}}{\lambda^{x}}= \bigoplus_{x\in J\ba G/G^{0}}
(\pi^{0})^{x}.$$
La longueur $f$ de $\pi_{|G^{0}}$ est donc donnée par 
\begin{eqnarray*}
f=[G:G^{0}J] & = & [G:G^{0}\varpi^{\ZM}][G^{0}J:G^{0}\varpi^{\ZM}]^{-1}=
d [J:J^{0}\varpi^{\ZM}]^{-1} \\ & =  & d[E^{\times}:\OC_{E}^{\times}K^{\times}]^{-1}
= d e(E/K)^{-1}= d'f'
\end{eqnarray*}

$ii)(a)$ Les sous-quotients irréductibles de $\pi_{|G^{0}}$ sont en nombre
fini, et permutés par conjugaison sous $G$.
En appliquant la proposition précédente à chacun d'eux, on constate que la
 sous-catégorie 
$\CC_{\pi}^{0}\subset \Mo{\bZl}{G^{0}}$ formée des $\bZl$-représentations de
  $G^{0}$ dont tous les sous-quotients irréductibles sont isomorphes à
  un  sous-quotient de $\pi_{|G^{0}}$
est \emph{facteur direct} de
$\Mo{\bZl}{G^{0}}$, et pro-engendrée par $\bigoplus_{x\in G/G^{0}\varpi^{\ZM}}
{P_{\pi^{0}}}^{x} \simeq {P_{\pi}^{0}}$.

Par définition, une
$\bZl$-représentation lisse de $G$ est dans $\CC_{\pi}$ \ssi\ sa
restriction à $G^{0}$ est dans $\CC_{\pi}^{0}$. Soit alors 
$V\in\Mo{\bZl}{G}$. Par unicité, la décomposition
 $V_{|G^{0}}=V_{\CC_{\pi}^{0}}\oplus V^{\CC_{\pi}^{0}}$ est
 nécessairement $G$-invariante, avec $V_{\CC_{\pi}^{0}} \in \CC_{\pi}$
 et $V^{\CC_{\pi}^{0}} \in \CC^{\pi}$.
Inversement, l'induite
d'un objet de $\CC_{\pi}^{0}$ à $G$ appartient à $\CC_{\pi}$, et
envoie tout progénérateur de $\CC_{\pi}^{0}$ sur un progénérateur de $\CC_{\pi}$.
Or, $\cind{G^{0}}{G}{{P_{\pi}^{0}}}\simeq
{P_{\pi}}^{[G:G^{0}\varpi^{\ZM}]}$, donc ${P_{\pi}}$ est bien un
progénérateur de $\CC_{\pi}$.

$ii)(b)$ Soit $J$ le normalisateur de la paire $(J^{\circ},\lambda)$
(noté $J^{\circ}E^{\times}$ dans \cite[III.5]{Vig}). On sait que
l'ensemble d'entrelacement de $\lambda$ dans $G$ est contenu dans (et donc
égal à) $J$. Celui de $P_{\lambda}$  est 
alors aussi contenu dans $J$, et on en déduit l'isomorphisme
$$ \endo{J}{\cind{J^{\circ}}{J}{{P_{\lambda}}}}\simto \endo{G}{{P_{\pi}}}.$$
Par unicité de
l'enveloppe projective, $J$ normalise aussi $P_{\lambda}$.
 Comme le quotient $J/J^{\circ}$ est isomorphe à
$\ZM$, on peut prolonger la représentation ${P_{\lambda}}$ de
$J^{\circ}$ en une représentation ${P_{\lambda}}^{\dag}$ de $J$. On
a alors un isomorphisme naturel de $\bZl$-représentations de $J$
$$\cind{J^{\circ}}{J}{{P_{\lambda}}}\simeq{P_{\lambda}}^{\dag}\otimes
    \bZl[J/J^{\circ}]$$
où $J$ agit diagonalement sur le produit tensoriel (et par
translations  sur $\bZl[J/J^{\circ}]$). On en déduit un morphisme
d'algèbres
\begin{equation}
\endo{J}{{P_{\lambda}}^{\dag}}\otimes
    \bZl[J/J^{\circ}]\To{} \endo{J}{\cind{J^{\circ}}{J}{{P_{\lambda}}}}.\label{morph}
 \end{equation}
De plus, la réciprocité de Frobenius nous donne une décomposition
de $\bZl[J/J^{\circ}]$-modules
$$\endo{J}{\cind{J^{\circ}}{J}{{P_{\lambda}}}}\simeq\Hom_{J^{\circ}}\left({P_{\lambda}},
{P_{\lambda}}^{\dag}\otimes
\bZl[J/J^{\circ}]\right)\simeq \endo{J^{\circ}}{{P_{\lambda}}} \otimes \bZl[J/J^{\circ}],$$
et le morphisme (\ref{morph}) est induit par l'inclusion
$\endo{J}{{P_{\lambda}}^{\dag}}\subset\endo{J^{\circ}}{{P_{\lambda}}}$
et l'identité de $\bZl[J/J^{\circ}]$.
Reste donc à voir que l'inclusion
$\endo{J}{{P_{\lambda}}^{\dag}}\subset\endo{J^{\circ}}{{P_{\lambda}}}$
est une égalité. Comme il s'agit d'une inclusion de $\bZl$-modules libres
de rang fini qui est  manifestement scindée, il suffit de voir que ces
modules ont même rang sur $\bZl$, et pour cela on peut étendre les
scalaires à $\oQl$. Or on sait d'une part que ${P_{\lambda}}\otimes\oQl$ est la
somme directe sans multiplicité des relèvements de $\lambda$ à
$\oQl$, et d'autre part que chacun de ces relèvements est normalisé
par $J$. On en déduit que les deux algèbres ont même rang égal au
nombre de relèvements de $\lambda$.
\end{proof}

\begin{coro}\label{equivdefsupercusp}
  La représentation $\pi$ n'est pas sous-quotient d'une induite
  parabolique propre.
\end{coro}
\begin{proof}
  En effet, si on a $ W\subset \ip{P}{G}{U}$ avec $W \twoheadrightarrow
  \pi$, alors par projectivité on a un morphisme non nul
  $P_{\pi}\To{}W\subset \ip{P}{G}{U}$, et par adjonction le module de Jacquet
  $(P_{\pi})_{N_{P}}=\bigoplus_{x\in G/G^{0}}
  (P_{\pi^{0}}^{x})_{N_{P}}$ est non nul, 
donc il existe $x\in G$ tel que $(P_{\pi^{0}}^{x})_{N_{P}}$ est
  non nul, auquel cas $((\pi^{0})^{x})_{N_{P}}$ est aussi non nul, 
et finalement $\pi_{N_{P}}$ est non nul, ce qui est absurde.
\end{proof}

\alin{Preuve de la proposition \ref{catsupercusp}}
\label{coroscindscusp}
Soit $V\in\Mo{\Zl}{G}$. Posons $V_{\Zlnr}:=V\otimes_{\Zl}\Zlnr$ et $\Gamma:=\gal(\Zlnr/\Zl)$.
 D'après la proposition \ref{propscindageG0}, on a une décomposition
$$ V_{\Zlnr} = (V_{\Zlnr})_{\rm sc} \oplus
(V_{\Zlnr})', \,\hbox{ avec }
(V_{\Zlnr})_{\rm sc}:=\bigoplus_{\pi\in {\rm
    Scusp}_{\oFl}(G)/\sim} (V_{\Zlnr})_{\CC_{\pi}}$$
et $(V_{\Zlnr})'$ n'a aucun $\Zlnr G^{0}$-sous-quotient isomorphe à un
$\Zlnr G^{0}$-sous-quotient d'une
$\oFl$-représentation supercuspidale.
Comme une $\oFl$-représentation irréductible est supercuspidale \ssi\
ses  conjuguées sous Galois le sont, on constate que la décomposition
ci-dessus est stables par $\Gamma$,
donc provient d'une décomposition $V=V_{\scusp}\oplus V'$, évidemment
fonctorielle en $V$.
Par construction, $V'$ n'a aucun sous-quotient supercuspidal au sens
de \ref{defsupercusp}, et par le corollaire précédent, $V_{\rm sc}$ est bien un objet supercuspidal
au sens de \ref{defsupercusp}.

\ali Le lemme suivant est utilisé dans la preuve du théorème \ref{theocoho}.
\begin{lem} \label{lemmecommut}
  Soit $V$ une $\bZl$-représentation lisse de $G^{0}$. Fixons un ensemble
  $[G/G^{0}J]$ de représentants des $G^{0}J$ classes dans $G$. Alors il existe un
  isomorphisme de $\ZG_{\pi}$-modules
$$   \Hom_{\bZl G}({P_{\pi}},\cind{G^{0}}{G}{V}) \simeq \ZG_{\pi}\otimes_{\ZG_{\pi^{0}}}\left(\bigoplus_{x\in
    [G/G^{0}J]} \Hom_{\bZl
    G}({P_{\pi^{0}}},V^{x})\right).
$$
\end{lem}

\begin{proof}
La réciprocité de Frobenius
  et la formule de Mackey fournissent
  \begin{eqnarray*}
    \Hom_{\bZl G}({P_{\pi}},\cind{G^{0}}{G}{V}) 
    & \simeq & \Hom_{\bZl J^{\circ}}
    \left(P_{\lambda}, \bigoplus_{x\in [G/JG^{0}]} \bigoplus_{j\in
        J}V^{xj}\right) \\
  & \simeq & \Hom_{\bZl J^{\circ}}
    \left(P_{\lambda}, \bigoplus_{x\in [G/JG^{0}]} \bigoplus_{j\in
        J}V^{x}\right) \\
 & \simeq & 
    \Hom_{\bZl J^{\circ}}
    \left(P_{\lambda}, \bigoplus_{x\in G/JG^{0}} {V^{x}}\right)
  \otimes_{\bZl} \bZl[J/J^{\circ}] \\  
    & \simeq & 
    \Hom_{\bZl J^{\circ}}
    \left(P_{\lambda}, \bigoplus_{x\in G/JG^{0}} {V^{x}}\right)
    \otimes_{\ZG_{\pi^{0}}}\ZG_{\pi}
  \end{eqnarray*}
Pour passer à la deuxième ligne, on identifie 
$\Hom_{\bZl J^{\circ}}(P_{\lambda},V^{xj})\simto\Hom_{\bZl
  J^{\circ}}(P_{\lambda},V^{x})$ en envoyant $\alpha$ sur $\alpha\circ
P_{\lambda}^{\dag}(j)$, où $P_{\lambda}^{\dag}$ est le prolongement
choisi dans (\ref{morph}). On constate alors que la décomposition de
la troisième ligne est compatible avec (\ref{morph}).
\end{proof}

\ali Pour traiter le cas de $\CC_{\pi}^{\varpi}$, nous introduisons
quelques notations supplémen\-taires. 
Comme dans les preuves précédentes, 
soit $J$ le normalisateur de $\lambda$ dans $G$ et soit $\lambda'$
l'unique prolongement de $\lambda$ à $J$ contenu dans
$\pi_{|J}$. Ce prolongement est nécessairement trivial sur $\varpi$,
et puisque le groupe $J/\varpi^{\ZM}$ est compact, il
 admet une enveloppe projective $P_{\lambda'}$ dans
$\Mo{\bZl}{J/\varpi^{\ZM}}$.
On renvoie à l'introduction pour la notion de
$\varpi$-relèvement et de $\varpi$-déformation, et on rappelle que $v$
désigne la valuation de $\varpi$, et $\breve\Zl$ le complété de $\Zlnr$.

  \begin{pro}\label{propscindageG}
    Soit ${P_{\pi}^{\varpi}}:=\cind{J}{G}{{P_{\lambda'}}}$.
    \begin{enumerate}
  \item 
    La catégorie $\CC_{\pi}^{\varpi}$ est facteur direct dans
    $\Mo{\bZl}{G/\varpi^{\ZM}}$, pro-engendrée par
    ${P_{\pi}^{\varpi}}$, qui est une enveloppe projective de $\pi$
    dans $\Mo{\bZl}{G/\varpi^{\ZM}}$.
  \item Le commutant $\ZG_{\pi}^{\varpi}:=\endo{\bZl G}{{P_{\pi}^{\varpi}}}$ est isomorphe à
    l'algèbre $\bZl[{\rm Syl}_{\ell}(\FM_{q^{f}}^{\times}\times
    f\ZM/dv\ZM)]$ du $\ell$-Sylow du groupe
    $\FM_{q^{f}}^{\times}\times f\ZM/dv\ZM$.
  \item Soit
    $\Lambda_{\pi}^{\varpi}:=\ZG_{\pi}^{\varpi}\otimes_{\Zlnr}\breve\Zl$. 
La $\Lambda_{\pi}^{\varpi}$-représentation
$\breve{P}_{\pi}^{\varpi}:=P_{\pi}^{\varpi}\otimes_{\Zlnr}\breve\Zl$ 
est une $\varpi$-déformation universelle de $\pi$. En particulier, l'anneau de
$\varpi$-déformations de $\pi$ est isomorphe à  $\breve\Zl[{\rm Syl}_{\ell}(\FM_{q^{f}}^{\times}\times
    f\ZM/dv\ZM)]$.
  \end{enumerate}

  \end{pro}

  \begin{proof}
$i)$ Puisque l'on sait que tous les sous-quotients de 
$P_{\pi}^{\varpi}$ sont isomorphes à $\pi$, \cite[III.5.16]{Vig},
l'argument utilisé pour prouver le $i)(a)$ de la proposition
\ref{propscindageG0} montre que $\CC_{\pi}^{\varpi}$ est facteur direct et
pro-engendrée par $P_{\pi}^{\varpi}$. Le
fait que $P_{\pi}^{\varpi}$ soit une enveloppe projective découlera du
point ii) qui, à travers l'équivalence de catégories $V\mapsto
\Hom_{\bZl G}(P_{\pi}^{\varpi},V)$ de $\CC_{\pi}^{\varpi}$ vers
$\Mod(\ZG_{\pi}^{\varpi})$,  montre que
$P_{\pi}^{\varpi}$ est l'unique objet projectif
indécomposable de $\CC_{\pi}^{\varpi}$.

$ii)$ On sait que l'ensemble d'entrelacement de $\lambda'$, donc aussi celui de
${P_{\lambda'}}$, est contenu dans $J$. Le morphisme canonique
$$ \endo{J}{{P_{\lambda'}}}\To{}\endo{G}{{P_{\pi}^{\varpi}}}$$
est donc un isomorphisme.
Considérons l'induite
$\cind{J^{\circ}}{J/\varpi^{\ZM}}{{P_{\lambda}}}$.
D'une part, elle est isomorphe aux coinvariants
$\cind{J^{\circ}}{J}{{P_{\lambda}}}_{\varpi^{\ZM}}$, donc par le  $ii)(b)$ du
lemme précédent (et sa preuve), on a
$$ \endo{J}{\cind{J^{\circ}}{J/\varpi^{\ZM}}{{P_{\lambda}}}}\simeq
\bZl[{\rm Syl}_{\ell}(\FM_{q^{f}}^{\times})\times J/J^{\circ}\varpi^{\ZM}].  $$
D'autre part elle se décompose en une somme
$$\cind{J^{\circ}}{J/\varpi^{\ZM}}{{P_{\lambda}}} \simeq
\bigoplus_{\lambda^{\dag}} {P_{\lambda^{\dag}}} $$
où $\lambda^{\dag}$ décrit tous les prolongements de $\lambda$ à
$J/\varpi^{\ZM}$ (parmi lesquels figure $\lambda'$).

Or, ces prolongements sont en
bijection avec les $\oFl$-caractères de $J/J^{\circ}\varpi^{\ZM}\simeq
f\ZM/dv\ZM$ \cite[III.4.27.2]{Vig}, et l'algèbre décrite ci-dessus se décompose aussi en un produit
$$\bZl[{\rm Syl}_{\ell}(\FM_{q^{f}}^{\times})\times J/J^{\circ}\varpi^{\ZM}] =
\prod_{\chi}
\bZl[{\rm Syl}_{\ell}(\FM_{q^{f}}^{\times,\ell} )\times
{\rm Syl}_{\ell}(J/J^{\circ}\varpi^{\ZM})]$$
indexé par les $\oFl$-caractères de $J/J^{\circ}\varpi^{\ZM}$.
En identifiant les deux décompositions, on en déduit la propriété annoncée.

$iii)$ Puisque le foncteur $M\mapsto
P_{\pi}^{\varpi} \otimes_{\ZG_{\pi}^{\varpi}}M$ est une équivalence de
catégories entre $\Mod(\ZG_{\pi}^{\varpi})$ et $\CC_{\pi}^{\varpi}$, on
voit que $P_{\pi}^{\varpi}$ est plat sur $\ZG_{\pi}^{\varpi}$, de réduction
$P_{\pi}^{\varpi}\otimes_{\ZG_{\pi}^{\varpi}}\oFl$ isomorphe à l'unique
objet simple de $\CC_{\pi}^{\varpi}$, à savoir $\pi$. Ainsi,
$\breve{P}_{\pi}^{\varpi}$ est bien un $\varpi$-relèvement de $\pi$ sur $\Lambda_{\pi}^{\varpi}$.

Soit maintenant $\wt\pi$ un $\varpi$-relèvement de $\pi$ sur une
$\bZl$-algèbre locale complète noethé\-rienne $\Lambda$ de corps
résiduel $\oFl$. Montrons
d'abord que $\endo{\Lambda G}{\wt\pi}=\Lambda$. 
En effet, si $H$ est un 
pro-$p$-sous-groupe ouvert tel que $\pi^{H}\neq 0$, alors la
sous-$\Lambda$-algèbre de $\endo{\Lambda}{\wt\pi^{H}}$ engendrée par les opérateurs
de Hecke se surjecte sur la $\oFl$-algèbre $\endo{\oFl}{\pi^{H}}$
(par le théorème de densité de Jacobson appliqué à $\pi$), donc  est
égale à $\endo{\Lambda}{\wt\pi^{H}}$ par le lemme de Nakayama.

Considérons alors l'action canonique du centre $\ZG(\CC_{\pi}^{\varpi})$ de la
catégorie $\CC_{\pi}^{\varpi}$ sur
$\wt\pi$. Cette action ``commute au commutant'' donc est
$\Lambda$-linéaire, et fournit un morphisme d'anneaux
$\ZG(\CC_{\pi}^{\varpi})\To{}\endo{\Lambda G}{\wt\pi}=\Lambda$.
Dans le cas $\wt\pi=P_{\pi}^{\varpi}$, le morphisme obtenu
$\ZG(\CC_{\pi}^{\varpi})\To{}\ZG_{\pi}^{\varpi}$ est un isomorphisme. Par composition, on
en déduit un morphisme $\ZG_{\pi}^{\varpi}\To{}\Lambda$, qui se
prolonge au complété $\Lambda_{\pi}^{\varpi}\To{}\Lambda$.

Puisque $P_{\pi}^{\varpi}$ est projectif, on peut choisir un
relèvement $P_{\pi}^{\varpi}\To{}\wt\pi$ de
$P_{\pi}^{\varpi}\To{}\pi$. Ce relèvement est $\ZG(\CC_{\pi}^{\varpi})$-linéaire,
donc se factorise à travers un morphisme $\Lambda$-linéaire
$P_{\pi}^{\varpi}\otimes_{\ZG_{\pi}^{\varpi}}\Lambda \To{}\wt\pi$ qui
se prolonge uniquement en 
$\breve{P}_{\pi}^{\varpi}\otimes_{\Lambda_{\pi}^{\varpi}}\Lambda \To{}\wt\pi$. Par
Nakayama, ce dernier morphisme est surjectif. Par égalité des rangs
(des $H$-invariants pour $H$ pro-$p$-sous-groupe ouvert tel que
$\pi^{H}\neq 0$), il est même bijectif. Cela montre que la
$\varpi$-déformation $\breve{P}_{\pi}^{\varpi}$ est verselle.

Supposons maintenant donnés deux morphismes
$\alpha_{1},\alpha_{2}:\,\Lambda_{\pi}^{\varpi}\To{}\Lambda$ tels que les
déformations
$\breve{P}_{\pi}^{\varpi}\otimes_{\Lambda_{\pi}^{\varpi},\alpha_{1}}\Lambda$
et $\breve{P}_{\pi}^{\varpi}\otimes_{\Lambda_{\pi}^{\varpi},\alpha_{2}}\Lambda$ soient
$\Lambda G$-isomorphes. L'action du centre $\ZG(\CC_{\pi}^{\varpi})$ sur ces deux
objets est alors donnée par un morphisme $\ZG(\CC_{\pi}^{\varpi})\To{}\Lambda$ égal à
la restriction de $\alpha_{1}$ et $\alpha_{2}$ à
$\ZG_{\pi}^{\varpi}\subset \Lambda_{\pi}^{\varpi}$.
 Par densité de $\ZG_{\pi}^{\varpi}$ dans $\Lambda_{\pi}^{\varpi}$, on
 a donc $\alpha_{1}=\alpha_{2}$ et la
$\varpi$-déformation $P_{\pi}^{\varpi}$ est universelle.
  \end{proof}

Remarquons que les déformations d'une représentation cuspidale de
$\GL_{2}(\QM_{p})$ ont été étudiées dans \cite{HelmGL2} par une
approche différente (mais reposant sur les types). 
Comme M. Emerton nous l'a signalé, notre argument reposant sur les
propriétés très particulières de l'enveloppe projective de $\pi$
(notamment la commutativité du commutant) est un cas particulier
d'une théorie plus générale développée par
Paskunas dans \cite[Ch. 3]{Paskunas}.

\subsection{Représentations supercuspidales et scindages de $\Rep(D^{\times})$}
Soit $\rho$ une $\oFl$-représentation irréductible de
$D^{\times}$ dont le caractère central est trivial sur $\varpi\in
K^{\times}\subset D^{\times}$, et soit $\rho^{0}$ un sous-quotient irréductible de
$\rho_{|D^{0}}$, où l'on pose $D^{0}:=\OC_{D}^{\times}$ pour
homogénéiser les notations. Comme dans le paragraphe précédent, on
définit des sous-catégories pleines $\CC_{\rho} \subset
\Mo{\bZl}{D^{\times}}$, $\CC_{\rho}^{\varpi} \subset \Mo{\bZl}{D^{\times}/\varpi^{\ZM}}$
et $\CC_{\rho^{0}} \subset \Mo{\bZl}{D^{0}}$. 

\begin{prop}\label{propscindageD0}
Soit ${P_{\rho^{0}}}$ une
enveloppe projective de $\rho^{0}$ dans
  $\Mo{\bZl}{D^{0}}$.
  \begin{enumerate}\item 
    La sous-catégorie $\CC^{0}_{\rho^{0}}$ est facteur direct dans
    $\Mo{\bZl}{D^{0}}$,  pro-engendrée par 
    ${P_{\rho^{0}}}$.
  \item     La sous-catégorie $\CC_{\rho}$ est facteur direct dans
    $\Mo{\bZl}{D^{\times}}$,  pro-engendrée par l'induite
    ${P_{\rho}}:=\cind{D^{0}}{D^{\times}}{{P_{\rho^{0}}}}$.
    \end{enumerate}
Supposons de plus que $\JL_{\oFl}(\rho)$ est supercuspidale, et notons
$f$ la longueur de $\rho_{|D^{0}}$. Alors
\begin{enumerate} \setcounter{enumi}{2}
 \item Le commutant $\ZG_{\rho^{0}}^{\opp}:=\endo{\bZl D^{0}}{{P_{\rho^{0}}}}$ est isomorphe à $\bZl[{\rm
      Syl}_{\ell}(\FM_{q^{f}}^{\times})]$.
  \item Le commutant $\ZG_{\rho}^{\opp}:=\endo{\bZl D^{\times}}{{P_{\rho}}}$ est isomorphe à
    $\bZl[{\rm Syl}_{\ell}(\FM_{q^{f}}^{\times})\times\ZM]$.
\end{enumerate}

\end{prop}
\begin{proof}
$i)$ Soit $\mG_{D}$ l'idéal maximal de $\OC_{D}$.
Le quotient $\OC_{D}^{\times}/(1+\mG_{D})$ étant cyclique, la
théorie de Clifford nous dit que $\rho^{0}\simeq
\cind{N_{\tau}}{D^{0}}{\tau^{\dag}}$, où  
$\tau^{\dag}$ est un prolongement d'un sous-quotient
  irréductible $\tau$ de $\rho^{0}_{|1+\mG_{D}}$ à son normalisateur
  $N_{\tau}$ dans $D^{0}$. Toujours par cyclicité de $N_{\tau}/(1+\mG_{D})$,
   $N_{\tau}$ possède un plus grand $\ell'$-sous-groupe
   $N_{\tau}^{\ell}$, dont le quotient est un $\ell$-groupe cyclique.
On en déduit que
$$ P_{\rho^{0}}\simeq\cind{N^{\ell}_{\tau}}{D^{0}}{\tau^{\dag}},$$
donc en particulier que tous les sous-quotients irréductibles de
$P_{\rho^{0}}$ sont isomorphes à $\rho^{0}$. Il en va de même de ceux
de ${P_{\rho^{0}}}$. \`A partir de là, on raisonne comme dans la
preuve du point $i)(a)$ de la proposition \ref{propscindageG0}.

$ii)$  Même preuve que le point $ii)(a)$ de la proposition
\ref{propscindageG0}.

La preuve de $iii)$ et $iv)$ nécessite plus de notations. Elle est
donnée à la fin de la preuve de la proposition \ref{propscindageD} ci-dessous.
\end{proof}

\alin{Equivalences de catégories et dualité} 
La proposition précédente  nous donne, grâce à \ref{faitequiv},
une paire d'équivalences ``inverses'' 
$\xymatrix{ \CC_{\rho}  \ar@<.5ex>[r]^-{} &
  \ar@<.5ex>[l]^-{} \Mod(\ZG_{\rho})}$
Nous allons donner une autre forme à ces foncteurs à partir de la
contragrédiente $\rho^{\vee}$ de $\rho$.

Auparavant, il nous faut identifier $\ZG_{\rho^{\vee}}$ et 
$\ZG^{\rm opp}_{\rho}$. Pour cela, remarquons d'abord que 
la contragrédiente $P_{\rho^{0}}^{\vee}$ de
  $P_{\rho^{0}}$ est un objet projectif
  indécomposable de $\Mo{\bZl}{D^{0}}$ dont tous les sous-quotients
  sont isomorphes à $\rho^{0,\vee}$. C'est donc une enveloppe
  projective de $\rho^{0,\vee}$ et nous identifierons
  $P_{\rho^{0}}^{\vee}$ et $P_{\rho^{0, \vee}}$.
Maintenant, comme $P_{\rho^{0}}$ est monogène, on peut trouver un idempotent
$\varepsilon_{\rho}$ de l'algèbre des mesures localement constantes
 $\HC^{0}:=\HC(D^{0},\bZl)$ sur $D^{0}$ tel que
$P_{\rho^{0}}\simeq \HC^{0}\varepsilon_{\rho}$. On fixe un tel
idempotent et un tel isomorphisme. Alors $P_{\rho^{0}}^{\vee}$
s'identifie à $\HC^{0} \check{\varepsilon_{\rho}}$ où $\check{f}$ désigne l'image
de $f$ par l'automorphisme $d\mapsto d^{-1}$. 
 En notant
$\HC:=\HC(D^{\times},\bZl)$ l'algèbre des mesures localement
constantes à support compact sur $D^{\times}$, on a maintenant
 $P_{\rho} = \HC\varepsilon_{\rho}$ et $P_{\rho^{\vee}}=\HC\check{\varepsilon_{\rho}}$. 
En particulier on a $\ZG_{\rho}=\varepsilon_{\rho}\HC\varepsilon_{\rho}$ et
$\ZG_{\rho^{\vee}}=\check\varepsilon_{\rho}\HC\check\varepsilon_{\rho}$
Ainsi, l'application $d\mapsto d^{-1}$ induit un isomorphisme
\begin{equation}
 \ZG_{\rho^{\vee}}\mapsto \ZG_{\rho}^{\rm opp}\label{centredual}
 \end{equation}

\begin{pro}\label{dualite}
Identifions 
$\ZG_{\rho}$ et $\ZG_{\rho^{\vee}}^{\rm opp}$
au moyen de l'isomorphisme  (\ref{centredual}). 
\begin{enumerate}
\item  Il y a un $\ZG_{\rho}$-isomorphisme fonctoriel en $V\in
  \Mo{\bZl}{D^{\times}}$
$$  P_{\rho^{\vee}}\otimes_{\bZl D^{\times}} V \simto 
\Hom_{\bZl D^{\times}}(P_{\rho},V).$$
\item Il y a un $D^{\times}$-isomorphisme
fonctoriel en $M\in \Mod(\ZG_{\rho})$
$$ P_{\rho}\otimes_{\ZG_{\rho}} M\simto \Hom_{\ZG_{\rho}}(P_{\rho^{\vee}},M).$$
\end{enumerate}

\end{pro}
Dans le produit tensoriel du premier isomorphisme, on fait
implicitement de $P_{\rho^{\vee}}$ un $\Zlnr D^{\times}$-module
\emph{à droite} en composant l'action à gauche par $d\mapsto d^{-1}$. 
\begin{proof}
C'est essentiellement de l'``abstract nonsense''.
Pour $V\in\Mo{\bZl}{D^{\times}}$, on a d'une part
$$\Hom_{\bZl D^{\times}}(P_{\rho},V)=\Hom_{\bZl D^{\times}}(\HC\varepsilon_{\rho},V)
\simto \varepsilon_{\rho}V $$
et d'autre part
$$P_{\rho^{\vee}}\otimes_{\bZl D^{\times}} V =(\HC\check\varepsilon_{\rho}\otimes
V)_{D^{\times}} = \varepsilon_{\rho}\HC\otimes_{\HC}V \simto
\varepsilon_{\rho} V.$$
Cela règle le premier isomorphisme. Pour le second, observons que le
foncteur $M\mapsto \Hom_{\ZG_{\rho}}(P_{\rho^{\vee}},M)$ est adjoint à
droite du foncteur $V\mapsto P_{\rho^{\vee}}\otimes_{\bZl
  D^{\times}}V$. 
Or, ce dernier est une équivalence $\CC_{\rho}\simto
\Mod(\ZG_{\rho})$ par le premier isomorphisme, donc $M\mapsto
\Hom_{\ZG_{\rho}}(P_{\rho^{\vee}},M)$ en est un équivalence ``inverse'',
nécessairement isomorphe à l'équivalence ``inverse'' $M\mapsto
M\otimes_{\ZG_{\rho}}P_{\rho}$.
\end{proof}

Passons maintenant à la catégorie $\CC^{\varpi}_{\rho}$. L'entier $f$ désigne la longueur
de $\rho_{|D^{0}}$ et $v$ est toujours la valuation de $\varpi$.

\begin{prop}\label{propscindageD}
Supposons que $\JL_{d,\oFl}(\rho)$
est une représentation supercuspidale, et soit
 ${P_{\rho}^{\varpi}}$ une
enveloppe projective de  $\rho$ dans
 $\Mo{\bZl}{D^{\times}/\varpi^{\ZM}}$.
 \begin{enumerate}
 \item     La sous-catégorie $\CC_{\rho}^{\varpi}$ est facteur direct dans
    $\Mo{\bZl}{D^{\times}/\varpi^{\ZM}}$, pro-engendrée par ${P_{\rho}^{\varpi}}$.
 \item Le commutant $\ZG_{\rho}^{\varpi}:=\endo{\bZl D^{\times}}{{P_{\rho}^{\varpi}}}$ est isomorphe à
    $\bZl[{\rm Syl}_{\ell}(\FM_{q^{f}}^{\times}\times f\ZM/dv\ZM)]$.
  \item Soit $\Lambda_{\rho}^{\varpi}:=
    \ZG_{\rho}^{\varpi}\otimes_{\Zlnr}\breve\Zl$. La
    $\Lambda_{\rho}^{\varpi}$-représentation  $P_{\rho}^{\varpi}$ est
    une $\varpi$-déformation universelle de $\rho$. L'anneau de
$\varpi$-déformations de $\rho$ est isomorphe à  $\breve\Zl[{\rm Syl}_{\ell}(\FM_{q^{f}}^{\times}\times
    f\ZM/dv\ZM)]$.
  \end{enumerate}  
\end{prop}
\begin{proof}
 Nous allons utiliser la théorie des types de Broussous pour
$D^{\times}$ pour décrire l'enveloppe projective $P_{\rho}^{\varpi}$ de manière suffisamment explicite.
L'article \cite{Broussous} n'est écrit que pour les
$\oQl$-représentations, mais s'étend sans problème aux
$\oFl$-représentations comme dans le chapitre III de
\cite{Vig}. Ainsi, 
la représentation $\rho$ est de la forme
$\rho\simeq \cind{J}{D^{\times}}{\lambda}$ pour un ``type simple
étendu'' $(J,\lambda)$ formé d'un sous-groupe ouvert  $J$ de
$D^{\times}$ contenant le centre $K^{\times}$  et
d'une $\oFl$-représentation irréductible $\lambda$ de $J$. 
Nous avons besoin de quelques précisions sur la forme de ces objets.
Notons $J^{0}:=D^{0}\cap J$ le sous-groupe compact maximal de $J$, et
$J^{1}$ son pro-$p$-radical. Il existe une extension $E$ de $K$
contenue dans $D$, dont nous noterons $B$ le commutant (qui est une
algèbre à division de centre $E$)  telle que 
$J=J^{1}B^{\times}$ ($B^{\times}$ normalise $J^{1}$) et $J^{1}\cap
B^{\times}=B^{1}:= 1+\MG_{B}$.
De plus $\lambda$ est de la forme $\lambda\simeq \eta\otimes\tau$ où
$\eta_{|J^{1}}$ est irréductible, et $\tau$ est une représentation de
$J/J^{1}\varpi^{\ZM}$, \emph{i.e.} une représentation
modérément ramifiée de $B^{\times}$. 
Une telle  représentation est facile à décrire, puisque
$B^{\times}/(1+\mG_{B})$, resp. $B^{\times}/(1+\mG_{B})\varpi^{\ZM}$, est isomorphe au groupe
$\FM_{q^{f'd'}}^{\times}\rtimes {\ZM}$, resp. au groupe $\FM_{q^{f'd'}}^{\times}\rtimes {\ZM}/e'd'v\ZM$ 
où 
\begin{itemize}\item 
$f'$ est le degré résiduel
de $E$ sur $K$ et $e'$ est son indice de ramification
\item $(d')^{2}$ est la dimension de $B$ sur
$E$ %. On a donc $d=d'f'e'$.
\item 
  le générateur positif de $\ZM$ agit sur $\FM_{q^{f'd'}}$ par le
  Frobenius relatif à $\FM_{q^{f'}}$.
\end{itemize}
 Par conséquent, il existe un unique diviseur $m$ de $d'$, et un
caractère Frob-régulier $\chi$ de $\FM_{q^{f'm}}^{\times}\times m\ZM/e'd'v\ZM$ tel que 
$$\tau=\cind{\FM_{q^{f'd'}}^{\times}\rtimes
  m\ZM/e'd'v\ZM}{\FM_{q^{f'd'}}^{\times}\rtimes \ZM/e'd'v\ZM}{\chi\circ N_{\FM_{q^{f'd'}}|\FM_{q^{f'm}}}}.$$
L'entier $m$ est le même que celui de \cite[Déf. 10.1.4]{Broussous}. 
Un point crucial maintenant est que, vu notre hypothèse de
supercuspidalité de $\JL_{d,\oFl}(\rho)$, on a $m=d'$. En effet, soit
$\wt\rho$ un relèvement de $\rho$ à $\oQl$ ; la représentation
$\JL_{d,\oQl}(\wt\rho)$ est alors également supercuspidale, et
l'égalité $m=d'$ est une reformulation du point (2) de  \cite[Cor. 1]{BH}.
On a donc finalement
$$\tau=\cind{\FM_{q^{f'd'}}^{\times}\times d'\ZM/e'd'v\ZM}{\FM_{q^{f'd'}}^{\times}\rtimes
  \ZM/e'd'v\ZM}{\chi}$$
pour un caractère Frob-régulier $\chi$ de $\FM_{q^{f'd'}}^{\times}\times
d'\ZM/e'd'v\ZM$.

Nous allons maintenant produire une enveloppe
projective ${P_{\rho}^{\varpi}}$ à partir de ces objets.
De la discussion précédente on déduit que si ${P_{\chi}^{\varpi}}$ est une enveloppe
projective de $\chi$ dans $\Mo{\bZl}{\FM_{q^{f'd'}}^{\times}\times
d'\ZM/e'd'v\ZM}$, alors, son induite 
$${P_{\tau}^{\varpi}}=\cind{\FM_{q^{f'd'}}^{\times}\times d'\ZM/e'd'v\ZM}{\FM_{q^{f'd'}}^{\times}\rtimes
  \ZM/e'd'v\ZM}{{P_{\chi}^{\varpi}}}$$
est une enveloppe projective de $\tau$ dans
$\Mo{\bZl}{J/J^{1}\varpi^{\ZM}}$, dont tous
les sous-quotients irréductibles sont
isomorphes à $\tau$.
Par ailleurs, comme dans le cas des types de Bushnell-Kutzko, \cite[III.4.20]{Vig}, on
peut relever $\eta$ en une $\bZl$-représentation $\wt\eta$.
La $\bZl$-représentation 
$${P_{\lambda}^{\varpi}}:=\wt\eta\otimes{P_{\tau}^{\varpi}}$$ de
$J/\varpi^{\ZM}$ est projective et se surjecte sur $\lambda$. Elle est
aussi indécomposable (car si
$\wt\eta\otimes{P_{\tau}^{\varpi}}=W_{1}\oplus W_{2}$, on obtient
en appliquant le foncteur $\Hom_{J^{1}}(\wt\eta,-)$ une décomposition
de ${P_{\tau}^{\varpi}}$, or ce foncteur n'annule aucun sous-objet
non nul de $\wt\eta\otimes{P_{\tau}^{\varpi}}$ puisque la
restriction de ce dernier à $J^{1}$
est $\wt\eta$-isotypique). Ainsi, ${P_{\lambda}^{\varpi}}$ est une
enveloppe projective de $\lambda$ dans $\Mo{\bZl}{J/\varpi^{\ZM}}$.
Considèrons maintenant l'induite
$${P_{\rho}^{\varpi}}:=\cind{J}{D^{\times}}{{P_{\lambda}^{\varpi}}}.$$
Elle est projective dans $\Mo{\bZl}{D^{\times}/\varpi^{\ZM}}$, et tous
ses sous-quotients irréductibles sont isomorphes à $\rho$. Par
réciprocité de Frobenius, la dimension de
$\Hom_{D^{\times}}({P_{\rho}^{\varpi}},\rho)$ sur $\oFl$ est $1$,
donc ${P_{\rho}^{\varpi}}$ est également indécomposable, et
finalement est une enveloppe projective de $\rho$ dans $\Mo{\bZl}{D^{\times}/\varpi^{\ZM}}$.

Passons maintenant à la preuve de la proposition. Puisque
$P_{\rho}^{\varpi}\in\CC_{\rho}^{\varpi}$, le point $i)$ se prouve comme
le $i)(a)$ de la proposition \ref{propscindageG0}.

 $ii)$
Par construction on a une suite de morphismes de $\bZl$-algèbres
$$ \endo{\FM_{q^{f'd'}}^{\times}\times
    d'\ZM/e'd'v\ZM}{{P_{\chi}^{\varpi}}}\To{}
\endo{J}{{P_{\tau}^{\varpi}}}\To{} \endo{J}{{P_{\lambda}^{\varpi}}}
\To{}\endo{D^{\times}}{{P_{\rho}^{\varpi}}}. 
$$
Le premier est un isomorphisme car le caractère $\chi$ est
Frob-régulier. Le troisième est un isomorphisme car
l'ensemble d'entrelacement de
$(J,\lambda)$ est égal à $J$ \cite[(10.1.3)]{Broussous}. Enfin, le second est un isomorphisme car
$\endo{J^{1}}{\wt\eta}=\bZl$.
Maintenant, la première algèbre est facile à calculer :
$$ \endo{\FM_{q^{f'd'}}^{\times}\times
    d'\ZM/e'd'v\ZM}{{P_{\chi}^{\varpi}}} \simeq
  \bZl[{\rm Syl}_{\ell}(\FM_{q^{f'd'}}^{\times}\times d'\ZM/e'd'v\ZM)].$$
Notons que $f'd'e'=d$, de sorte que $d'\ZM/e'd'v\ZM\simeq
f'd'\ZM/dv\ZM$. Il ne nous reste donc plus qu'à vérifier que $f'd'=f$.
Avec les notations ci-dessus, on a $\rho=\cind{J^{0}E^{\times}}{D^{\times}}{\eta\otimes\chi}$.
Par la formule de Mackey on a donc
$$ \rho_{|D^{0}}\simeq \bigoplus_{x\in J^{0}E^{\times}\ba D^{\times}/D^{0}}
\left(\cind{J^{0}}{D^{0}}{\eta\otimes\chi}\right)^{x}. $$
Chaque induite est irréductible puisque l'entrelacement de
$\eta\otimes\chi$ dans $D^{0}$ est $J^{0}$. Donc 
$$f = [D^{\times}:D^{0}E^{\times}]= d
[D^{0}E^{\times}:D^{0}K^{\times}]^{-1}= d
[E^{\times}:K^{\times}]^{-1}= d{e'}^{-1}=d'f'.$$

$iii)$ La preuve est la même que celle du $iii)$ de la proposition \ref{propscindageG}.

\medskip

Nous pouvons maintenant prouver les points $iii)$ et $iv)$ de la
proposition \ref{propscindageD0}. Comme ci-dessus, on utilise le type
pour produire une enveloppe projective $P_{\rho^{0}}$ sous la forme :
$$ P_{\rho^{0}}:=\cind{J^{0}}{D^{0}}{\wt\eta\otimes P_{\chi}^{0}} $$
où $P_{\chi}^{0}$ est une enveloppe projective de $\chi$ dans 
$\Mo{\bZl}{\FM_{q^{f'd'}}^{\times}}$. L'homomorphisme canonique
$$ \endo{\bZl \FM_{q^{f'd'}}^{\times}}{P_{\chi}^{0}}\To{} \endo{\bZl
    D^{0}}{P_{\rho^{0}}} $$
est un isomorphisme, et on identifie facilement le terme de gauche à 
$\bZl[{\rm Syl}_{\ell}(\FM_{q^{f'd'}}^{\times})]$.
Par ailleurs, posons
$P_{\chi}:=\cind{\FM_{q^{f'd'}}^{\times}}{\FM_{q^{f'd'}}^{\times}\times
d'\ZM}{P_{\chi}^{0}}$ et 
$P_{\tau}:=\cind{\FM_{q^{f'd'}}^{\times}\times d'\ZM}{\FM_{q^{f'd'}}^{\times}\rtimes
\ZM}{P_{\chi}}$, de sorte que 
$P_{\rho}=\cind{J}{D^{\times}}{\wt\eta\otimes P_{\tau}}$. Les deux homomorphismes
canoniques
$$ \endo{\bZl( \FM_{q^{f'd'}}^{\times}\times
d'\ZM)}{P_{\chi}}\To{} 
\endo{\bZl(\FM_{q^{f'd'}}^{\times}\rtimes
\ZM)}{P_{\tau}} \To{} \endo{\bZl D^{\times}}{P_{\rho}} $$
sont des isomorphismes, et on identifie facilement le terme de gauche
à $\bZl[{\rm Syl}_{\ell}(\FM_{q^{f'd'}}^{\times})\times\ZM]$ de la
même manière que dans la preuve du $ii)(b)$ de la proposition \ref{propscindageG0}.
\end{proof}

\subsection{$\varphi$-déformations de représentations irréductibles de
  $W_{K}$}
Fixons une $\oFl$-représentation irréductible $\sigma$ de $W_{K}$ de
dimension $d$. 
D'après \cite[2.6]{VigAENS}, on peut écrire $\sigma$ sous la forme
$$ \sigma=\cind{W_{L}}{W_{K}}{\tau}, \,\, \hbox{où}$$
\begin{itemize}\item 
  $L$ est une extension modérément ramifiée de $K$, de degré résiduel
  noté $d'e'$ et d'indice de ramification noté $f'$.
\item La restriction $\nu:=\tau_{|P_{K}}$ de $\tau$ au sous-groupe d'inertie
  sauvage de $W_{K}$ est irréductible, et son normalisateur ${\rm
    Norm}_{W_{K}}(\nu)$ est le groupe de Weil $W_{E}$ d'une sous-extension
  $E\subset L$ sur laquelle $L$ est non
  ramifiée de degré $d'$.
\end{itemize}

La forme de $\sigma$ montre que $\tau$ y apparait avec multiplicité
$1$. En fait on a mieux ;  soit $I_{L}^{\ell'}$ le plus grand sous-groupe fermé d'ordre
  premier à $\ell$ du sous-groupe d'inertie $I_{L}$ de
  $W_{L}$. C'est un sous-groupe distingué de $W_{E}$.
et  $W_{L}/I_{L}^{\ell'}\simeq\ZM_{\ell}\rtimes \ZM$.

\begin{lemme} Pour toute $\oFl$-représentation irréductible $\sigma'$ de
  $W_{K}$, on a
$$ \cas{\dim_{\oFl}(\Hom_{\oFl I_{L}^{\ell'}}(\tau,\sigma'))}{1}
%{\sigma'\simeq    \sigma\psi \hbox{ pour un }  \psi:W_{K}/I_{K}\To{}\oFl^{\times}}
{\sigma'\sim\sigma}
{0}{\sigma' \nsim \sigma}
$$ 
où $\sigma\sim\sigma'$ signifie que $\sigma'$ est un twist non ramifié
de $\sigma$.
\end{lemme}
\begin{proof}
Si $\sigma'\sim \sigma$, on a
$\dim_{\oFl}(\Hom_{\oFl I_{L}^{\ell'}}(\tau,\sigma'))=
\dim_{\oFl}(\Hom_{\oFl I_{L}^{\ell'}}(\tau,\sigma))\geq 1.$
Pour montrer l'égalité, on écrit $\tau$ sous la forme $\rho_{|W_{L}}\otimes
\chi$ où $\rho$ est un prolongement de $\nu$ à $W_{E}$ et $\chi$ est
un caractère de $W_{L}$, dont le normalisateur dans $W_{E}$ est nécessairement
$W_{L}$.  Il faut alors voir que
le normalisateur de $\chi_{|I_{L}^{\ell'}}$ dans $W_{E}$ est encore
$W_{L}$. Via le corps de classes, on a une décomposition de l'abélianisé
$$ W_{L}^{\rm ab}\simeq {L}^{\times}=
\OC_{L}^{\times,\ell'}\times\OC_{L}^{\times,\ell}\times
\varpi_{E}^{\ZM}
$$
où $I_{L}^{\ell',\rm ab}\simeq\OC_{L}^{\times,\ell'}$, 
où $\OC_{L}^{\times,\ell}$ est le $\ell$-Sylow (cyclique) de
$\OC_{L}^{\times}$, et $\varpi_{E}$ est une uniformisante de $E$ qui
en est une aussi de $L$ puisque $L$ est non ramifiée sur $E$.
L'action de $W_{E}$ est triviale sur $\varpi_{E}^{\ZM}$ et le
caractère $\chi$ est trivial sur $\OC_{L}^{\times,\ell}$. Le
normalisateur de $\chi_{|I_{L}^{\ell'}}$ est donc le même que celui
de $\chi$, à savoir $W_{L}$.

Soit maintenant $\sigma'$ irréductible telle que $\Hom_{\oFl
  I_{L}^{\ell'}}(\tau,\sigma')\neq 0$. Il existe alors 
un caractère $\psi$ de $W_{L}/I_{L}^{\ell'}$ tel que $\Hom_{\oFl W_{L}}(\tau\psi,\sigma')\neq 0$.
Comme $\psi$ est aussi trivial sur le $\ell$-Sylow de $W_{L}^{\rm
  ab}$,  il est non ramifié.  Il est donc
 restriction d'un caractère non ramifié encore noté $\psi$ de
$W_{K}$. Ainsi, $\Hom_{\oFl
  W_{L}}(\tau,\sigma'\psi^{-1})\neq 0$, ce qui montre que
$\sigma'\psi^{-1}\simeq \sigma$.
\end{proof}

\ali On note $\CC_{\sigma}=\CC_{\sigma}(K)$ la sous-catégorie
pleine de $\Moc{\bZl}{W_{K}}$ formée des objets dont tous les
$\bZl I_{K}$-sous-quotients irréductibles sont isomorphes à un
sous-quotient de
$\sigma_{|I_{K}}$. De même on dispose de la sous-catégorie $\CC_{1}(L)$ (pour
la représentation triviale de $W_{L}$) de
$\Moc{\bZl}{W_{L}}$. 
On choisit maintenant un relèvement $\wt\tau$ de $\tau$ dans $\Mo{\bZl}{W_{L}}$. 
Comme le pro-ordre de $I_{L}^{\ell'}$ est premier à $\ell$, la
restriction $\wt\tau_{|I_{L}^{\ell'}}$ est uniquement déterminée.

\begin{pro} \label{propscindageW} Soit
  $P_{\sigma}:=\cind{I_{L}}{W_{K}}{\wt\tau_{|I_{L}^{\ell'}}\otimes_{\bZl[[I_{L}^{\ell'}]]} \bZl[[I_{L}]]}$.
  \begin{enumerate}
  \item La sous-catégorie $\CC_{\sigma}(K)$ de $\Moc{\bZl}{W_{K}}$ est
    facteur direct, pro-engendrée par $P_{\sigma}$. Le commutant
    $\ZG_{\sigma}^{\opp}:=\endo{\bZl W_{K}}{P_{\sigma}}$ est isomorphe au produit croisé
    $\bZl[[I_{L}/I_{L}^{\ell'}]]\rtimes \varphi_{L}^{\ZM}$.
  \item Le foncteur
$$\application{I_{\wt\tau}:\,}{\CC_{1}(L)}{\CC_{\sigma}(K)}
{M}{\cind{W_{L}}{W_{K}}{\wt\tau\otimes_{\bZl} M}}.$$
est une équivalence de
    catégorie, dont une équivalence inverse est donnée par le foncteur
$$\application{R_{\wt\tau}:\,}{\CC_{\sigma}(K)}{\CC_{1}(L)}{(V,\sigma_{V})}{\Hom_{\bZl
  I_{L}^{\ell'}}(\wt\tau,V)}$$
où $W_{L}$ agit sur le terme de gauche par composition
$w\cdot\alpha:= \sigma_{V}(w)\circ\alpha\circ \wt\tau(w)^{-1}$.
  \end{enumerate}
\end{pro}
\begin{proof}
  $i)$ Posons $\sigma^{0}:=\cind{I_{L}}{I_{K}}{\tau}$. C'est un
  sous-quotient irréductible de $\sigma_{|I_{K}}$ et tous les autres
  lui sont conjugués. Le lemme précédent nous dit que 
$\sigma^{0}$ est la seule 
  $\oFl$-représentation  irréductible de $I_{K}$ dont la restriction à
  $I_{L}^{\ell'}$ contient $\tau_{|I_{L}^{\ell'}}$. Il s'ensuit
  immédiatement que  la sous-catégorie $\CC^{0}_{\sigma^{0}}$
  de $\Moc{\Zlnr}{I_{K}}$ formée des objets dont les sous-quotients
  irréductibles sont isomorphes à $\sigma^{0}$ est stable par
  sous-objet (et quotients et sommes directes), et
  pro-engendrée par l'induite
$P_{\sigma^{0}}:=\cind{I_{L}}{I_{K}}{\wt\tau_{|I_{L}^{\ell'}}\otimes_{\bZl[[I_{L}^{\ell'}]]}
  \bZl[[I_{L}]]}$. Ceci s'appliquant à toute représentation
irréductible de $I_{K}$, on en déduit aussi qu'elle est facteur direct.
La première assertion du point $i)$ en découle, comme dans la preuve
de \ref{propscindageG0} $ii)(a)$.
 La deuxième assertion peut se voir par un calcul direct,
  ou comme conséquence du $ii)$, \emph{cf} plus bas.

$ii)$ Les deux foncteurs sont visiblement exacts. On a une
 transformation naturelle
$$\alpha_{V}:\,\,\cind{W_{L}}{W_{K}}{\wt\tau\otimes\Hom_{I_{L}^{\ell'}}(\wt\tau,V)}\To{}
V $$
donnée par évaluation des morphismes et réciprocité de Frobenius.
En utilisant le fait que  $\Hom_{I_{L}^{\ell'}}(\tau,\tau^{w})$ est
nul si $w$ n'est pas dans $W_{L}$, on constate que
$R_{\wt\tau}(\alpha_{V})$ est inversible. Comme ce foncteur est
pleinement fidèle par $i)$, $\alpha_{V}$ est aussi inversible.
Dans l'autre sens, on a la transformation naturelle
$$\beta_{M}:\,\, M\To{} \Hom_{I_{L}^{\ell'}}\left(\wt\tau,\cind{W_{L}}{W_{K}}{\wt\tau\otimes M}\right)$$
qui envoie $m$ sur $\id_{\wt\tau}\otimes m$. 
Toujours grâce au fait que  $\Hom_{I_{L}^{\ell'}}(\tau,\tau^{w})$ est
nul si $w$ n'est pas dans $W_{L}$, on constate que $\beta_{M}$ est un isomorphisme.
On notera que
$P_{\sigma}=I_{\wt\tau}(P_{1})$, donc le
commutant de $P_{\sigma}$ est isomorphe à celui de $P_{1}$ qui est
bien comme annoncé au $i)$.
\end{proof}

\alin{La catégorie $\CC_{\sigma}^{\rm ab}$} \label{csigmaab}
Soit $\ZG_{\sigma}^{\rm ab}$ le plus grand quotient
commutatif de $\ZG_{\sigma}$. La catégorie $\Mod(\ZG_{\sigma}^{\rm
  ab})$ est naturellement une sous-catégorie
pleine de $\Mod(\ZG_{\sigma})$, stable par sous-quotients et
sommes directes. 

\begin{defn}
  On note $\CC_{\sigma}^{\rm ab}$ l'image essentielle du foncteur
  $M\in \Mod(\ZG_{\sigma}^{\rm ab}) \mapsto
  P_{\sigma}\otimes_{\ZG_{\sigma}}M\in\CC_{\sigma}$, et $P_{\sigma}^{\rm ab}:=
  P_{\sigma}\otimes_{\ZG_{\sigma}}\ZG_{\sigma}^{\rm ab}$.
\end{defn}

Ainsi $\CC_{\sigma}^{\rm ab}$ est une sous-catégorie pleine de
$\CC_{\sigma}$, stable par sous-quotients et sommes directes. En tant
que catégorie abélienne, elle est pro-engendrée par $P_{\sigma}^{\rm
  ab}$, mais on prendra garde au fait que celui-ci n'est pas un objet
projectif dans $\CC_{\sigma}$. Par construction, on a un isomorphisme canonique
$\ZG_{\sigma}^{\rm ab} \simto \endo{\bZl W_{K}}{P_{\sigma}^{\rm ab}}$.

Concrètement, le foncteur $R_{\wt\tau}$ induit une équivalence
$\CC_{\sigma}^{\rm ab}(K)\simto \CC_{1}^{\rm
  ab}(L)$, ce qui via le corps de classes montre que
$$\ZG_{\sigma}^{\rm ab} \simeq \bZl[{\rm Syl}_{\ell}(k_{L}^{\times})\times \ZM].$$
D'autre part, 
$$P_{\sigma}^{\rm ab} \simeq \cind{I_{L}^{\ell\rm
    -ab}}{W_{K}}{\wt\tau_{|I_{L}^{\ell-\rm ab}}}$$
 où $I_{L}^{\ell\rm -ab}$ désigne le noyau de la composée $I_{L}\To{\rm
  Art_{L}}\OC_{L}^{\times}\twoheadrightarrow {\rm
  Syl}_{\ell}(k_{L}^{\times})$.

\alin{Dualité et équivalences}
Ce paragraphe est l'analogue du paragraphe \ref{dualite}.
Soit $\HC=\HC(W_{K},\bZl)$ l'algèbre des mesures continues à support
compact sur $W_{K}$ et localement constantes pour l'action de $I_{K}^{\ell'}$.
La représentation $\wt\tau_{|I_{L}^{\ell'}}$ fournit un idempotent
abusivement noté $\varepsilon_{\sigma}$ qui permet d'identifier  $P_{\sigma}$ à
$\HC \varepsilon_{\sigma}$ et fournit un progénérateur
$P_{\sigma^{\vee}}:= \HC \check{\varepsilon}_{\sigma}$ de
$\CC_{\sigma^{\vee}}$.
L'inversion $w\mapsto w^{-1}$ permet alors d'identifier $\ZG_{\sigma}=e_{\sigma}\HC e_{\sigma}$
à $\ZG_{\sigma^{\vee}}$.

\begin{pro}\label{dualiteW}
 Avec cette identification,
\begin{enumerate}
\item  Il y a un $\ZG_{\sigma}$-isomorphisme fonctoriel en $V\in\Moc{\bZl}{W_{K}}$  
 $$ P_{\sigma}\otimes_{\bZl W_{K}} V \simto 
\Hom_{\bZl W_{K}}(P_{\sigma^{\vee}},V).$$
\item Il y a un $W_{K}$-isomorphisme
fonctoriel en $M\in \Mod(\ZG_{\sigma})$
$$ P_{\sigma}\otimes_{\ZG_{\sigma}} M\simto \Hom_{\ZG_{\sigma}}(P_{\sigma^{\vee}},M).$$
\end{enumerate}
\end{pro}
\begin{proof}
 Cela se prouve comme la proposition \ref{dualite}.
\end{proof}

\ali Notons  $\RC\rm el(\sigma)$ la catégorie cofibrée en groupoïdes
au-dessus de la catégorie des $\bZl$-algèbres locales complètes
noethériennes, dont les objets sont les paires $(\Lambda,\wt\sigma)$
avec $\wt\sigma$ un relèvement de $\sigma$ sur $\Lambda$.   
De même on a la catégorie $\RC\rm el(1_{W_{L}})$ des relèvements de la représentation
triviale de $W_{L}$. La proposition \ref{propscindageW} nous assure que
\emph{les foncteurs $I_{\wt\tau}$ et $R_{\wt\tau}$ induisent
  des équivalences (fibrées) inverses entre $\RC\rm el(\sigma)$ et $\RC\rm el(1_{W_{L}})$.}

Fixons maintenant un élément $\varpi$ de valuation $v>0$ dans $K$
et un élément $\varphi\in W_{K}$ d'image $\varpi$ via l'homomorphisme
$W_{K}\To{} K^{\times}$ du corps de classes,
puis supposons que le déterminant de $\sigma$ sur $\varphi$ soit égal
à $1$.
 Si $\varepsilon$ désigne la signature de l'action
de $\varphi$ sur $W_{K}/W_{L}$, et si $t:\, W_{K}^{\rm ab}\To{}
W_{L}^{\rm ab}$ désigne le transfert, on a la formule \cite{Gallagher} 
$$ \det(\sigma(\varphi)) = \varepsilon(\varphi) \det(\tau(
t(\varphi))) =1. $$
Quittes à ajuster notre choix de relèvement $\wt\tau$ par un caractère
non ramifié de $W_{L}$, \emph{nous supposerons que }
$$    \varepsilon(\varphi) \det(\wt\tau(
t(\varphi))) =1. $$

On s'intéresse 
à la catégorie  $\varphi$-$\RC\rm el(\sigma)$ des
$\varphi$-relèvements de $\sigma$. De l'autre côté, on note
$t(\varphi)^{\delta}$-$\RC\rm el(1_{W_{L}})$ la catégorie des relèvements du caractère
trivial de $W_{L}$ qui valent $1$ sur l'élément $t(\varphi)^{\delta}\in
W_{L}^{\rm ab}$. On rappelle que $\delta$ est la dimension de $\tau$.

\begin{coro}\label{coroW}
les foncteurs $I_{\wt\tau}$ et $R_{\wt\tau}$ induisent
  des équivalences (fibrées) inverses entre $\varphi$-$\RC\rm
  el(\sigma)$ et $t(\varphi)^{\delta}$-$\RC\rm el(1_{W_{L}})$. En particulier,
  \begin{enumerate}
  \item L'anneau $\Lambda_{\sigma}$ de $\varphi$-déformation de $\sigma$ est isomorphe à
    $\bZl[{\rm Syl}_{\ell}(\FM_{q^{f}}^{\times}\times f\ZM/dv\ZM)]$, où $f$
    désigne la longueur de $\sigma_{|I_{K}}$.
  \item Soit $\wt\sigma$ un $\varphi$-relèvement de $\sigma$ sur
    $\Lambda_{\sigma}$ tel que 
$$\wt\sigma\otimes\oQl \simeq
\bigoplus_{r_{\ell}\sigma^{\dag}=\sigma,\rm det\sigma^{\dag}(\varphi)=1}
\sigma^{\dag}.$$
Alors $\wt\sigma$ est la $\varphi$-déformation universelle de $\sigma$.
  \end{enumerate}
 \end{coro}
 \begin{proof}
   On sait que ces foncteurs induisent des équivalences entre
   catégories de défor\-mations sans conditions de
   déterminant. Il suffit donc de vérifier que ces conditions se
   correspondent. Ecrivons
$$\wt\sigma=I_{\wt\tau}(\wt{1_{W_{L}}})=
  \cind{W_{L}}{W_{K}}{\wt\tau\otimes\wt{1_{W_{L}}}}.$$ La
  formule
$$ \det(\wt\sigma(\varphi)) = \varepsilon(\varphi) \det(\wt\tau(
t(\varphi))) \wt{1_{W_{L}}}(t(\varphi))^{\delta}= \wt{1_{W_{L}}}(t(\varphi))^{\delta}$$
(vu notre choix de relèvement $\wt\tau$) montre que les conditions se
correspondent bien. 

$i)$ Il suffit de calculer la $t(\varphi)$-déformation universelle de
$1_{W_{L}}$, c'est-à-dire la déformation universelle du caractère
trivial du groupe $W_{L}^{\rm ab}/t(\varphi)^{\delta\ZM}$. Via le corps de
classe, ce groupe est isomorphe à $L^{\times}/\varpi^{\delta\ZM}$, puisque
le transfert s'identifie à l'inclusion de $K^{\times}$ dans
$L^{\times}$. Ce dernier groupe se décompose en
$$ L^{\times}/\varpi^{\delta\ZM}\simeq (1+\mG_{L})\times
k_{L}^{\times}\times \ZM/e'\delta v\ZM.$$
La catégorie des $t(\varphi)^{\delta}$-relèvements de $1_{W_{L}}$ est
donc équivalente à la catégorie des relève\-ments du caractère trivial
du $\ell$-Sylow ${\rm Syl}_{\ell}(k_{L}^{\times}\times
\ZM/e'\delta v\ZM)$. En particulier l'anneau de $\varpi$-déformation de
$\sigma$ est isomorphe à la $\bZl$-algèbre de ce $\ell$-Sylow.
Maintenant, vu nos notations, on a $k_{L}\simeq \FM_{q^{f'd'}}$,
et $\ZM/e'\delta v\ZM\simeq f'd'\ZM/dv\ZM$. Il suffit donc de vérifier que
$f=f'd'$. %, ce qui découle de \cite[2.6.4]{VigAENS}.
La formule de Mackey donne
$$ \sigma_{|I_{K}}\simeq \bigoplus_{w\in W_{K}/I_{K}W_{L}}
\cind{I_{L}}{I_{K}}{\tau_{|I_{L}}}^{w} .$$
On sait que $\sigma_{|I_{K}}$ est semi-simple, donc l'induite
$\cind{I_{L}}{I_{K}}{\tau}$ l'est aussi, mais puisque l'entrelacement
de $\tau_{|I_{L}}$ est égal à $I_{L}$, cette induite est
indécomposable, donc finalement irréductible. La longueur $f$ est donc
égale à $[W_{K}:W_{L}I_{K}]=f'd'$.

$ii)$ En vertu du $i)$ et de sa preuve ci-dessus, il suffit de montrer
que si $\wt{1}$ est une déformation du caractère trivial du groupe 
$S_{\ell}:={\rm Syl}_{\ell}(\FM_{q^{f}}^{\times}\times f\ZM/dv\ZM)$ sur
son algèbre de fonctions $\Lambda:=\bZl[S_{\ell}]$ telle que
$\wt{1}\otimes\oQl\simeq\bigoplus_{\chi:\,S_{\ell}\To{}\oQl} \chi$,
alors elle est universelle.

Partons de la déformation universelle, à savoir la représentation
régulière $\wt{1}_{\rm un}$ de $S_{\ell}$ vue comme déformation du caractère trivial sur
l'algèbre $\Lambda_{\rm un}=\bZl[S_{\ell}]$. Par universalité, la
déformation $\wt{1}$ est obtenue en poussant $\wt{1}_{\rm un}$ par un
morphisme $\Lambda_{\rm un}\To{\alpha}\Lambda$. Par hypothèse,
$\alpha$ induit une bijection entre $\oQl$-caractères de
$\Lambda$ et de $\Lambda_{\rm un}$. En particulier $\alpha$ est injectif,
car $\Lambda_{\rm un}$ est réduit et sans $\ell$-torsion. Il induit
donc une injection $\mu_{\ell^{\infty}}(\Lambda_{\rm
  un})\To{}\mu_{\ell^{\infty}}(\Lambda)$, laquelle est une bijection
puisque ces deux groupes sont finis de même cardinal. Or, $\Lambda$
est engendré par $\mu_{\ell^{\infty}}(\Lambda)$ sur $\bZl$, donc
$\alpha$ est aussi surjectif.
 \end{proof}

\bigskip

\noindent\textsc{Institut de Mathématiques de Jussieu, Université
  Pierre et Marie Curie. \\
%  et Institut Universitaire de France 
 4, place Jussieu,  75252 Paris cedex 05, France}

\noindent\texttt{dat@math.jussieu.fr}

%\bibliography{artbiblio}

\end{document}